\documentclass[10pt]{amsart}
\usepackage{latexsym}
\usepackage{amsmath}
\usepackage{amssymb}
\usepackage{mathrsfs}
\usepackage{graphicx}
\usepackage{color}
\usepackage{pgfpages}
\usepackage{ifthen}
\usepackage{leftidx,tensor}
\usepackage[T1]{fontenc}
\usepackage[latin1]{inputenc}
\usepackage{mathtools}
\usepackage{comment}
\usepackage{dsfont}
\usepackage[nocompress]{cite}

\usepackage[shortlabels]{enumitem}
\usepackage{aliascnt}
\usepackage[bookmarks=true,pdfstartview=FitH, pdfborder={0 0 0}, colorlinks=true,citecolor=red, linkcolor=blue]{hyperref}
\usepackage{bbm}
\usepackage{nicefrac}

\theoremstyle{plain}
\newtheorem{thm}{Theorem}[section]
\newaliascnt{cor}{thm}
\newaliascnt{prop}{thm}
\newaliascnt{lem}{thm}
\newtheorem{cor}[cor]{Corollary}
\newtheorem{prop}[prop]{Proposition}
\newtheorem{lem}[lem]{Lemma}
\aliascntresetthe{cor}
\aliascntresetthe{prop}
\aliascntresetthe{lem} 
\theoremstyle{definition}
\newaliascnt{defn}{thm}
\newaliascnt{asu}{thm}
\newaliascnt{con}{thm}
\aliascntresetthe{defn}
\aliascntresetthe{asu}
\aliascntresetthe{con}
\newcounter{stp}
\newcounter{stpi}
\newcounter{stpci}
\newcounter{stpiii}

\theoremstyle{remark}
\newaliascnt{rem}{thm}
\newaliascnt{exa}{thm}
\newaliascnt{masu}{thm}
\newaliascnt{nota}{thm}
\newaliascnt{sett}{thm}
\newtheorem{rem}[rem]{Remark}

\aliascntresetthe{rem}
\aliascntresetthe{exa}
\aliascntresetthe{masu}
\aliascntresetthe{nota}
\aliascntresetthe{sett}
\numberwithin{equation}{section}

\setlist[enumerate]{font = \normalfont}

\newcommand{\E}{\mathbb{E}}
\newcommand{\R}{\mathbb{R}}
\newcommand{\C}{\mathbb{C}}

\newcommand{\N}{\mathbb{N}}
\newcommand{\bP}{\mathbb{P}}

\newcommand{\D}{\mathbb{D}}

\renewcommand{\th}{\Tilde{h}}
\newcommand{\tZ}{\Tilde{Z}}

\newcommand{\rW}{\mathrm{W}} 
\newcommand{\rL}{\mathrm{L}}

\newcommand{\rH}{\mathrm{H}}

\newcommand{\rC}{\mathrm{C}}
\newcommand{\rB}{\mathrm{B}}

\newcommand{\rD}{\mathrm{D}}

\newcommand{\rX}{\mathrm{X}}
\newcommand{\rY}{\mathrm{Y}}

\newcommand{\rc}{\mathrm{c}}

\newcommand{\rd}{\mathrm{d}}
\newcommand{\srd}{\,\mathrm{d}}

\newcommand{\fs}{\mathrm{fs}}

\newcommand{\fl}{\mathrm{fl}}
\newcommand{\pr}{\mathrm{pr}}

\newcommand{\mre}{\mathrm{e}}

\newcommand{\ess}{\mathrm{ess}}

\newcommand{\rf}{\mathrm{f}}
\newcommand{\rs}{\mathrm{s}}

\newcommand{\rp}{\mathrm{p}}
\newcommand{\per}{\mathrm{per}}
\renewcommand{\rm}{\mathrm{m}}

\newcommand{\cA}{\mathcal{A}}
\newcommand{\cS}{\mathcal{S}}
\newcommand{\cF}{\mathcal{F}}

\newcommand{\cB}{\mathcal{B}}
\newcommand{\cL}{\mathcal{L}}

\newcommand{\cG}{\mathcal{G}}

\newcommand{\cT}{\mathcal{T}}
\newcommand{\cR}{\mathcal{R}}

\newcommand{\cH}{\mathcal{H}}

\newcommand{\cW}{\mathcal{W}}
\newcommand{\Hinfty}{\cH^\infty}
\newcommand{\cBIP}{\mathcal{BIP}}

\newcommand{\frF}{\mathfrak{F}}

\newcommand{\tv}{\Tilde{v}}

\newcommand{\tA}{\Tilde{A}}

\newcommand{\tpi}{\Tilde{\pi}}
\newcommand{\tu}{\Tilde{u}}
\newcommand{\teta}{\Tilde{\eta}}

\newcommand{\bfone}{\boldsymbol{1}}
\newcommand{\ocF}{\overline{\cF}}
\newcommand{\eps}{\varepsilon}
\newcommand{\del}{\partial}
\DeclareMathOperator{\tr}{tr}
\DeclareMathOperator{\diag}{diag}
\DeclareMathOperator{\Id}{Id}
\DeclareMathOperator{\mdiv}{div}
\DeclareMathOperator{\Rep}{Re}

\newcommand{\tin}{\enspace \text{in} \enspace}
\newcommand{\ton}{\enspace \text{on} \enspace}
\newcommand{\tfor}{\enspace \text{for} \enspace}
\newcommand{\tforall}{\enspace \text{for all} \enspace}
\newcommand{\tand}{\enspace \text{and} \enspace}

\newcommand{\tif}{\enspace \text{if} \enspace}
\newcommand{\twhere}{\enspace \text{where} \enspace}
\newcommand{\twith}{\enspace \text{with} \enspace}

\begin{document}

\title[Stochastically forced Navier-Stokes equations interacting with an elastic structure]{Stochastically forced Navier-Stokes equations interacting with an elastic structure}

\author{Felix Brandt}
\address{Department of Mathematics, University of California at Berkeley, Berkeley, 94720, CA, USA.}
\email{fbrandt@berkeley.edu}
\author{Matthias Hieber}
\address{Technische Universit\"{a}t Darmstadt,
Schlo\ss{}gartenstra{\ss}e 7, 64289 Darmstadt, Germany.}
\email{hieber@mathematik.tu-darmstadt.de}
\author{Arnab Roy}
\address{Basque Center for Applied Mathematics (BCAM), Alameda de Mazarredo 14, 48009 Bilbao, Spain.}
\address{IKERBASQUE, Basque Foundation for Science, Plaza Euskadi 5, 48009 Bilbao, Bizkaia, Spain.}
\email{aroy@bcamath.org}

\begin{abstract}
We prove global-in-time strong pathwise well-posedness for a stochastic fluid-structure interaction problem coupling a two-dimensional incompressible Navier-Stokes fluid to a one-dimensional damped Kirchhoff plate. 
The coupling is imposed on a fixed interface through continuity of velocities and balance of normal stresses, and stochastic forcing, modeled by a cylindrical Wiener process, acts on both the fluid and structure equations.
We split the problem into a linear stochastic part and a nonlinear deterministic remainder. 
The linear stochastic problem is treated by proving that the associated fluid-structure operator admits a bounded \(\mathcal{H}^\infty\)-calculus, yielding stochastic maximal regularity. 
This requires a decoupling procedure for the non-diagonal operator domain, and pressure estimates via suitable lifting constructions. 
The deterministic remainder is solved locally by quasilinear methods, and the resulting blow-up criterion is ruled out by higher-order a priori estimates.

This is the first global-in-time strong pathwise well-posedness result for a stochastically forced Navier-Stokes system interacting with a deformable elastic structure.
\end{abstract}

\maketitle

\section{Introduction}\label{sec:intro}

This paper establishes a strong pathwise well-posedness theory for a stochastically forced Navier-Stokes system interacting with an elastic structure located at a part of the fluid boundary.
The main result is global-in-time existence and uniqueness for a two-dimensional incompressible Navier-Stokes fluid coupled to a one-dimensional damped Kirchhoff plate subject to additive noise.
A central ingredient, and an independent contribution of the paper, is the investigation of the underlying fluid-structure operator.
The main analytical difficulties are the non-diagonal operator domain, and the non-local fluid pressure in the structure equation.
In order to decouple the non-diagonal domain of this operator, we invoke a Stokes lifting to obtain its bounded $\Hinfty$-calculus, which is the basis for treating the stochastic convolution in UMD spaces.
In particular, the stochastic forcing is considered in the space of $\gamma$-radonifying operators.
For handling the pressure, we employ Neumann lifting arguments.
For the global-in-time existence, we establish a blow-up criterion.
The blow-up scenario is ruled out by a priori estimates that differ significantly from the classical Navier-Stokes case:
First, the contributions from the stochastic forcing have to be taken into account.
Second, the coupling with the elastic structure requires new elliptic estimates for the associated stationary Stokes problem.

The deterministic Navier-Stokes equations, which form the basic fluid model in the present work, have been studied extensively over the past century.
We refer to the pioneering works of Leray \cite{Ler:34} and Hopf \cite{Hopf:51} on global-in-time weak solutions, as well as to the works of Ladyzhenskaya \cite{Lad:69}, Fujita and Kato \cite{FK:64}, or Koch and Tataru \cite{MR1808843} on strong solutions.
In two spatial dimensions, the classical theory yields global-in-time uniqueness and regularity under natural finite-energy assumptions.
A systematic treatment of the deterministic Navier-Stokes theory can be found in the monographs of Temam \cite{Tem:79}, Sohr \cite{Sohr:01}, Galdi \cite{Gal:11}, and Bedrossian and Vicol \cite{MR4475666}.

The {\em stochastic} Navier-Stokes equations have also been studied extensively. Early rigorous contributions include the work of Bensoussan and Temam \cite{BT:73}, and a comprehensive early survey was given by Bensoussan~\cite{Ben:95}.
Martingale and stationary solutions for stochastic Navier-Stokes equations under general assumptions
on the diffusion coefficient were obtained by Flandoli and G{\k{a}}tarek \cite{FG:95}.
The long-time behavior of stochastic Navier-Stokes dynamics has been investigated in several directions, including ergodicity and invariant measures; we mention, among others, the works of Da Prato and Debussche \cite{DPD:03}, Hairer and Mattingly \cite{HM:06}, and Flandoli and Romito \cite{FR:08}.
The question whether noise may improve the well-posedness theory for fluid equations, especially in three dimensions, is also a central theme in the literature; see, for instance, the review by Bianchi and Flandoli \cite{BF:20} as well as the monograph by Flandoli and Luongo \cite{FL:23} and the references therein.

A related line of research concerns stochastic perturbations of fluid equations by transport-type noise, arising naturally from stochastic Lagrangian descriptions of fluid motion.
In particular, Holm introduced a stochastic variational framework for fluid dynamics leading to the so-called stochastic advection by Lie transport, or SALT, equations \cite{Hol:15}.
This framework has motivated a substantial amount of recent work on fluid equations with transport noise.
For Navier-Stokes equations, local and global well-posedness results with stochastic Lie transport and related noise structures have been obtained, for instance, by Goodair~\cite{Goo:23a, Goo:23b}.
For the role of transport noise for turbulent models including the Navier-Stokes equations, we also refer to the very recent article \cite{DP:26}.

Before addressing the stochastic FSI problem under consideration, let us first elaborate on the underlying deterministic model.
In fact, the mathematical theory of deterministic fluid-structure interaction problems is itself extensive.
The existence of global weak solutions to such FSI problems was established in works such as \cite{CDEG:05, Gra:08, MC13, LR:14, mindrilua2026existence}.
Strong well-posedness theory for such FSI problems is particularly delicate due to coupling at the interface and the limited regularity transferred between the fluid and the plate.
With regard to the dynamic coupling condition made precise in \eqref{eq:damped beam eq} below, the fluid pressure enters the plate equation.
For strong solutions, this generally requires invoking the so-called {\em added mass operator} in order to control the pressure at the boundary, since a priori, the fluid equations only yield control of the gradient of the pressure.
Local-in-time strong solutions were studied, e.g., by Beir\~ao da Veiga \cite{BdV:04}; in this context, we also mention the works \cite{GHL:19, MRR:20, MT:21, BMR:26}.
A remarkable result is due to Grandmont and Hillairet \cite{GH:16}, who proved global-in-time strong solutions for a two-dimensional Navier-Stokes fluid interacting with a one-dimensional viscoelastic beam.

We now describe the underlying deterministic problem.
For simplicity, and since this does not affect the analysis, all physical constants are normalized to one.
Let \(T>0\) and let $\cF_0 = (0,L)\times(-1,0)\subset\R^2$ be the rectangular fluid domain.
The fluid is assumed to be viscous, incompressible, and Newtonian.
Its velocity and pressure are denoted by $u \colon (0,T)\times\cF_0\to\R^2$ and $\pi \colon (0,T)\times\cF_0\to\R$, respectively.
The fluid equations are given by
\begin{equation}\label{eq:fluid eq}
    \begin{aligned}
        \del_t u+(u\cdot\nabla)u-\mdiv\sigma_\rf(u,\pi)
        &=0, \enspace \mdiv u = 0, &&\tin (0,T)\times\cF_0,
    \end{aligned}
\end{equation}
where the fluid stress tensor is defined by $\sigma_\rf(u,\pi) = (\nabla u+\nabla u^\top)-\pi I_2$.
The motion of the structure is described by a damped linear beam equation.
For the structure displacement $\eta \colon (0,T)\times(0,L)\to\R$, this equation reads
\begin{equation}\label{eq:damped beam eq}
    \del_{tt}\eta
    +\del_{xxxx}\eta
    +\del_{xxxxt}\eta
    =
    \phi(u,\pi),
    \tin (0,T)\times(0,L).
\end{equation}
The right-hand side $\phi(u,\pi)$ in \eqref{eq:damped beam eq} accounts for the continuity of normal stresses at the interface and is given by $\phi(u,\pi) = -\mre_2\cdot\sigma_\rf(u,\pi)\mre_2$, where \(\mre_2=\binom{0}{1}\) denotes the outer unit normal vector on the upper boundary of the fluid domain.

The system is complemented by coupling and boundary conditions.
The kinematic coupling condition expresses equality of the fluid and structure velocities at the interface and reads
\begin{equation}\label{eq:kin cc}
    u(t,x,0)
    =
    \del_t\eta(t,x)\mre_2,
    \tfor t \in (0,T), \enspace x \in (0,L).
\end{equation}
Moreover, we assume periodicity in the horizontal variable \(x\) for the fluid, the pressure, and the structure.
Since the fluid is assumed to be at rest at the lower boundary, the boundary conditions are given by
\begin{equation*}
    \begin{aligned}
        (u,\pi)(t,0,y)
        &=
        (u,\pi)(t,L,y), \enspace \eta(t,0) = \eta(t,L)
        &&\tfor t\in(0,T), \enspace y\in(-1,0),\\
        u(t,x,-1)
        &=
        0,
        &&\tfor t\in(0,T), \enspace x\in(0,L).
    \end{aligned}
\end{equation*}

The incompressibility condition in \eqref{eq:fluid eq} and the kinematic coupling condition \eqref{eq:kin cc} imply conservation of the spatial average of the beam displacement.
Indeed, invoking the boundary conditions of $u$, for all $t > 0$, we get $\frac{\rd}{\rd t}\int_0^L \eta(t,x)\srd x = \int_0^L u(t,x,0)  \cdot \mre_2 \srd x = \int_{\cF_0} \mdiv u \srd(x,y) = 0$, so the quantity $m\coloneqq \int_0^L \eta(t,x)\srd x$ is constant in time.
Thus, if the initial displacement $\eta_1^0(x) = \eta(0,x)$ satisfies $\int_0^L \eta_1^0(x)\srd x=0$, then the same property is preserved by \(\eta(t,\cdot)\) for all \(t\in(0,T)\).
This motivates the use of the space of functions with zero spatial average $\rL_\rm^q(0,L)$ with associated projection $P_\rm \colon \rL^q(0,L) \to \rL_\rm^q(0,L)$, defined by
\begin{equation}\label{eq:avg zero space and projection}
    \rL_\rm^q(0,L)
    \coloneqq
    \left\{
        f\in\rL^q(0,L):
        \int_0^L f\srd x=0
    \right\}, \tand P_\rm f
    \coloneqq
    f-\frac1L\int_0^L f\srd x,
    \tfor f\in\rL^q(0,L).
\end{equation}
After applying this projection to the damped beam equation, the deterministic FSI system underlying the present work can be written as
\begin{equation}\label{eq:underlying det FSI probl}
    \left\{
    \begin{aligned}
        \del_t u+(u\cdot\nabla)u-\mdiv\sigma_\rf(u,\pi)
        &=0, \enspace \mdiv u = 0,
        &&\tin (0,T)\times\cF_0,\\
        \del_{tt}\eta
        +P_\rm\del_{xxxx}\eta
        +P_\rm\del_{xxxxt}\eta
        &=
        -P_\rm\bigl(\mre_2\cdot\sigma_\rf(u,\pi)\mre_2\bigr),
        &&\tin (0,T)\times(0,L),\\
        u(t,x,0)
        &=
        \del_t\eta(t,x)\mre_2,
        &&\tfor t\in(0,T), \enspace x\in(0,L),\\
        (u,\pi)(t,0,y)
        &=
        (u,\pi)(t,L,y), \enspace \eta(t,0) = \eta(t,L),
        &&\tfor t\in(0,T), \enspace y\in(-1,0),\\
        u(t,x,-1)
        &=
        0,
        &&\tfor t\in(0,T), \enspace x\in(0,L),
    \end{aligned}
    \right.
\end{equation}
supplemented by suitable initial data $u_0(x) = u(0,x)$, $\eta_1^0(x) = \eta(0,x)$, and $\eta_2^0(x) = \del_t\eta(0,x)$.

The analysis of {\em stochastic} FSI problems, which is motivated by real-world applications such as blood flow in vessels and the associated question whether FSI problems are robust under stochastic perturbations, is much more recent.
The first result in this direction appears to be due to Kuan and \v{C}ani\'c \cite{KC:22}, who studied a reduced linearly coupled model described by a stochastic viscous wave equation.
In \cite{KC:24}, the same authors considered a benchmark stochastic FSI problem coupling a 2D Stokes fluid to a 1D linearly elastic membrane.
Nonlinear stochastic FSI problems on moving domains were then studied by Tawri and \v{C}ani\'c \cite{TC:25}, who proved the existence of martingale solutions for a 2D Navier-Stokes fluid coupled to an elastic lateral wall.
Further developments include stochastic FSI models with nonzero longitudinal displacement \cite{Taw:25} and Navier-slip boundary conditions \cite{Taw:24}.
In three spatial dimensions, Breit, Mensah, and Moyo \cite{BMM:24} obtained martingale solutions for an incompressible Navier-Stokes fluid interacting with a Koiter-type shell subject to transport noise.
All these works construct solutions in a {\em weak PDE sense}.
For an overview of deterministic and stochastic FSI problems, we also refer to the recent monograph \cite{CKMT:25}.

We now describe the stochastic problem.
Let \((\Omega,\cA,\bP)\) be a probability space endowed with a filtration \(\frF=(\frF_t)_t\), and let \(\cH\) be a separable Hilbert space.
The precise definition of an \(\frF\)-cylindrical Wiener process $W_\cH \colon \rL^2((0,\infty);\cH)\to\rL^2(\Omega)$ will be recalled in \autoref{sec:lin stoch probl}.
For an orthonormal basis \((e_n)_n\) of \(\cH\), the processes $\beta_n(t)\coloneqq W_\cH(t)e_n$ are standard \(\frF\)-Brownian motions, and one has the representation
\begin{equation}\label{eq:repr cylindr Wiener proc}
    W_\cH(t)\colon\cH\to\rL^2(\Omega),
    \twith
    W_\cH(t)f
    =
    \sum_{n=1}^\infty
    \beta_n(t)\langle f,e_n\rangle_\cH.
\end{equation}
The stochastic integral with respect to \(W_\cH\) is defined for processes with values in spaces of \(\gamma\)-radonifying operators; see \cite{vNVW:12a} and also \cite[Ch.~9]{HvNVW:17}.
The relevant spaces and assumptions will be specified in the context of the linear stochastic problem in \autoref{sec:lin stoch probl}.

For suitable stochastic forcing terms \(f\) and \(g\), and for the fluid velocity, the fluid pressure, and the structure displacement being random variables $\tu\colon (0,T)\times\Omega\times\cF_0\to\R^2$, $\tpi\colon (0,T)\times\Omega\times\cF_0\to\R$, and $\teta\colon (0,T)\times\Omega\times(0,L)\to\R$, the linearly coupled stochastic FSI problem studied in this manuscript reads
\begin{equation}\label{eq:lin coupled stoch FSI probl}
    \left\{
    \begin{aligned}
        \rd\tu +(\tu\cdot\nabla)\tu\srd t-\mdiv\sigma_\rf(\tu,\tpi)\srd t
        &=
        f\srd W_\cH, \enspace \mdiv\tu = 0,
        &&\tin (0,T)\times\cF_0,\\
        \rd\del_t\teta
        +\bigl(P_\rm\del_{xxxx}\teta+P_\rm\del_{xxxxt}\teta\bigr)\srd t
        &=
        -P_\rm\bigl(\mre_2\cdot\sigma_\rf(\tu,\tpi)\mre_2\bigr)\srd t\\
        &\quad +g\srd W_\cH,
        &&\tin (0,T)\times(0,L),\\
        \tu(t,x,0)
        &=
        \del_t\teta(t,x)\mre_2,
        &&\tfor t\in(0,T), \enspace x\in(0,L),\\
        (\tu,\tpi)(t,0,y)
        =
        (\tu,\tpi)(t,L,y), \enspace \teta(t,0)
        &=
        \teta(t,L),
        &&\tfor t\in(0,T), \enspace y\in(-1,0),\\
        \tu(t,x,-1)
        &=
        0,
        &&\tfor t\in(0,T), \enspace x\in(0,L),\\
        \tu(0,\cdot)
        &=
        \tu_0(\cdot),
        &&\tin\cF_0,\\
        \teta(0,\cdot)
        &=
        \teta_1^0(\cdot),
        \enspace
        \del_t\teta(0,\cdot)
        =
        \teta_2^0(\cdot),
        &&\tin(0,L).
    \end{aligned}
    \right.
\end{equation}

The paraphrased main result is as follows.
For the precise statement, see \autoref{thm:global strong well posedness 2d1d}.

\begin{thm}\label{thm:simplified main result}
Under the compatibility conditions and regularity assumptions on the data made precise in \autoref{thm:global strong well posedness 2d1d}, the system~\eqref{eq:lin coupled stoch FSI probl} admits a unique, global, strong pathwise solution.
\end{thm}

To the best of our knowledge, this is the first global-in-time strong pathwise well-posedness result for a stochastically forced Navier-Stokes system coupled to a deformable elastic structure located at the fluid boundary.
In \cite{BR:26}, local strong well-posedness and blow-up criteria were obtained for a stochastic fluid-rigid body interaction problem.
The present manuscript continues this line of research, but the presence of an elastic boundary structure leads to a substantially different linear theory.
While maximal regularity for certain fluid-structure operators with elastic boundary dynamics has been established in deterministic settings, see for instance \cite{MT:21}, we prove here the boundedness of the \(\Hinfty\)-calculus for the fluid-structure operator associated with the damped Kirchhoff plate model.
This operator-theoretic result is of independent interest and is the key input for stochastic maximal regularity.

The strategy for analyzing \eqref{eq:lin coupled stoch FSI probl} is to split the solution into a linear stochastic part and a deterministic nonlinear remainder.
The linear stochastic problem is treated using the theory of stochastic maximal regularity developed by van Neerven, Veraar, and Weis \cite{vNVW:12a, vNVW:12b}.
A substantial part of the paper is thus devoted to proving that the fluid-structure operator \(A_\fs\), which encodes the Stokes dynamics, the damped plate dynamics, and the kinematic and dynamic coupling conditions, admits a bounded \(\Hinfty\)-calculus, which paves the way for the stochastic maximal regularity in a second step.
The non-diagonal domain caused by the kinematic coupling condition is handled by a decoupling procedure, while the pressure contribution in the plate equation requires lifting constructions and estimates in negative Sobolev spaces.
The perturbation argument is carried out at the level of the bounded \(\Hinfty\)-calculus, which is considerably more delicate than for sectoriality or maximal \(\rL^p\)-regularity; see \cite{McIY:90}.

The deterministic remainder is handled by quasilinear evolution methods \cite{PW:17, PSW:18, PS:16}, which yield local well-posedness and a blow-up criterion.
The passage from local to global solutions is not a direct consequence of the standard two-dimensional Navier-Stokes energy inequality.
Indeed, the blow-up criterion is formulated in the interpolation space associated with \(A_\fs\), and therefore requires uniform control of the fluid \(\rH^1\)-norm together with high Sobolev norms of the plate.
We obtain these bounds by combining the basic energy identity with higher-order estimates, elliptic regularity for the coupled Stokes problem, and a time-local absorption argument for the nonlinear convection term.

The manuscript is organized as follows.
\autoref{sec:prelims} addresses some preliminaries.
In \autoref{sec:Hoo-calculus fluid-structure op}, we prove the boundedness of the \(\Hinfty\)-calculus of the fluid-structure operator in this setting.
\autoref{sec:lin stoch probl} is dedicated to clarifying the stochastic setting and analyzing the linear stochastic problem by combining the bounded \(\Hinfty\)-calculus of the fluid-structure operator with stochastic maximal regularity theory.
In \autoref{sec:loc wp}, we use the solution constructed in \autoref{sec:lin stoch probl} to reduce the nonlinear stochastic problem to a deterministic non-autonomous one, and we prove local-in-time well-posedness together with blow-up criteria.
The purpose of \autoref{sec:glob wp lin coupled case} is to establish the a priori estimates needed to rule out the blow-up scenario for the deterministic part of the solution, thereby proving global strong pathwise well-posedness of the stochastic FSI problem.
In the final \autoref{sec:further discussion}, we discuss well-posedness results in the 3D/2D case.

\section{Preliminaries}\label{sec:prelims}

The purpose of this section is to introduce some pieces of notation such as function spaces, and to present some auxiliary (deterministic) tools that are needed for the analysis carried out in this section.
This includes the analysis of the Stokes operator on a cylindrical domain as well as the structure operator associated with the present plate problem.

We start with function spaces with periodic boundary conditions.
For $m \in \N_0$, we say that a smooth function $f \colon \ocF_0 \to \R$ is {\em periodic of order $m$} (with respect to the variable $x$) if $ \frac{\del^{\alpha}f}{\del x^\alpha}(0,y) = \frac{\del^{\alpha}f}{\del x^\alpha}(L,y)$ for all multi-indices $\alpha$ with $0 \le |\alpha| \le m$.
Likewise, one can define periodicity of order $m$ for a smooth function $g \colon [0,L] \to \R$.
In the following, we also use $\Gamma_l$ to denote the lateral boundary of the fluid domain, i.e., $\Gamma_l \coloneqq \{0\} \times (-1,0) \cup \{L\} \times (-1,0)$, and we introduce the spaces
\begin{equation*}
    \begin{aligned}
        \rC_\per^\infty(\ocF_0) 
        &\coloneqq \{f \in \rC^\infty(\ocF_0) : f \text{ is periodic of arbitrary order on } \Gamma_l\}, \tand\\
        \rC_\per^\infty([0,L])
        &\coloneqq \{f \in \rC^\infty([0,L]) : f \text{ is periodic of arbitrary order}\}.
    \end{aligned}
\end{equation*}
For $m \in \N_0$ and $q \in [1,\infty]$, we define the periodic Sobolev spaces by
\begin{equation*}
    \rW_\per^{m,q}(\cF_0) \coloneqq \overline{\rC_\per^\infty(\ocF_0)}^{\| \cdot \|_{\rW^{m,q}(\cF_0)}} \tand \rW_\per^{m,q}(0,L) \coloneqq \overline{\rC_\per^\infty([0,L])}^{\| \cdot \|_{\rW^{m,q}(0,L)}}.
\end{equation*}
Note that the spaces $\rW_\per^{0,q}(\cF_0)$ and $\rW_\per^{0,q}(0,L)$ are identified with $\rL^q(\cF_0)$ and $\rL^q(0,L)$, respectively.
Moreover, for $s \ge 0$, $p \in (1,\infty)$, and $q \in [1,\infty]$, the periodic Bessel potential, Sobolev-Slobodeckij, and Besov spaces can be defined by interpolation.
More precisely, for $[\cdot,\cdot]_\theta$ and $(\cdot,\cdot)_{\theta,p}$ denoting the complex and real interpolation functor, respectively, we set
\begin{equation*}
    \begin{aligned}
        \rH_\per^{s,q}(\cF_0) 
        &\coloneqq [\rL^q(\cF_0),\rW_\per^{m,q}(\cF_0)]_{\frac{s}{m}}, \enspace \rW_\per^{s,q}(\cF_0) \coloneqq (\rL^q(\cF_0),\rW_\per^{m,q}(\cF_0))_{\frac{s}{m},q}, \tand \\
        \rB_{qp,\per}^s(\cF_0) 
        &\coloneqq (\rL^q(\cF_0),\rW_\per^{m,q}(\cF_0))_{\frac{s}{m},p}.
    \end{aligned}
\end{equation*}
Note that the spaces $\rH_\per^{s,q}(0,L)$, $\rW_\per^{s,q}(0,L)$, and $\rB_{qp,\per}^s(0,L)$ are defined analogously.

In the sequel, we discuss certain properties of the concept of the boundedness of the $\Hinfty$-calculus that will be important in the context of the present fluid-structure operator.
Instead of elaborating on this concept in detail, we refer to the monographs \cite{DHP:03, KW:04}.
We only note that we will denote the class of sectorial operators on a Banach space $\rX$ by $\cS(\rX)$, while we use $\Hinfty(\rX)$ to denote the class of operators admitting a bounded $\Hinfty$-calculus.
For $A \in \cS(\rX)$, the associated spectral angle is denoted by $\phi_A$, while~$\phi_A^\infty$ represents the so-called $\Hinfty$-angle of an operator $A \in \Hinfty(\rX)$.

We start by discussing stability and perturbation properties of the bounded $\Hinfty$-calculus.
The assertion of~(a) can be found in \cite[Lemma~2.3.15]{Bra:24}, while for~(b), we refer, e.g., to \cite[Prop.~2.11(vi)]{DHP:03}, and the perturbation result in~(c) can be found in \cite[Cor.~3.3.15]{PS:16}.

\begin{lem}\label{lem:stability and pert results Hoo}
\begin{enumerate}[(a)]
    \item Let $A \in  \Hinfty(\rX)$ with $\phi_A^\infty \in [0,\pi)$, and let $\rY \subset \rX$ be a closed subspace such that there exists a projection $P \in \cL(\rX,\rY)$ with $P x \in \rD(A)$ and $P A x = A P x$ for all $x \in \rD(A)$.
    Then for $B \coloneqq P A$, it follows that $B \in \Hinfty(\rY)$ with $\Phi_B^\infty \le \Phi_A^\infty$.
    \item Let $\rX$ and $\rY$ be Banach spaces, consider $T \in \cL(\rX,\rY)$ with inverse $T^{-1}$, and set $A_1 \coloneqq T A T^{-1}$.
    Then $A \in \Hinfty(\rX)$ if and only if $A_1 \in \Hinfty(\rY)$, and we have $\phi_A^\infty = \phi_{A_1}^\infty$.
    \item Assume that $A \in \Hinfty(\rX)$ with angle $\phi_A^\infty$, and let $B$ be a linear operator in $\rX$ such that for $\alpha \in [0,1)$, we have $\rD(A^\alpha) \subset \rD(B)$, and for constants, it holds that
    \begin{equation*}
        \| B x \| \le a \cdot \| x \| + b \cdot \| A^\alpha x \|, \tfor x \in \rD(A^\alpha).
    \end{equation*}
    If $A + B \in \cS(\rX)$ with $\phi_{A + B}$ and $0 \in \rho(A+B)$, then $A + B \in \Hinfty(\rX)$ and $\phi_{A+B}^\infty \le \max\{\phi_A^\infty,\phi_{A+B}\}$.
\end{enumerate}
\end{lem}

The following results on the $\Hinfty$-calculus of block operator matrices will be crucial for the analysis of the fluid-structure operator.
We refer here to \cite[Cor.~7.1]{AH:23} for part~(a) and \cite[Cor.~7.11]{AH:23} for part~(b).

\begin{lem}\label{lem:bdd Hoo-calculus block op matrix}
\begin{enumerate}[(a)]
    \item For spaces $\rX_1$, $\rX_2$ and operators $A \colon \rD(A) \subset \rX_1 \to \rX_1$, $B \colon \rD(D) \subset \rX_2 \to \rX_1$ and $D \colon \rD(D) \subset \rX_2 \to \rX_2$, consider the operator matrix $\cA \colon \rD(\cA) = \rD(A) \times \rD(D) \subset \rX \to \rX$, where $\rX = \rX_1 \times \rX_2$, defined by
    \begin{equation*}
        \cA = \begin{pmatrix}
            A & B\\
            0 & D
        \end{pmatrix}.
    \end{equation*}
    If $A \in \Hinfty(\rX_1)$ and $D \in \Hinfty(\rX_2)$ with $\phi_A^\infty$ and $\phi_D^\infty$, and if there exists $\delta > 0$ such that
    \begin{equation*}
        B(\rD(D^{1+\delta})) \subset \rD(A^\delta) \tand \| A^\delta B y \|_{\rX_1} \le C \cdot \| D^{1+\delta} y \|_{\rX_2}, \tforall y \in \rD(D^{1+\delta})
    \end{equation*}
    for some constant $C > 0$, then $\cA \in \Hinfty(\rX)$ with $\phi_\cA^\infty \le \max\{\phi_A^\infty,\phi_D^\infty\}$.
    \item Let $S \colon \rD(S) \subset \rY_0 \to \rY_0$ and $T \colon \rD(T) \subset \rY_0 \to \rY_0$ be operators on a Banach space $\rY_0$, and suppose that $T$ is closed and densely defined, and $S$ is relatively $T$-bounded, i.e., there are $C_0$, $C_1 > 0$ such that $\| S y \|_{\rY_0} \le C_0 \cdot \| T y \|_{\rY_0} + C_1 \cdot \| y \|_{\rY_0}$ for all $y \in \rD(T)$.
    Then for $\rY \coloneqq \rD(T) \times \rY_0$
    \begin{equation*}
        \tand \cB \colon \rD(\cB) \subset \rY \to \rY \enspace \text{defined by} \enspace \cB = \begin{pmatrix}
            0 & -\Id\\
            S & T
        \end{pmatrix}, \twith \rD(\cB) = \rD(T) \times \rD(T),
    \end{equation*}
    if $T \in \Hinfty(\rY_0)$ with $\phi_T^\infty \in [0,\pi)$, then there is $\omega \ge 0$ such that $\cB + \omega \in \Hinfty(\rY)$ with $\phi_{\cB}^\infty \in [0,\phi_T^\infty]$.
\end{enumerate}
\end{lem}

Next, we discuss two relevant (deterministic) consequences of the boundedness of the $\Hinfty$-calculus.
The one stated in~(a) below concerns a characterization of the fractional power domains. 
It follows from the fact that the boundedness of the $\Hinfty$-calculus implies the {\em boundedness of imaginary powers}, denoted by $A \in \cBIP(\rX)$, see, e.g., \cite[Sec.~4.4]{DHP:03}, together with the fact that $A \in \cBIP(\rX)$ yields a characterization of the fractional power domains, see, for example, \cite[Thm.~3.3.7]{PS:16}.
The second consequence is on the {\em deterministic} maximal $\rL^p$-regularity that especially follows from the bounded $\Hinfty$-calculus on so-called {\em UMD (unconditional martingale differences)} spaces.
Many classical function spaces such as $\rL^q$-spaces, Sobolev spaces $\rW^{s,q}$, Bessel potential spaces $\rH^{s,q}$, and Besov spaces $\rB_{qp}^s$ are UMD spaces if $p$, $q \in (1,\infty)$, see also \cite[Secs.~III.4.4 and III.4.5]{Ama:95}.
The assertion of~(b) is implied by the fact that the bounded $\Hinfty$-calculus especially implies the so-called {\em $\cR$-sectoriality}, see again \cite[Sec.~4.4]{DHP:03}, and the celebrated result on the characterization of the maximal $\rL^p$-regularity in terms of $\cR$-sectoriality due to Weis \cite[Thm.~4.2]{Wei:01}.

\begin{lem}\label{lem:cons of H00}
\begin{enumerate}[(a)]
    \item Let $A \in \Hinfty(\rX)$.
    Then for $\theta \in (0,1)$, we have $\rX_{A^\theta} \simeq [\rX,\rD(A)]_\theta$, where $\rX_{A^\theta}$ is defined by $\rX_{A^\theta} \coloneqq (\rD(A^\theta),\|\cdot\|_\theta)$, and $\| x \|_\theta \coloneqq \| x \| + \| A^\theta x \|$.
    \item Let $\rX$ be a UMD space, and $A \in \Hinfty(\rX)$ with $\phi_A^\infty < \frac{\pi}{2}$.
    Then $A$ has maximal $\rL^p$-regularity.
\end{enumerate}
\end{lem}

We now make precise the fluid operator in the present setting.
At this stage, we also introduce two pieces of notation for the upper and bottom boundaries, namely, we set $\Gamma_u \coloneqq (0,L) \times \{0\}$ and $\Gamma_b \coloneqq (0,L) \times \{-1\}$.

We start with the space of solenoidal vector fields $\rL_\sigma^q(\cF_0)$.
In fact, it is defined by
\begin{equation}\label{eq:L_sigma^q}
    \rL_\sigma^q(\cF_0) \coloneqq \overline{\{f \in \rC_\rc^\infty(\cF_0)^2 : \mdiv f = 0 \tin \cF_0\}}^{\| \cdot \|_{\rL^q(\cF_0)^2}}.
\end{equation}
In a similar way as in \cite{SS:92}, it can be shown that the space $\rL_\sigma^q(\cF_0)$ admits the characterization
\begin{equation}\label{eq:char L_sigma^q}
    \rL_\sigma^q(\cF_0) = \{f \in \rL^q(\cF_0)^2 : \mdiv f = 0 \tand f \cdot \nu = 0 \ton \Gamma_u \cup \Gamma_b\},
\end{equation}
where $\nu = \pm \mre_2$ denotes the outer unit normal vector.

We refer to the projection $\bP \colon \rL^q(\cF_0)^2 \to \rL_\sigma^q(\cF_0)$ onto the space $\rL_\sigma^q(\cF_0)$ as the \textit{Helmholtz projection}, whose existence in $\rL^q(\cF_0)$ for cylindrical domains as the present fluid domain has been shown in \cite{Nau:15}, and we define the associated Stokes operator by
\begin{equation}\label{eq:Stokes op}
    A_0 u \coloneqq \bP \Delta u, \tfor u \in \rD(A_0) \coloneqq \left\{u \in \rW_\per^{2,q}(\cF_0)^2 \cap \rL_\sigma^q(\cF_0) : u = 0 \ton \Gamma_u \cup \Gamma_b\right\}. 
\end{equation}

The following result on the boundedness of the $\Hinfty$-calculus of the Stokes operator can be obtained in a similar manner as \cite[Thm.~9.9]{Nau:12}.
Let us provide a few more details on the ideas to obtain \autoref{lem:Hoo-calc Stokes op}.
In fact, the situation of a cylindrical domain with periodicity at the lateral boundary can be reduced to the case of an infinite layer and then also the half space by means of a suitable extension and cutoff procedure.
For the boundedness of the $\Hinfty$-calculus of the Stokes operator on the half space, we refer to \cite[Thm.~6.3]{DHP:01}.

\begin{lem}\label{lem:Hoo-calc Stokes op}
For $q \in (1,\infty)$, the Stokes operator $A_0$ satisfies $-A_0 \in \Hinfty(\rL_\sigma^q(\cF_0))$ with angle $\phi_{-A_0}^\infty < \frac{\pi}{2}$.
\end{lem}

Next, we show that the operator matrix corresponding to the reformulation of the plate equation as a first-order-in-time problem admits a bounded $\Hinfty$-calculus.

For $\rL_\rm^q(0,L)$ as introduced in \eqref{eq:avg zero space and projection}, we define the ground space by $\rX_0^\rs \coloneqq \rW_\per^{4,q}(0,L) \cap \rL_\rm^q(0,L) \times \rL_\rm^q(0,L)$.
We will denote the $\rL^q(0,L)$-realization of $\del_{xxxx}$ by $\Delta_\rs^2$, i.e., $\Delta_\rs^2 g \coloneqq \del_{xxxx} g$ for $g \in \rW_\per^{4,q}(0,L)$.
Recalling $P_\rm \colon \rL^q(0,L) \to \rL_\rm^q(0,L)$ from \eqref{eq:avg zero space and projection}, we set the operator matrix $A_\rs \colon \rD(A_\rs) \subset \rX_0^\rs \to \rX_0^\rs$ to be
\begin{equation}\label{eq:strongly damped plate op matrix}
    A_\rs \coloneqq \begin{pmatrix}
        0 & \Id\\
        -P_\rm \Delta_\rs^2 & -P_\rm \Delta_\rs^2
    \end{pmatrix}, \twith \rD(A_\rs) = \rW_\per^{4,q}(0,L) \cap \rL_\rm^q(0,L) \times \rW_\per^{4,q}(0,L) \cap \rL_\rm^q(0,L).
\end{equation}
Before studying $A_\rs$ from \eqref{eq:strongly damped plate op matrix}, we first consider the operator $P_\rm \Delta_s^2 \colon \rW_\per^{4,q}(0,L) \cap \rL_\rm^q(0,L) \to \rL_\rm^q(0,L)$.

\begin{lem}\label{lem:bdd Hoo-calculus proj bi-Laplacian}
For $q \in (1,\infty)$, we have $P_\rm \Delta_\rs^2 \in \Hinfty(\rL_\rm^q(0,L))$ with $\phi_{P_\rm \Delta_\rs^2}^\infty < \frac{\pi}{2}$.
\end{lem}

\begin{proof}
First, \cite[Thm.~4.6]{NS:12} yields that there exists $\omega \ge 0$ such that $\Delta_\rs^2 + \omega \in \Hinfty(\rL^q(0,L))$ with $\phi_{\Delta_\rs^2 + \omega}^\infty < \frac{\pi}{2}$.
The assertion of the lemma then follows from \autoref{lem:stability and pert results Hoo}(a), since $P_\rm$ fits into the framework of \autoref{lem:stability and pert results Hoo}(a) by the periodic boundary conditions.
The shift can be omitted thanks to the fact that the projection $P_\rm$ rules out the constants, thereby leading to the invertibility of $P_\rm \Delta_\rs^2$.
\end{proof}

We now discuss the boundedness of the $\Hinfty$-calculus of $A_\rs$.

\begin{lem}\label{lem:Hoo-calc thick layer op matrix}
Let $q \in (1,\infty)$.
Then for $A_\rs$ from \eqref{eq:strongly damped plate op matrix}, it holds that $-A_\rs \in \Hinfty(\rX_0^\rs)$ with $\phi_{-A_\rs}^\infty  < \frac{\pi}{2}$.
\end{lem}

\begin{proof}
From \autoref{lem:bdd Hoo-calculus proj bi-Laplacian}, it follows that $P_\rm \Delta_\rs^2 \in \Hinfty(\rL_\rm^q(0,L))$ with $\phi_{P_\rm \Delta_\rs^2}^\infty  < \frac{\pi}{2}$.
Moreover, the structure of $A_\rs$ directly yields that we are in the scope of \autoref{lem:bdd Hoo-calculus block op matrix}(b), since $P_\rm \Delta_\rs^2$ is relatively bounded with respect to itself.
Thus, up to a shift, the assertion of the lemma is a consequence of \autoref{lem:bdd Hoo-calculus block op matrix}(b).
The shift can be removed upon invoking the explicit form of the inverse of $A_\rs^{-1}$ given by
\begin{equation*}
    A_\rs^{-1} = \begin{pmatrix}
        -\Id & (-P_\rm \Delta_\rs^2)^{-1}\\
        \Id & 0
    \end{pmatrix}. \qedhere
\end{equation*}
\end{proof}

\section{Bounded $\Hinfty$-calculus of the fluid-structure operator}\label{sec:Hoo-calculus fluid-structure op}

In this section, we investigate the fluid-structure operator associated with the linearized problem.
The main result, \autoref{thm:bdd Hoo-calculus fluid-structure op}, discusses the boundedness of the $\Hinfty$-calculus of the fluid-structure operator.
On the one hand, this will pave the way for the stochastic maximal regularity discussed in \autoref{sec:lin stoch probl}.
On the other hand, this result goes beyond the state-of-the-art concerning operator theoretic results in the context of fluid-structure interaction problems with the structure located at a part of the fluid boundary, and it is of independent interest.

The linearized {\em deterministic} problem is given by
\begin{equation}\label{eq:lin per multilayered interaction probl}
    \left\{
    \begin{aligned}
        \del_t u - \mdiv \sigma_\rf(u,\pi)
        &= f, \enspace \mdiv u = 0, && t > 0, \enspace (x,y) \in \cF_0,\\
        \del_t \eta_1
        &= \eta_2, && t > 0, \enspace x \in (0,L),\\
        \del_t \eta_2 + P_\rm \Delta_s^2 \eta_1 + P_\rm \Delta_s^2 \eta_2
        &= -P_\rm \left(\mre_2 \cdot \sigma_\rf(u,\pi) \mre_2 \right) + P_\rm g, && t > 0, \enspace x \in (0,L),\\
        u(t,x,0)
        &= P_\rm(\eta_2(t,x)) \mre_2, && t > 0, \enspace x \in (0,L),\\
        u(t,x,-1)
        &= 0, && t > 0, \enspace x \in (0,L),\\
        (u,\pi)(t,0,y) = (u,\pi)(t,L,y), \enspace (\eta_1,\eta_2)(t,0)
        &= (\eta_1,\eta_2)(t,L), && t > 0, \enspace y \in (-1,0),\\
        u(0,x,y) 
        &= u_0(x,y), && (x,y) \in \cF_0,\\
        \eta_1(0,x) = \eta_{1,0}(x) \tand \eta_2(0,x) 
        &= \eta_{2,0}(x), && x \in (0,L).
    \end{aligned}
    \right.
\end{equation}

For the reformulation of \eqref{eq:lin per multilayered interaction probl} in operator form, we start by discussing the lifting of the kinematic coupling condition in the Stokes problem.
In fact, we consider the following stationary Stokes problem:
\begin{equation}\label{eq:stat Stokes probl with inhom bc}
    \left\{
    \begin{aligned}
        -\mdiv \sigma_\rf(w,\psi)
        &= 0, \enspace \mdiv w = 0, &&\tin \cF_0,\\
        w(x,0)
        &= b(x), &&\tfor x \in (0,L),\\
        w(x,-1)
        &= 0, &&\tfor x \in (0,L),\\
        (w,\psi)(0,y) = (w,\psi)(L,y), \enspace b(0)
        &= b(L), &&\tfor y \in (-1,0).
    \end{aligned}
    \right.
\end{equation}
The case of $b$ with low regularity will be of particular importance to establish the bounded $\Hinfty$-calculus of the fluid-structure operator.
In the situation of a sufficiently regular boundary, for $b \in \rW^{2-\nicefrac{1}{q},q}$, the analogue of \autoref{lem:Stokes lifting} is classical, and we refer e.g., to \cite[Prop.~2.3]{Tem:79}.
The present case of a cylindrical domain can be handled in a similar manner upon using the extension and cutoff procedure mentioned in the context of the Stokes operator to reduce the problem to an infinite layer, or the half-space case.
For the (non-stationary) Stokes equations on the half-space and in the situation of inhomogeneous Dirichlet boundary conditions, we also refer to \cite[Sec.~7.2]{PS:16}.

We require some further function spaces before stating \autoref{lem:Stokes lifting}.
For $q \in (1,\infty)$, we define
\begin{equation*}
    \rW_{\per,n}^{s,q}(0,L) \coloneqq \left\{g \in \rW_\per^{s,q}(0,L)^2 : \int_0^L g \cdot \mre_2 \srd x = 0\right\}, \tfor s \ge 0, \tand  \rW_{\per,n}^{s,q}(0,L) \coloneqq (\rW_{\per,n}^{-s,q'}(0,L))',
\end{equation*}
for $s < 0$, and where $q' \in (1,\infty)$ is such that $\frac{1}{q} + \frac{1}{q'} = 1$.
The subscript ``$\mathrm{n}$'' represents ``normal condition'' here.
Likewise, we define
\begin{equation*}
    \cW^{s,q}(\cF_0) /\R \coloneqq \rW_\per^{s,q}(\cF_0) /\R, \tfor s \ge 0, \tand \cW^{s,q}(\cF_0) /\R \coloneqq (\rW_\per^{-s,q'}(\cF_0) /\R)', \tfor s < 0.
\end{equation*}
The situation of data $b$ with negative Sobolev regularity was addressed in \cite[Thm.~A.1 and Cor.~A.1]{Ray:07} in the Hilbert space setting, and in the case of standard $\rC^2$-domains.
The result below can be obtained in a similar way as \cite[Cor.~A.1]{Ray:07}. 

\begin{lem}\label{lem:Stokes lifting}
Let $q \in (1,\infty)$ and $-\frac{1}{q} \le s \le 2 - \frac{1}{q}$, and consider $b \in \rW_{\per,n}^{s,q}(0,L)$.
Then there exists a unique solution $(w,\psi) \in \rW_\per^{s+\nicefrac{1}{q},q}(\cF_0)^2 \times \cW^{s-1+\nicefrac{1}{q},q}(\cF_0) / \R$ to \eqref{eq:stat Stokes probl with inhom bc}.
\end{lem}

Based on \autoref{lem:Stokes lifting}, we now introduce suitable Stokes lifting operators.
In fact, for $s \ge -\frac{1}{q}$ as well as $q$, $q' \in (1,\infty)$ with $\frac{1}{q} + \frac{1}{q'} = 1$, we define
\begin{equation*}
    \rW_{\per,\rm}^{s,q}(0,L) \coloneqq \rW_\per^{s,q}(0,L) \cap \rL_\rm^q(0,L), \tfor s \ge 0, \tand \rW_{\per,\rm}^{s,q}(0,L) \coloneqq (\rW_{\per,\rm}^{-s,q'}(0,L))', \tfor s < 0.
\end{equation*}
We then observe that if $b \in \rW_{\per,\rm}^{s,q}(0,L)$, then for $b' \coloneqq P_\rm(b) \mre_2$, it follows that $b' \in \rW_{\per,n}^{s,q}(0,L)$.
This gives rise to the following lemma.

\begin{lem}\label{lem:mapping props Stokes lifting}
Let $q \in (1,\infty)$ and $s \in [-\frac{1}{q},2-\frac{1}{q}]$.
Then the Stokes lifting operators $D_\fl$ and $D_\pr$ defined by $D_\fl b \coloneqq w$ and $D_\pr b \coloneqq \psi$, where $(w,\psi)$ is the unique solution to \eqref{eq:stat Stokes probl with inhom bc} with boundary data $b' = P_\rm(b) \mre_2$, satisfy $D_\fl \in \cL(\rW_{\per,\rm}^{s,q}(0,L),\rW_\per^{s+\nicefrac{1}{q},q}(\cF_0)^2)$ and $D_\pr \in \cL(\rW_{\per,\rm}^{s,q}(0,L),\cW^{s-1+\nicefrac{1}{q},q}(\cF_0) / \R)$.
\end{lem}

After addressing the lifting of the kinematic boundary conditions, we next elaborate on the way to handle the pressure.
This is done by solving suitable Neumann problems.
More precisely, we consider
\begin{equation}\label{eq:weak Neumann probl}
    \Delta \varphi = 0 \tin \cF_0, \enspace \del_\nu \varphi = c \ton \Gamma_u \cup \Gamma_b, \enspace \varphi \text{ periodic } \ton \Gamma_l, \tand \int_{\cF_0} \varphi \srd (x,y) = 0,
\end{equation}
where, as above, $\nu = \pm \mre_2$ represents the outer unit normal vector to $\Gamma_u$ and $\Gamma_b$.
Note that for $c$, we will also identify $\Gamma_u$ and $\Gamma_b$ with $(0,L)$.
If $\varphi$ is a solution to \eqref{eq:weak Neumann probl}, then the divergence theorem yields that $\int_0^L c \srd x = 0$.
Below, the subscript ``$\rm$'' in the space $\rW_{\per,\rm}^{s+1+\nicefrac{1}{q}-\eps,q}(\cF_0)$ indicates again that this space incorporates functions with average zero on $\cF_0$. 
We now define the solution operator $N$ to the Neumann problem \eqref{eq:weak Neumann probl} by $N c \coloneqq \varphi$.
In a similar way as in \cite[Thm.~4.2 and~4.3]{LM:62}, for $q \in (1,\infty)$ and $s \in (-1,1-\frac{1}{q}]$, it can be shown that
\begin{equation}\label{eq:mapping props Neumann op}
    N \in \cL(\rW_{\per,\rm}^{s,q}(\Gamma_u \cup \Gamma_b),\rW_{\per,\rm}^{s+1+\nicefrac{1}{q}-\eps,q}(\cF_0)),
\end{equation}
where $\eps \ge 0$ can be chosen equal to zero provided $s$ is not an integer.
For practical reasons, we also define the adjusted Neumann lifting operator $N_1$ by
\begin{equation}\label{eq:lifting op N_1}
    N_1 c_1 = N c, \twhere c = 
    \left\{
    \begin{aligned}
        c_1, &\ton \Gamma_u,\\
        0, &\ton \Gamma_b.
    \end{aligned}
    \right.
\end{equation}
Finally, we invoke the Neumann lifting for the Helmholtz projection.
In fact, we investigate the problem 
\begin{equation}\label{eq:weak Neumann probl Helmholtz}
    \Delta \varphi = \mdiv f, \tin \cF_0, \enspace \del_\nu \varphi = f \cdot \nu, \ton \Gamma_u \cup \Gamma_b, \enspace \varphi \text{ periodic } \ton \Gamma_l,
\end{equation}
so $\varphi \in \rW_\per^{1,q}(\cF_0)$ solves $\int_{\cF_0} \nabla \varphi \cdot \nabla \psi \srd y = \int_{\cF_0} f \cdot \nabla \psi \srd y$, for $\psi \in \rW_\per^{1,q'}(\cF_0)$ and $\nicefrac{1}{q} + \nicefrac{1}{q'} = 1$.
For the solvability of this problem in a slightly different geometric setting, we refer to \cite{SS:92}.
For $\varphi$ solving \eqref{eq:weak Neumann probl Helmholtz}, we define the solution operator $N_2 \in \cL(\rL^q(\cF_0)^2,\rW_{\per,\rm}^{1,q}(\cF_0))$ by
\begin{equation}\label{eq:lifting op N_2}
    N_2 f \coloneqq \varphi.
\end{equation}

We are now in a position to reformulate \eqref{eq:lin per multilayered interaction probl} in terms of the above lifting operators.
In fact, we have the following result.
We observe that the forcing $-P_\rm \left(\mre_2 \cdot \left.\sigma_\rf(u,\pi)\right|_{\Gamma_u} \mre_2 \right)$ exerted by the fluid on the structure via the dynamic coupling condition reduces to the pressure $P_\rm(\left.\pi \right|_{\Gamma_u})$ due to the divergence free condition, the periodic boundary conditions, and the kinematic coupling condition.

\begin{lem}\label{lem:reform lin probl}
Let $p$, $q \in (1,\infty)$, and consider
\begin{equation*}
    \begin{aligned}
        u
        &\in \rW^{1,p}(0,T;\rL^q(\cF_0)^2) \cap \rL^p(0,T;\rW_\per^{2,q}(\cF_0)^2), \enspace \pi \in \rL^p(0,T;\rW_{\per,\rm}^{1,q}(\cF_0)),\\
        \eta_1 
        &\in \rW^{2,p}(0,T;\rL_\rm^q(0,L)) \cap \rW^{1,p}(0,T;\rW_{\per,\rm}^{2,q}(0,L)) \cap \rL^p(0,T;\rW_{\per,\rm}^{4,q}(0,L)), \tand\\
        \eta_2 
        &\in \rW^{1,p}(0,T;\rL_\rm^q(0,L)) \cap \rL^p(0,T;\rW_{\per,\rm}^{4,q}(0,L)). 
    \end{aligned}
\end{equation*}
Then $(u,\pi,\eta_1,\eta_2)$ is a solution to \eqref{eq:lin per multilayered interaction probl} if and only if
\begin{equation*}
    \left\{
    \begin{aligned}
        \bP u'
        &= A_0 \bP(u - D_\fl \eta_2) + \bP f, &&\tin (0,T),\\
        \eta_1'
        &= \eta_2, &&\tin (0,T),\\
        (\Id + \gamma_\rm N_1) \eta_2' + P_\rm \Delta_\rs^2 \eta_1 + P_\rm \Delta_\rs^2 \eta_2
        &= \gamma_\rm N(\Delta \bP u \cdot \nu) + P_\rm g + \gamma_\rm N_2 f, &&\tin (0,T),\\
        (\Id - \bP)u
        &= (\Id - \bP)D_\fl \eta_2,\\
        \pi
        &= N(\Delta \bP u \cdot \nu) - N_1 \eta_2' + N_2 f,\\
        (\bP u,\eta_1,\eta_2)(0)
        &= (\bP u_0,\eta_{1,0},\eta_{2,0}),
    \end{aligned}
    \right.
\end{equation*}
where $\gamma_\rm \coloneqq P_\rm \tr_{\Gamma_u}$ denotes the modified trace to the upper boundary $\Gamma_u$.
\end{lem}

Next, we introduce the so-called added mass operator defined by $M_s \coloneqq \Id + \gamma_\rm N_1$.
The result below on the invertibility of $M_s$ and its mapping properties will be important for the boundedness of the $\Hinfty$-calculus of the fluid-structure operator.
Note that this seems to be the first result discussing the added mass operator on Sobolev spaces with negative regularity.

\begin{lem}\label{lem:mapping props added mass}
Let $s \in (\max\{-1,-\frac{3}{2}+\frac{1}{q}\},1)$ and $q \in (1,\infty)$.
Then 
\begin{enumerate}[(a)]
    \item the added mass operator $M_s = \Id + \gamma_\rm N_1$ is an automorphism on $\rW_{\per,\rm}^{s,q}(0,L)$,
    \item for $\eps > 0$ sufficiently small, it holds that $M_s^{-1} - \Id \in \cL(\rW_{\per,\rm}^{s,q}(0,L),\rW_{\per,\rm}^{s+1-\eps,q}(0,L))$, and
    \item the operator $M_s^{-1} - \Id$ is compact on $\rL_\rm^q(0,L)$.
\end{enumerate}
\end{lem}

\begin{proof}
From \eqref{eq:mapping props Neumann op} and \eqref{eq:lifting op N_1}, for $\eps > 0$ sufficiently small, we get $N_1 \in \cL(\rW_{\per,\rm}^{s,q}(0,L),\rW_{\per,\rm}^{s+1+\nicefrac{1}{q}-\eps,q}(\cF_0))$.
In particular, $\eps > 0$ can be chosen sufficiently small such that $s + 1 + \frac{1}{q} - \eps > \frac{1}{q}$.
Thus, the trace and then also $\gamma_\rm$ are well-defined on $\rW_{\per,\rm}^{s+1+\nicefrac{1}{q}-\eps,q}(\cF_0)$, and we obtain
\begin{equation}\label{eq:mapping props gamma_m N_1}
    \gamma_\rm N_1 \in \cL(\rW_{\per,\rm}^{s,q}(0,L),\rW_{\per,\rm}^{s+1-\eps,q}(0,L)),
\end{equation}
where $s+1-\eps > 0$.
In order to show the assertion of~(a), it hence suffices to show that $\ker M_s = \{0\}$.
To this end, let $f \in \rW_{\per,\rm}^{s,q}(0,L)$, and suppose that $f \in \ker M_s$.
In particular, we find that $f = - \gamma_\rm N_1 f$.
Together with $s > -\frac{3}{2} + \frac{1}{q}$ and a choice of $\eps > 0$ sufficiently small such that $s - \eps > -\frac{3}{2} + \frac{1}{q}$ as well as the resulting Sobolev embedding $\rW^{s+1-\eps,q}(0,L) \hookrightarrow \rL^2(0,L)$, this yields that $f \in \rW_{\per,\rm}^{s+1-\eps,q}(0,L) \hookrightarrow \rL_\rm^2(0,L)$.
Therefore, $N_1 f \in \rH_{\per,\rm}^1(\cF_0)$, and setting $g \coloneqq N_1 f$, we test the identity $(\Id + \gamma_\rm N_1)f = 0$ by $f$, integrate by parts, and exploit the definition of $g = N_1 f$ to get
\begin{equation}\label{eq:proof of ker M_s equals 0}
    0 = \int_0^L [(\Id + \gamma_\rm N_1)f] f \srd x = \int_0^L f^2 \srd x + \int_{\cF_0} |\nabla g|^2 \srd(x,y).
\end{equation}
This implies that $f = 0$ and thus $\ker M_s = \{0\}$, completing the proof of~(a).
Let us now invoke
\begin{equation*}
    M_s^{-1} - \Id = M_s^{-1} - M_s^{-1}(\Id + \gamma_\rm N_1) = -M_s^{-1} \gamma_\rm N_1.
\end{equation*}
The assertion of~(b) then follows from \eqref{eq:mapping props gamma_m N_1} together with~(a), while for~(c), we additionally use the Rellich-Kondrachov theorem for the compactness of the Sobolev embedding.
\end{proof}

\begin{rem}
The condition $s > -\frac{3}{2} + \frac{1}{q}$ in \autoref{lem:mapping props added mass} is needed in order to get the embedding $\rW^{s+1-\eps,q}(0,L) \hookrightarrow \rL^2(0,L)$, which is in turn required to make sense of the integral $\int_0^L f^2 \srd x$ in \eqref{eq:proof of ker M_s equals 0}. 
\end{rem}

We now define the fluid-structure operator in the present setting.
Let us first make precise the underlying ground space $\rX_0$.
It is given by
\begin{equation}\label{eq:ground space}
    \rX_0 \coloneqq \{(\bP u,\eta_1,\eta_2) \in \rL_\sigma^q(\cF_0) \times \rW_{\per,\rm}^{4,q}(0,L) \times \rL_\rm^q(0,L) : \bP u - \bP D_\fl \eta_2 \in \rL_\sigma^q(\cF_0)\}.
\end{equation}
With regard to the characterization of $\rL_\sigma^q(\cF_0)$ and the construction of the lifting operator $D_\fl$, this means that the kinematic coupling condition is satisfied in normal direction, i.e., $u \cdot \mre_2 = \eta_2$ on $\Gamma_u$.
The fluid-structure operator $A_\fs \colon \rD(A_\fs) \subset \rX_0 \to \rX_0$ is then defined by
\begin{equation}\label{eq:fluid-structure op}
    \begin{aligned}
        A_\fs 
        &\coloneqq \begin{pmatrix}
        A_0 & 0 & -A_0 \bP D_\fl\\
        0 & 0 & \Id\\
        M_s^{-1} \gamma_\rm N \Delta(\cdot) \cdot \nu & -M_s^{-1} P_\rm \Delta_\rs^2 & -M_\rs^{-1} P_\rm \Delta_\rs^2
    \end{pmatrix}, \twith
    \end{aligned}
\end{equation}
\begin{equation*}
    \rX_1 \coloneqq \rD(A_\fs)
    \coloneqq \{(\bP u,\eta_1,\eta_2) \in \rW_\per^{2,q}(\cF_0)^2 \cap \rL_\sigma^q(\cF_0) \times \rW_{\per,\rm}^{4,q}(0,L) \times \rW_{\per,\rm}^{4,q}(0,L): \bP u - \bP D_\fl \eta_2 \in \rD(A_0)\}.
\end{equation*}
Note in particular that $A_\fs$ has a {\em non-diagonal domain}.

By \autoref{lem:reform lin probl}, we can rewrite the linearized problem \eqref{eq:lin per multilayered interaction probl} in terms of $A_\fs$ as 
\begin{equation}\label{eq:reform of lin probl in terms of the fluid-structure op}
    \frac{\srd}{\srd t} \begin{pmatrix}
        \bP u\\ \eta_1\\ \eta_2
        \end{pmatrix} = A_\fs \begin{pmatrix}
            \bP u\\ \eta_1\\ \eta_2
        \end{pmatrix} + \begin{pmatrix}
            \bP f\\ 0\\ \overline{g}
        \end{pmatrix}, \tfor t \in (0,T), \enspace 
        \begin{pmatrix}
            \bP u\\ \eta_1\\ \eta_2
        \end{pmatrix}(0)
        = \begin{pmatrix}
            \bP u_0\\ \eta_{1,0}\\ \eta_{2,0}
        \end{pmatrix},
\end{equation}
where $\overline{g} = M_s^{-1} P_\rm g + M_s^{-1} \gamma_\rm N_2 f$.

Next, we state the main result of this section on the boundedness of the $\Hinfty$-calculus of the fluid-structure operator $A_\fs$.
To the best of our knowledge, this is the first such result in the situation of a(n) (visco-)elastic structure located at a part of the fluid boundary, and this result is of independent interest.

\begin{thm}\label{thm:bdd Hoo-calculus fluid-structure op}
Let $q \in (1,\infty)$.
Then $-A_\fs \in \Hinfty(\rX_0)$ with $\phi_{-A_\fs}^\infty < \frac{\pi}{2}$.
\end{thm}

The proof of \autoref{thm:bdd Hoo-calculus fluid-structure op} consists of two parts:
showing that the assertion of \autoref{thm:bdd Hoo-calculus fluid-structure op} is valid up to a shift, and discussing the spectral properties of $A_\fs$.
For the first part addressed in \autoref{prop:Hinfty up to shift}, the main difficulty arises from the non-diagonal domain of $A_\fs$.
To overcome this, we use a decoupling approach based on a similarity transform.
We then decompose the decoupled operator matrix with diagonal domain into a main part handled by means of theory for block operator matrices and a perturbative part.

\begin{prop}\label{prop:Hinfty up to shift}
Let $q \in (1,\infty)$.
Then there is $\mu \ge 0$ such that $-A_\fs + \mu \in \Hinfty(\rX_0)$ with $\phi_{-A_\fs + \mu}^\infty < \frac{\pi}{2}$.
\end{prop}

\begin{proof}
We first make precise the decoupling argument.
In fact, with regard to $\rD(A_\fs)$, the aim is to consider $\bP \Tilde{u} \coloneqq \bP u - \bP D_\fl \eta_2$.
Thus, the similarity transform is given by
\begin{equation*}
    S = \begin{pmatrix}
        \Id & 0 & -\bP D_\fl\\
        0 & \Id & 0\\
        0 & 0 & \Id
    \end{pmatrix} \tand S^{-1} = \begin{pmatrix}
        \Id & 0 & \bP D_\fl\\
        0 & \Id & 0\\
        0 & 0 & \Id
    \end{pmatrix}.
\end{equation*}
In this context, we also introduce the associated ``decoupled'' ground space $\rY_0$ given by
\begin{equation*}
    \rY_0 = S \rX_0 = \rL_\sigma^q(\cF_0) \times \rW_{\per,\rm}^{4,q}(0,L) \times \rL_\rm^q(0,L).
\end{equation*}
We then calculate the decoupled operator matrix $\tA_\fs \coloneqq S A_\fs S^{-1} \colon \rD(\tA_\fs) \subset \rY_0 \to \rY_0$ given by
\begin{equation}\label{eq:decoupled fluid-structure op}
    \begin{aligned}
        &\quad \tA_\fs = \\
        &\begin{pmatrix}
        A_0 - \bP D_\fl M_s^{-1} \gamma_\rm N \Delta(\cdot) \nu & \bP D_\fl M_s^{-1} P_\rm \Delta_\rs^2 & -\bP D_\fl M_s^{-1} \gamma_\rm N (\Delta(\bP D_\fl(\cdot)) \nu) + \bP D_\fl M_s^{-1} P_\rm \Delta_\rs^2\\
        0 & 0 & \Id\\
        M_s^{-1} \gamma_\rm N \Delta(\cdot) \nu & -M_s^{-1} P_\rm \Delta_\rs^2 & M_s^{-1} \gamma_\rm N (\Delta(\bP D_\fl(\cdot)) \nu) -M_s^{-1} P_\rm \Delta_\rs^2
        \end{pmatrix},
    \end{aligned}
\end{equation}
with $\rD(\tA_\fs) = \rD(A_0) \times \rD(A_\rs)$.
For the analysis, we further decompose $\tA_\fs$ as $\tA_\fs = A_\fs^1 + B_\fs $, where
\begin{equation*}
    \begin{aligned}
        A_\fs^1 
        &= \begin{pmatrix}
            A_0 & \bP D_\fl P_\rm \Delta_\rs^2 & \bP D_\fl P_\rm \Delta_\rs^2\\
            0 & 0 & \Id\\
            0 & -P_\rm \Delta_\rs^2 & -P_\rm \Delta_\rs^2
        \end{pmatrix}, \tand\\
        B_\fs
        &= \begin{pmatrix}
            - \bP D_\fl M_s^{-1} \gamma_\rm N \Delta(\cdot) \nu & \bP D_\fl M_\eta & -\bP D_\fl M_s^{-1} \gamma_\rm N (\Delta(\bP D_\fl(\cdot)) \nu) + \bP D_\fl M_\eta\\
            0 & 0 & 0\\
            M_s^{-1} \gamma_\rm N \Delta(\cdot) \nu & -M_\eta & M_s^{-1} \gamma_\rm N (\Delta(\bP D_\fl(\cdot)) \nu)-M_\eta
        \end{pmatrix}, 
    \end{aligned}
\end{equation*}
where $M_\eta \coloneqq (M_s^{-1} - \Id) P_\rm \Delta_\rs^2$.
We start by showing that $A_\fs^1 \in \Hinfty(\rY_0)$.
First, observe that 
\begin{equation*}
    \begin{aligned}
        A_\fs^1 
        &= \begin{pmatrix}
            A_0 & B\\
            0 & A_\rs
        \end{pmatrix}, \twhere B = \begin{pmatrix}
            \bP D_\fl P_\rm \Delta_\rs^2 & \bP D_\fl P_\rm \Delta_\rs^2
        \end{pmatrix} \tand \rD(A_\fs^1) = \rD(\tA_\fs) = \rD(A_0) \times \rD(A_\rs).
    \end{aligned}
\end{equation*}

From \autoref{lem:Hoo-calc Stokes op} and \autoref{lem:Hoo-calc thick layer op matrix}, we recall $-A_0 \in \Hinfty(\rL_\sigma^q(\cF_0))$ with $\phi_{-A_0}^\infty < \frac{\pi}{2}$ and $-A_\rs \in \Hinfty(\rX_0^\rs)$ with $\phi_{-A_\rs}^\infty < \frac{\pi}{2}$, respectively.
Moreover, it readily follows that $B \colon \rD(A_\rs) \to \rL_\sigma^q(\cF_0)$ due to \autoref{lem:mapping props Stokes lifting}.
By \autoref{lem:bdd Hoo-calculus block op matrix}(a), it remains to show that there is $\delta > 0$ such that $B(\rD((-A_\rs)^{1+\delta})) \subset \rD((-A_0)^\delta)$ and
\begin{equation*}
    \begin{aligned}
        \| (-A_0)^\delta B (\eta_1,\eta_2) \|_{\rL_\sigma^q(\cF_0)} 
        &\le C \cdot \| (-A_\rs)^{1+\delta} (\eta_1,\eta_2) \|_{\rX_0^\rs}, \tforall (\eta_1,\eta_2) \in \rD((-A_\rs)^{1+\delta})
    \end{aligned}
\end{equation*}
for some constant $C > 0$.
Let us observe that $(\eta_1,\eta_2) \in \rD((A_\rs)^{1+\delta})$ especially implies $\eta_1$, $\eta_2 \in \rW_\rm^{4,q}(0,L)$.
From \autoref{lem:Hoo-calc Stokes op} and \autoref{lem:cons of H00}(a), it follows that $\rD((-A_0)^\delta) \simeq \rH^{2\delta,q}(\cF_0)^2 \cap \rL_\sigma^q(\cF_0)$, so the task reduces to showing that for $f \in \rW_\rm^{4,q}(0,L)$, it holds that $\bP D_\fl P_\rm \Delta_\rs^2 f \in \rH^{2\delta,q}(\cF_0)^2 \cap \rL_\sigma^q(\cF_0)$, or, equivalently, $D_\fl P_\rm \Delta_\rs^2 f \in \rH^{2\delta,q}(\cF_0)^2$.
Now, $f \in \rW_\rm^{4,q}(0,L)$ implies $P_\rm \Delta_\rs^2 f \in \rL_\rm^q(0,L)$, so from \autoref{lem:Stokes lifting}, it follows that $D_\fl P_\rm \Delta_\rs^2 f \in \rW^{\nicefrac{1}{q},q}(\cF_0)^2 \hookrightarrow \rH^{2 \delta,q}(\cF_0)^2$ if $\delta < \nicefrac{1}{2q}$.
In summary, $B(\rD((-A_\rs)^{1+\delta})) \subset \rD((-A_0)^\delta)$ for $\delta \in (0,\frac{1}{2q})$.
Similarly, for such $\delta > 0$, and for $(\eta_1,\eta_2) \in \rD((A_\rs)^{1+\delta})$, we find 
\begin{equation*}
    \begin{aligned}
        \|  (-A_0)^\delta B (\eta_1,\eta_2) \|_{\rL_\sigma^q(\cF_0)} 
        &\le C \cdot \| D_\fl P_\rm \Delta_\rs^2 (\eta_1 + \eta_2) \|_{\rH^{2\delta,q}(\cF_0)} \le C \cdot \| P_\rm \Delta_\rs^2(\eta_1 + \eta_2) \|_{\rL_\rm^q(0,L)}\\
        &\le C \cdot \left\| -A_\rs \binom{\eta_1}{\eta_2} \right\|_{\rX_0^\rs} \le C \cdot \left\| (-A_\rs)^{1+\delta} \binom{\eta_1}{\eta_2} \right\|_{\rX_0^\rs}.
    \end{aligned}
\end{equation*}
From \autoref{lem:bdd Hoo-calculus block op matrix}(a), we then derive that $-A_\fs^1 \in \Hinfty(\rY_0)$ with $\phi_{-A_\fs^1}^\infty < \frac{\pi}{2}$.

The second part of the proof consists of showing that $A_\fs^1$ and $B_\fs$ fit into the setting of \autoref{lem:stability and pert results Hoo}(c).
First, thanks to $-A_\fs^1 \in \Hinfty(\rY_0)$, \autoref{lem:cons of H00}(a),  and the fact that the (strong) trace is well-defined for $\alpha > \frac{1}{2q}$, for $\alpha \in (\frac{1}{2q},1)$, it holds that $\rD((-A_\fs^1)^\alpha) = \rD((-A_0)^\alpha) \times \rD((-A_\rs)^\alpha)$, so
\begin{equation*}
    \begin{aligned}
        \rD((-A_\fs^1)^\alpha) 
        &= \left\{ (u,\eta_1,\eta_2) \in \rH_\per^{2 \alpha,q}(\cF_0)^2 \cap \rL_\sigma^q(\cF_0) \times \rW_{\per,\rm}^{4,q}(0,L) \times \rH_{\per,\rm}^{4 \alpha,q}(0,L) : u = 0 \ton \Gamma_u \cup \Gamma_b\right\}.
    \end{aligned}
\end{equation*}
For such $\alpha$, consider $(\bP \Tilde{u},\eta_1,\eta_2) \in \rD((-A_\fs^1)^\alpha)$.
The invertibility of $M_s$ on $\rL_\rm^q(0,L)$, the continuity of $\gamma_\rm \colon \rW_{\per,\rm}^{\nicefrac{1}{q}+\eps,q}(\cF_0) \to \rL_\rm^q(0,L)$, and the mapping properties of $N$ as determined in~\eqref{eq:mapping props Neumann op} imply that
\begin{equation}\label{eq:bddness of lowest entry}
    \begin{aligned}
        \| M_s^{-1} \gamma_\rm N (\Delta \bP \Tilde{u}) \cdot \nu \|_{\rL_\rm^q(0,L)}
        &\le C \cdot \| \gamma_\rm N (\Delta \bP \Tilde{u}) \cdot \nu \|_{\rL_\rm^q(0,L)}\le C \cdot \| N (\Delta \bP \Tilde{u}) \cdot \nu \|_{\rW_{\per,\rm}^{\nicefrac{1}{q}+\eps,q}(\cF_0)}\\
        &\le C \cdot \| (\Delta \bP \Tilde{u}) \cdot \nu \|_{\rW_{\per,\rm}^{-1 + \eps,q}(\Gamma_u \cup \Gamma_b)} 
        \le C \cdot \| \bP \Tilde{u} \|_{\rH_\per^{2 \alpha,q}(\cF_0) \cap \rL_\sigma^q(\cF_0)}\\
        &\le C \cdot \| (-A_0)^\alpha \bP \Tilde{u} \|_{\rL_\sigma^q(\cF_0)}
    \end{aligned}
\end{equation}
for $\alpha \in (\frac{1}{2}+\frac{1}{2q},1)$.
Likewise, invoking that $\bP D_\fl \in \cL(\rL_\rm^q(0,L),\rL_\sigma^q(\cF_0))$, for $\alpha \in (\frac{1}{2}+\frac{1}{2q},1)$, we find that
\begin{equation*}
    \| \bP D_\fl M_s^{-1} \gamma_\rm N (\Delta \bP \Tilde{u}) \cdot \nu \|_{\rL_\rm^q(0,L)} \le C \cdot \| (-A_0)^\alpha \bP \Tilde{u} \|_{\rL_\sigma^q(\cF_0)}.
\end{equation*}
Next, using \autoref{lem:mapping props added mass}(b)  with $s > \max\{-1,-\frac{3}{2}+\frac{1}{q}\}$, and combining \autoref{lem:cons of H00}(a) with \autoref{lem:Hoo-calc thick layer op matrix}, for $\alpha \in (\max\{\frac{3}{4},\frac{5}{8}+\frac{1}{4q}\},1)$, we get
\begin{equation*}
    \begin{aligned}
        \| (M_s^{-1} - \Id) P_\rm \Delta_\rs^2 \eta_i \|_{\rL_\rm^q(0,L)}
        &\le C \cdot \| P_\rm \Delta_\rs^2 \eta_i \|_{\rW_{\per,\rm}^{s,q}(0,L)} \le C \cdot \| \eta_i \|_{\rW_{\per,\rm}^{4+s,q}(0,L)}\\
        &\le C \left(\left\| (-A_\rs)^\alpha \binom{\eta_1}{\eta_2} \right\|_{\rX_0^\rs} + \left\| \binom{\eta_1}{\eta_2} \right\|_{\rX_0^\rs}\right).
    \end{aligned}
\end{equation*}
Again, the corresponding term with $\bP D_\fl$ can be estimated in the same way by the above observation that $\bP D_\fl \in \cL(\rL_\rm^q(0,L),\rL_\sigma^q(\cF_0))$, which is due to \autoref{lem:mapping props Stokes lifting}.

Finally, we tackle $M_s^{-1} \gamma_\rm N (\Delta(\bP D_\fl(\cdot)) \nu)$ and $\bP D_\fl M_s^{-1} \gamma_\rm N (\Delta(\bP D_\fl(\cdot)) \nu)$.
In an analogous way as in \eqref{eq:bddness of lowest entry}, additionally using \autoref{lem:mapping props Stokes lifting} as well as \autoref{lem:cons of H00}(a) together with \autoref{lem:Hoo-calc thick layer op matrix} for the characterization of $\rD((A_s)^\alpha)$, for $\alpha \in (\frac{1}{4},1)$, and for $\eps > 0$ sufficiently small, we get
\begin{equation}\label{eq:est of last term}
    \begin{aligned}
        \| M_s^{-1} \gamma_\rm N (\Delta(\bP D_\fl(\eta_2)) \nu) \|_{\rL_\rm^q(0,L)}
        &\le C \cdot \| \bP D_\fl(\eta_2) \|_{\rW_\per^{1+\frac{1}{q}+\eps,q}(\cF_0)} \le C \cdot \| \eta_2 \|_{\rW_{\per,\rm}^{1+\eps,q}(0,L)}\\
        &\le C \cdot \| \eta_2 \|_{\rH_{\per,\rm}^{4 \alpha,q}(0,L)} \le C \cdot \left\| (-A_\rs)^\alpha \binom{\eta_1}{\eta_2} \right\|_{\rX_0^\rs}.
    \end{aligned}
\end{equation}
The corresponding estimate of the term $\bP D_\fl M_s^{-1} \gamma_\rm N (\Delta(\bP D_\fl(\cdot)) \nu)$ follows from there upon invoking that $\bP D_\fl \in \cL(\rL_\rm^q(0,L),\rL_\sigma^q(\cF_0))$.

In total, we argue that there exist constants $a$, $b > 0$ and $\alpha \in (\max\{\frac{3}{4},\frac{5}{8} + \frac{1}{4q},\frac{1}{2}+\frac{1}{2q}\},1)$ such that
\begin{equation}\label{eq:rel bddness wrt frac power}
    \left\| B_\fs \begin{pmatrix}
        \bP \Tilde{u}\\ \eta_1\\ \eta_2
    \end{pmatrix} \right\|_{\rY_0} \le a \cdot \left\| \begin{pmatrix}
        \bP \Tilde{u}\\ \eta_1\\ \eta_2
    \end{pmatrix} \right\|_{\rY_0} + b \cdot \left\| (-A_\fs^1)^\alpha \begin{pmatrix}
        \bP \Tilde{u}\\ \eta_1\\ \eta_2
    \end{pmatrix} \right\|_{\rY_0}.
\end{equation}
On the other hand, $-A_\fs^1 \in \Hinfty(\rY_0)$ especially yields that $-A_\fs^1$ is sectorial on $\rY_0$, and \eqref{eq:rel bddness wrt frac power} shows that~$B_\fs$ is relatively bounded with respect to a (lower) fractional power of $A_\fs^1$.
Perturbation theory for sectorial operators, see, e.g., \cite[Eq.~(3.17) and Cor.~3.1.6]{PS:16}, then yields that there exists $\mu \ge 0$ such that $-\tA_\fs + \mu = -(A_\fs^1 + B_\fs) + \mu$ is sectorial on $\rY_0$ with $\phi_{-\tA_\fs + \mu} < \frac{\pi}{2}$.
For sufficiently large $\mu \ge 0$, the invertibility of $-\tA_\fs + \mu$ is valid.
Therefore, \autoref{lem:stability and pert results Hoo}(c) implies that $-\tA_\fs + \mu \in \Hinfty(\rY_0)$ with $\phi_{-A_\fs + \mu}^\infty < \frac{\pi}{2}$.
The assertion of the proposition then follows from the decoupling argument and \autoref{lem:stability and pert results Hoo}(b), yielding that $-A_\fs + \mu \in \Hinfty(\rX_0)$ with $\phi_{-A_\fs + \mu}^\infty < \frac{\pi}{2}$.
\end{proof}

We now address the spectral theory of the fluid-structure operator.
Note that the Kelvin-Voigt-type damping and the resulting functional setting cause that the embedding $\rD(A_\fs) \hookrightarrow \rX_0$ is not necessarily compact, implying that the spectrum does not necessarily consist of eigenvalues.
Therefore, a careful spectral analysis inspired by \cite[Sec.~5.5]{BMR:26} is required.

\begin{prop}\label{prop:spectral props fluid-structure op}
Let $q \in (1,\infty)$.
Then there exists $\mu < 0$ such that $\sigma(A_\fs) \subset \{\lambda \in \C : \Rep \lambda \le \mu\}$, and it holds that $s(A_\fs) = \sup\{\Rep \lambda : \lambda \in \sigma(A_\fs)\} < 0$.
\end{prop}

\begin{proof}
The following proof consists of two parts.
First, we discuss the point spectrum of $A_\fs$, and in a second step, we establish control of the essential spectrum by invoking the decoupling approach used in \autoref{prop:Hinfty up to shift} to show that the operator admits a bounded $\Hinfty$-calculus.

We observe that the point spectrum $\sigma_\rp(A_\fs)$ is independent of $q$, so we may consider the eigenvalue problem $(\lambda - A_\fs)(\bP u,\eta_1,\eta_2) = 0$ in the case $q = 2$.
As in the reformulation of the linear problem in terms of the fluid-structure operator in \eqref{eq:reform of lin probl in terms of the fluid-structure op}, the eigenvalue problem is equivalent with the problem
\begin{equation}\label{eq:eigenvalue probl per setting}
    \left\{
    \begin{aligned}
        \lambda u - \mdiv \sigma_\rf(u,\pi)
        &= 0, &&\tin \cF_0,\\
        \mdiv u
        &= 0, &&\tin \cF_0,\\
        \lambda \eta_1 - \eta_2
        &= 0, &&\tin (0,L),\\
        \lambda \eta_2 + P_\rm \Delta_s^2 \eta_1 + P_\rm \Delta_s^2 \eta_2
        &= \gamma_\rm \pi, &&\tin (0,L),\\
        u
        &= P_\rm(\eta_2) \mre_2, &&\ton \Gamma_u,\\
        u
        &= 0, &&\ton \Gamma_b,\\
        (u,\pi)(0,y) = (u,\pi)(L,y), \enspace (\eta_1,\eta_2)(0)
        &= (\eta_1,\eta_2)(L), &&\tfor y \in (-1,0),
    \end{aligned}
    \right.
\end{equation}
where, as above, we observe that $-P_\rm \left(\mre_2 \cdot \left.\sigma_\rf(u,\pi)\right|_{\Gamma_u} \mre_2\right) = \gamma_\rm \pi$ results from the kinematic coupling condition and $\mdiv u = 0$.
Testing \eqref{eq:eigenvalue probl per setting}$_1$ by $\overline{u}$ and \eqref{eq:eigenvalue probl per setting}$_4$ by $\overline{\eta}_2 = \overline{\lambda} \overline{\eta}_1$, invoking $\int_0^L f P_\rm g \srd x = \int_0^L (P_\rm f)  g \srd x$ to cancel the terms involving $\pi$, and adding both equations, we infer that
\begin{equation*}
    \lambda \int_{\cF_0} |u|^2 \srd (x,y) + 2 \int_{\cF_0} |\D(u)|^2 \srd (x,y) + \lambda \int_0^L |\eta_2|^2 \srd x + \overline{\lambda} \int_0^L |\Delta_s \eta_1|^2 \srd x + \int_0^L |\Delta_s \eta_2|^2 \srd x = 0.
\end{equation*}
It follows that $\Rep \lambda \le 0$.
Moreover, considering $\lambda = 0$, we find that $\D(u) = 0$ as well as $\Delta_s \eta_2 = 0$.
With regard to the boundary and coupling conditions as well as the fact that $\eta_1$ and $\eta_2$ have spatial average zero, this yields that $u = 0$, $\eta_2 = 0$ and $\eta_1 = 0$.
In total, we have verified that $\sigma_\rp(A_\fs) \subset \{\lambda \in \C : \Rep \lambda < 0\}$.

The second step is now to analyze the essential spectrum, where we employ the decoupling argument from the proof of \autoref{prop:Hinfty up to shift}.
More precisely, from \eqref{eq:decoupled fluid-structure op}, we recall the decoupled operator $\tA_\fs$ with diagonal domain $\rD(\tA_\fs) = \rD(A_0) \times \rD(A_s)$ that emerges from the similarity transform given by $S$ and~$S^{-1}$ to decouple.
In particular, the spectra of $A_\fs$ and $\tA_\fs$ coincide, so we have
\begin{equation*}
    \sigma(A_\fs) = \sigma(\tA_\fs) = \sigma_\ess(\tA_\fs) \cup \sigma_\rd(\tA_\fs),
\end{equation*}
where we note that $\sigma_\rd(\tA_\fs) \subset \sigma_\rp(\tA_\fs) = \sigma_\rp(A_\fs)$, and the latter has been determined in the first part of the proof.
We also invoke the decomposition of $\tA_\fs = A_\fs^1 + B_\fs$, where
\begin{equation*}
    \begin{aligned}
        A_\fs^1
        &= \begin{pmatrix}
            A_0 & -\bP D_\fl P_\rm \Delta_\rs^2 & -\bP D_\fl P_\rm \Delta_\rs^2\\
            0 & 0 & \Id\\
            0 & -P_\rm \Delta_\rs^2 & -P_\rm \Delta_\rs^2
        \end{pmatrix} = \begin{pmatrix}
            A_0 & B\\
            0 & A_\rs
        \end{pmatrix}, \enspace B = \begin{pmatrix}
            -\bP D_\fl P_\rm \Delta_\rs^2 & -\bP D_\fl P_\rm \Delta_\rs^2
        \end{pmatrix}, \tand\\
        B_\fs
        &= \begin{pmatrix}
            - \bP D_\fl M_s^{-1} \gamma_\rm N \Delta(\cdot) \nu & \bP D_\fl M_\eta & -\bP D_\fl M_s^{-1} \gamma_\rm N (\Delta(\bP D_\fl(\cdot)) \nu) + \bP D_\fl M_\eta\\
            0 & 0 & 0\\
            M_s^{-1} \gamma_\rm N \Delta(\cdot) \nu & -M_\eta & M_s^{-1} \gamma_\rm N (\Delta(\bP D_\fl(\cdot)) \nu)-M_\eta
        \end{pmatrix}, 
    \end{aligned}
\end{equation*}
with $M_\eta = (M_s^{-1} - \Id) P_\rm \Delta_\rs^2$.
Employing the representation 
\begin{equation*}
    \lambda - A_\fs^1 = \begin{pmatrix}
        \lambda - A_0 & -B\\
        0 & \lambda - A_\rs
    \end{pmatrix} = \begin{pmatrix}
        \lambda - A_0 & 0\\
        0 & \lambda - A_\rs
    \end{pmatrix} \cdot \begin{pmatrix}
        \Id & -R(\lambda,A_0)B\\
        0 & \Id 
    \end{pmatrix},
\end{equation*}
for $\lambda \in \rho(\diag(A_0,A_\rs))$, we deduce that $\lambda \in \rho(A_\fs^1)$, implying that $\sigma(A_\fs^1) \subset \{\lambda \in \C : \Rep \lambda < 0\}$ thanks to \autoref{lem:Hoo-calc Stokes op} and \autoref{lem:Hoo-calc thick layer op matrix}, and then also $s(A_\fs^1) < 0$.

Now, the goal is to verify that $B_\fs$ is a relatively compact perturbation of $A_\fs^1$.
First, we recall from \autoref{lem:mapping props Stokes lifting} that $D_\fl \in \cL(\rW_{\per,\rm}^{s,q}(0,L),\rW_\per^{s+\nicefrac{1}{q},q}(\cF_0)^2)$ for $s \in [-\frac{1}{q},1+\frac{1}{q}]$.
In particular, $\bP D_\fl$ is compact as an operator from $\rL_\rm^q(0,L)$ to $\rL_\sigma^q(\cF_0)$.
In view of the shape of $B_\fs$, it is thus sufficient to establish the relative compactness of the operators $M_s^{-1} \gamma_\rm N \Delta(\cdot) \nu$, $(M_s^{-1} - \Id) P_\rm \Delta_\rs^2$, and $M_s^{-1} \gamma_\rm N (\Delta(\bP D_\fl(\cdot)) \nu)$.

Concerning the first operator, we consider $u \in \rD(A_0)$ and note that the embedding $\rD(A_0) \hookrightarrow \rD(A_0^\alpha)$ is compact for all $\alpha \in [0,1)$.
On the other hand, from \eqref{eq:bddness of lowest entry}, we recall that $M_s^{-1} \gamma_\rm N \Delta(\cdot) \nu$ is continuous from $\rD(A_0^\alpha)$ to $\rL_\rm^q(\omega)$ for $\alpha \in (\frac{1}{2}+\frac{1}{2q},1)$.
This reveals the relative compactness of $M_s^{-1} \gamma_\rm N \Delta(\cdot) \nu$.
With regard to $(M_s^{-1} - \Id) P_\rm \Delta_\rs^2$, we first note that $P_\rm \Delta_\rs^2 \in \cL(\rW_\rm^{4,q}(0,L),\rL_\rm^q(0,L))$, and we exploit the compactness of $M_s^{-1} - \Id$, see \autoref{lem:mapping props added mass}(c), to conclude that $(M_s^{-1} - \Id) P_\rm \Delta_\rs^2$ is relatively compact.
For $M_s^{-1} \gamma_\rm N (\Delta(\bP D_\fl(\cdot)) \nu)$, we invoke \eqref{eq:est of last term} to argue that this operator is continuous from $\rH_{\per,\rm}^{4 \alpha,q}(0,L)$ to $\rL_\rm^q(0,L)$ for $\alpha \in (\frac{1}{4},1)$.
By the compactness of the embedding $\rW_{\per,\rm}^{4,q}(0,L) \hookrightarrow \rH_{\per,\rm}^{4 \alpha,q}(0,L)$ for such~$\alpha$, and in view of the domain $\rD(A_\fs)$ from \eqref{eq:fluid-structure op}, we conclude the relative compactness of this term.

Putting the previous arguments together, we find that $B_\fs$ is a relatively compact perturbation of $A_\fs^1$, so \cite[Thm.~5.35]{Kat:95} yields that the essential spectra of $A_\fs^1$ and $\tA_\fs = A_\fs^1 + B_\fs$ coincide.
Together with the above considerations on the essential spectrum of $A_\fs^1$, we argue that
\begin{equation*}
    \sigma_\ess(\tA_\fs) = \sigma_\ess(A_\fs^1) \subset \sigma(A_\fs^1) \subset  \{\lambda \in \C : \Rep \lambda < 0\}.
\end{equation*}
Since on the other hand, we have $\sigma_\rd(\tA_\fs) \subset \sigma_\rp(\tA_\fs) = \sigma_\rp(A_\fs) \subset \{\lambda \in \C : \Rep \lambda < 0\}$, we finally obtain
\begin{equation*}
    \sigma(A_\fs) = \sigma(\tA_\fs) = \sigma_\ess(\tA_\fs) \cup \sigma_\rd(\tA_\fs) \subset \{\lambda \in \C : \Rep \lambda < 0\}. \qedhere
\end{equation*}
\end{proof}

\begin{proof}[Proof of~\autoref{thm:bdd Hoo-calculus fluid-structure op}]
The assertion of \autoref{thm:bdd Hoo-calculus fluid-structure op} follows by combining \autoref{prop:Hinfty up to shift} on the boundedness of the $\Hinfty$-calculus of $-A_\fs$ up to a shift, the negativity of the spectral bound $s(A_\fs) < 0$ as examined in \autoref{prop:spectral props fluid-structure op}, and the well-known fact that the shift in \autoref{prop:Hinfty up to shift} can be taken to be equal to the spectral bound.
\end{proof}

In the proposition below, we provide a characterization of the real and complex interpolation spaces.
The arguments to obtain these characterizations via the decoupling arguments parallel the ones employed in the proof of \cite[Cor.~8.5]{BBH:26}.
Also, let us observe that the boundedness of the $\Hinfty$-calculus of $A_\fs$, which is even valid without shift thanks to \autoref{prop:spectral props fluid-structure op}, especially yields that the complex interpolation spaces coincide with the domains of the fractional powers of the operator $A_\fs$.
In order to shorten notation, we introduce the map $\cT \colon \rL^q(0,L) \to \rL^q(\Gamma_u \cup \Gamma_b)$ defined by
\begin{equation}\label{eq:not cT}
    (\cT g)(x) = \left\{
    \begin{aligned}
        (P_\rm g(x)) \mre_2, &\tif (x,0) \in \Gamma_u,\\
        0, &\tif (x,-1) \in \Gamma_b.
    \end{aligned}
    \right.
\end{equation}

\begin{prop}\label{prop:char of the interpol spaces}
Let $p$, $q \in (1,\infty)$ and $\theta \in (0,1) \setminus \{\frac{1}{2q}\}$, and denote by $(\cdot,\cdot)_{\theta,p}$ and $[\cdot,\cdot]_\theta$ the real and complex interpolation functor, respectively.
Then it holds that
\begin{equation*}
    \begin{aligned}
        (\rX_0,\rX_1)_{\theta,p} 
        &= \left\{(u,\eta_1,\eta_2) \in \rB_{qp,\per}^{2 \theta}(\cF_0)^2 \cap \rL_\sigma^q(\cF_0) \times \rW_{\per,\rm}^{4,q}(0,L) \times \rB_{qp,\per,\rm}^{4 \theta}(0,L) : (u,\eta_1,\eta_2) \text{ satisfy } \eqref{eq:compat conds}\right\},\\
        [\rX_0,\rX_1]_\theta 
        &= \left\{(u,\eta_1,\eta_2) \in \rH_\per^{2\theta,q}(\cF_0)^2 \cap \rL_\sigma^q(\cF_0) \times \rW_{\per,\rm}^{4,q}(0,L) \times \rH_{\per,\rm}^{4 \theta,q}(0,L) : (u,\eta_1,\eta_2) \text{ satisfy } \eqref{eq:compat conds}\right\},
    \end{aligned}
\end{equation*}
where
\begin{equation}\label{eq:compat conds}
    \left\{
    \begin{aligned}
        u \cdot \nu = \cT \eta_2 \cdot \nu, &\ton \Gamma_u \cup \Gamma_b, &&\tif \theta < \frac{1}{2q},\\
        u = \cT \eta_2, &\ton \Gamma_u \cup \Gamma_b, &&\tif \theta > \frac{1}{2q}.
    \end{aligned}
    \right.
\end{equation}
\end{prop}

\section{The linear stochastic problem}\label{sec:lin stoch probl}

The purpose of this section is twofold.
We first introduce the required stochastic preliminaries, including the concept of {\em stochastic maximal regularity} and its relation with the boundedness of the $\Hinfty$-calculus.
In a second step, we employ these concepts in the present setting to get a result on the linear stochastic problem, capitalizing on \autoref{thm:bdd Hoo-calculus fluid-structure op} on the bounded $\Hinfty$-calculus of the fluid-structure operator.

For a separable Hilbert space $\cH$, a bounded linear operator $W_\cH \colon \rL^2((0,\infty);\cH) \to \rL^2(\Omega)$ is an $\frF$-cylindrical Wiener process or a cylindrical Brownian motion if for all $f$, $g \in \cH$ and $0 \le t \le t'$,
\begin{enumerate}[(a)]
    \item the random variable $W_\cH(t) f \coloneqq W_\cH(\bfone_{[0,t]} \otimes f)$ is centered Gaussian and $\frF_t$-measurable,
    \item it holds that $\E[W_\cH(t')f \cdot W_\cH(t)g] = t \langle f,g \rangle_{\cH}$, and
    \item the random variable $W_\cH(t')f - W_\cH(t)f$ is independent of $\frF_t$.
\end{enumerate}
 We then recall the representation of the $\frF$-cylindrical Wiener process $W_\cH(t) \colon \cH \to \rL^2(\Omega)$ from \eqref{eq:repr cylindr Wiener proc}.
The definition of the stochastic integral with respect to $W_\cH(t)$ has been discussed by van Neerven, Veraar, and Weis \cite{vNVW:12a}.
In fact, they considered the stochastic integral for functions with values in the space of $\gamma$-radonifying operators $\gamma(\cH,\rX_{\nicefrac{1}{2}})$.
Concerning $\gamma(\cH,\rX_{\nicefrac{1}{2}})$, we first invoke \autoref{prop:char of the interpol spaces} to find that
\begin{equation}\label{eq:X_1/2}
     \rX_{\nicefrac{1}{2}} = \left\{(u,\eta_1,\eta_2) \in \rW_\per^{1,q}(\cF_0)^2 \cap \rL_\sigma^q(\cF_0) \times \rW_{\per,\rm}^{4,q}(0,L) \times \rW_{\per,\rm}^{2,q}(0,L) : u = \cT \eta_2 \ton \Gamma_u \cup \Gamma_b\right\},
\end{equation}
where $\cT$ is as in \eqref{eq:not cT}.

For the set of $\gamma$-radonifying operators, if $(S,\Sigma,\mu)$ is a measure space, then $\gamma(\cH;\rL^q(S)) \simeq \rL^q(S;\cH)$, see \cite[Prop.~9.3.2]{HvNVW:17}, and this isomorphy extends to the situation of arbitrary Banach spaces of finite co-type, see \cite[Thm.~9.3.6]{HvNVW:17}.
It can be shown that $\rX_{\nicefrac{1}{2}}$ has finite co-type provided $q \ge 2$, see \cite[Sec.~7.1]{HvNVW:17} for more details.
Thus, for $q \ge 2$, it follows that
\begin{equation}\label{eq:gamma H X_1/2}
    \begin{aligned}
        \gamma(\cH;\rX_{\nicefrac{1}{2}}) \simeq \bigl\{(&u,\eta_1,\eta_2) \in \rW_\per^{1,q}(\cF_0;\cH)^2 \cap \rL_\sigma^q(\cF_0;\cH) \times \rW_{\per,\rm}^{4,q}(0,L;\cH) \times \rW_{\per,\rm}^{2,q}(0,L;\cH) :\\
        &u = \cT \eta_2 \ton \Gamma_u \cup \Gamma_b\bigr\} \eqqcolon \rX_{\nicefrac{1}{2}}(\cH),
    \end{aligned}
\end{equation}
where similarly as in \eqref{eq:L_sigma^q}, the space $\rL_\sigma^q(\cF_0;\cH)$ is defined by
\begin{equation*}
    \rL_\sigma^q(\cF_0;\cH) \coloneqq \overline{\{f \in \rC_\rc^\infty(\cF_0;\cH)^2 : \mdiv f = 0 \tin \cF_0\}}^{\| \cdot \|_{\rL^q(\cF_0;\cH)^2}}.
\end{equation*}
 Let us now address the linear stochastic problem associated with the fluid-structure operator.
More precisely, with the fluid-structure operator $A_\fs$ as in \eqref{eq:fluid-structure op},  and for the principal variable $Z = (U,H_1,H_2)$, the linear stochastic evolution equation is given by
\begin{equation}\label{eq:lin stoch evol eq}
    \rd Z(t) - A_\fs Z(t) \srd t = h(t) \srd W_\cH(t), \tfor t \in (0,T), \enspace Z(0) = Z_0,
\end{equation}
for $h \in \rX_{\nicefrac{1}{2}}(\cH)$, where the identification of the latter space with $\gamma(\cH;\rX_{\nicefrac{1}{2}})$ has been discussed in \eqref{eq:gamma H X_1/2}.
A strong solution to \eqref{eq:lin stoch evol eq} is then given by
\begin{equation}\label{eq:stoch conv}
    Z(t) \coloneqq \mre^{t A_\fs} Z_0 + \int_0^t \mre^{(t-s) A_\fs} h(s) \srd W_\cH(s).
\end{equation}

We discuss the well-posedness of \eqref{eq:stoch conv} and the regularity properties of the emerging $Z$ below.

\begin{prop}\label{prop:sol lin stoch probl}
Consider $T \in (0,\infty)$, and let $p$, $q \ge 2$ such that $q > 2$ if $p \neq 2$.
Then for all strongly $\frF_0$-measurable $Z_0 = (U_0,H_1^0,H_2^0) \colon \Omega \to (\rX_0,\rX_1)_{1-\nicefrac{1}{p},p}$, where the latter  space has been characterized in \autoref{prop:char of the interpol spaces}, and $\frF$-adapted $h \in \rL^p(\Omega;\rL^p(0,T;\rX_{\nicefrac{1}{2}}(\cH)))$, the stochastic convolution \eqref{eq:stoch conv} is well-defined, $Z$ given there is $\frF$-adapted, and constitutes the unique solution to \eqref{eq:lin stoch evol eq} that satisfies pathwise
\begin{equation*}
    Z \in \rH^{\theta,p}(0,T;\rX_{1-\theta}) \cap \rC([0,T];(\rX_0,\rX_1)_{1-\nicefrac{1}{p},p}) \tforall \theta \in [0,\nicefrac{1}{2}),
\end{equation*}
where the space $\rX_{1-\theta}$ has been made precise in \autoref{prop:char of the interpol spaces}.
\end{prop}

\begin{proof}
In the present case of $\rX_0$ being isomorphic to a closed subspace of $\rL^q$, for $q \ge 2$, it has been shown that $-A_\fs \in \Hinfty(\rX_0)$ as revealed in \autoref{thm:bdd Hoo-calculus fluid-structure op} is sufficient for the stochastic maximal regularity of~$-A_\fs$, see \cite{vNVW:12a, vNVW:12b} or also \cite[Thm.~3.13]{AV:25}.
Together with the isomorphism $\rX_{\nicefrac{1}{2}}(\cH) \simeq \gamma(\cH,\rX_{\nicefrac{1}{2}})$ discussed in \eqref{eq:gamma H X_1/2}, this yields the well-posedness of~\eqref{eq:stoch conv} and the asserted regularity.
\end{proof}

We now introduce compatibility conditions.
First, for $u_0$, $\eta_1^0$, and $\eta_2^0$ sufficiently regular, assume that
\begin{equation}\label{eq:comp conds init data}
    \mdiv u_0 = 0 \tin \cF_0, \, \int_0^L \eta_1^0 \srd x = \int_0^L \eta_2^0 \srd x = 0, \, u_0(x,-1) = 0, \, u_0(x,0) = \eta_2^0(x) \mre_2 \tforall x \in (0,L).
\end{equation}
Next, for $f$ and $g$ sufficiently regular, and with $\overline{g} = M_s^{-1} P_\rm g + M_s^{-1} \gamma_\rm N_2 f$ as in \eqref{eq:reform of lin probl in terms of the fluid-structure op}, suppose that
\begin{equation}\label{eq:comp conds RHS}
    \bP f(x,-1) = 0 \tand \bP f(x,0) = \overline{g}(x) \tforall x \in (0,L). 
\end{equation}
 Below, we discuss the solvability of the actual linear stochastic problem under consideration below.
\begin{cor}\label{cor:sol lin stoch probl}
Let $T \in (0,\infty)$ as well as $p$, $q \ge 2$ such that $q > 2$ if $p \neq 2$.
Consider the linear stochastic problem
\begin{equation}\label{eq:lin stoch PDE}
    \left\{
    \begin{aligned}
        \rd U - \mdiv \sigma_\rf(U,\Pi) \srd t
        &= f \srd W_\cH, \enspace \mdiv U = 0, &&\tin (0,T) \times \cF_0,\\
        \rd H_1 
        &= H_2 \srd t, &&\tin (0,T) \times (0,L),\\
        \rd H_2 + \bigl(P_\rm \Delta_s^2 H_1 + P_\rm \Delta_s^2 H_2\bigr) \srd t
        &= -P_\rm\bigl(\mre_2 \cdot \sigma_\rf(U,\Pi) \mre_2\bigr) \srd t\\
        &\quad + g \srd W_\cH, &&\tin (0,T) \times (0,L),\\
        U(t,x,0) 
        &= H_2(t,x) \mre_2, &&\tfor t \in (0,T), \enspace x \in (0,L),\\
        (U,\Pi)(t,0,y) 
        &= (U,\Pi)(t,L,y), &&\tfor t \in (0,T), \enspace y \in (-1,0),\\
        (H_1,H_2)(t,0) 
        &= (H_1,H_2)(t,L), &&\tfor t \in (0,T),\\
        U(t,x,-1) 
        &= 0, &&\tfor t \in (0,T), \enspace x \in (0,L),\\
        U(0,\cdot)
        &= U_0(\cdot), &&\tin \cF_0,\\
        H_1(0,\cdot) &= H_1^0(\cdot), \enspace H_2(0,\cdot) = H_2^0(\cdot), &&\tin (0,L).
    \end{aligned}
    \right.
\end{equation}
Assume that
\begin{enumerate}[(a)]
    \item $Z_0 = (U_0,H_1^0,H_2^0)$ is strongly $\frF_0$-measurable such that $U_0 \in \rB_{qp,\per}^{2-\frac{2}{p}}(\cF_0)^2$, $H_1^0 \in \rW_\per^{4,q}(0,L)$, and $H_2^0 \in \rB_{qp,\per}^{4-\frac{4}{p}}(0,L)$ satisfy \eqref{eq:comp conds init data},
    \item $h = (f,0,g)$ is $\frF$-adapted, and the forcing terms $f \in \rL^p(\Omega;\rL^p(0,T;\rW_\per^{1+\frac{1}{q},q}(\cF_0;\cH)^2))$ as well as $g \in \rL^p(\Omega;\rL^p(0,T;\rW_\per^{2,q}(0,L;\cH)))$ fulfill \eqref{eq:comp conds RHS}.
\end{enumerate}
Then there exists a unique solution $Z = (\bP U,H_1,H_2)$ to \eqref{eq:lin stoch PDE}, and $Z$ is $\frF$-adapted and satisfies pathwise
\begin{equation*}
    Z \in \rH^{\theta,p}(0,T;\rX_{1-\theta}) \cap \rC([0,T];(\rX_0,\rX_1)_{1-\nicefrac{1}{p},p}) \tforall \theta \in [0,\nicefrac{1}{2}).
\end{equation*}
\end{cor}

\begin{proof}
The task here is to show that the analysis of \eqref{eq:lin stoch PDE} reduces to the investigation of the linear stochastic problem \eqref{eq:lin stoch evol eq} given in terms of the fluid-structure operator $A_\fs$.
Note that similarly as in~\eqref{eq:reform of lin probl in terms of the fluid-structure op}, it can be shown that \eqref{eq:lin stoch PDE} is equivalent with 
\begin{equation*}
    \rd \tZ(t) - A_\fs \tZ(t) \srd t = \th(t) \srd W_\cH(t), \tfor t \in (0,T), \enspace \tZ(0) = \tZ_0,
\end{equation*}
where $\tZ = (\bP U,Z_1,Z_2)$, $\th = (\bP f,0,\overline{g})$, with $\overline{g} = M_s^{-1} P_\rm g + M_s^{-1} \gamma_\rm N_2 f$, and $\tZ_0 = (\bP U_0,H_1^0,H_2^0)$.
With regard to the characterization of the space $(\rX_0,\rX_1)_{1-\nicefrac{1}{p},p}$ in \autoref{prop:char of the interpol spaces}, the assumptions in~(a) ensure that, indeed, we have that $\tZ_0$ is strongly $\frF_0$-measurable and satisfies $\tZ_0 \colon \Omega \to (\rX_0,\rX_1)_{1-\nicefrac{1}{p},p}$.
In other words, the initial data lie in the scope of \autoref{prop:sol lin stoch probl}.

It remains to verify that the right-hand side also fits into the framework provided by \autoref{prop:sol lin stoch probl}.
To this end, we invoke the space $\rX_{\nicefrac{1}{2}}(\cH)$ introduced in \eqref{eq:gamma H X_1/2}.
Note that the compatibility conditions included in this space are satisfied with regard to the assumptions made in~(b), and the task reduces to verifying the required regularities of $\bP f$ and $\overline{g}$.
For the former, this is an immediate consequence of the assumption $f \in \rL^p(\Omega;\rL^p(0,T;\rW_\per^{1+\frac{1}{q},q}(\cF_0;\cH)^2))$.
For the latter, we need to show that
\begin{equation*}
    \overline{g} = M_s^{-1} P_\rm g + M_s^{-1} \gamma_\rm N_2 f \in \rL^p(\Omega;\rL^p(0,T;\rW_{\per,\rm}^{2,q}(0,L;\cH))).
\end{equation*}
For this purpose, we first deduce that by construction, it follows that $P_\rm g \in \rW_{\per,\rm}^{2,q}(0,L;\cH)$.
Mimicking the proof of \autoref{lem:mapping props added mass}, where we additionally exploit classical elliptic regularity, and observing that the $\cH$-valued case can be handled in the same way, we note that $M_s^{-1} P_\rm g \in \rW_{\per,\rm}^{2,q}(0,L;\cH)$, establishing the required regularity for the first addend of $\overline{g}$. 
For the second addend $M_s^{-1} \gamma_\rm N_2 f$, we first note that by elliptic regularity and the same argument as above for the $\cH$-valued case, we have $N_2 f \in \rW_{\per,\rm}^{2+\frac{1}{q},q}(\cF_0;\cH)$.
Thus mapping properties of the trace yield that $\gamma_\rm N_2 f \in \rW_{\per,\rm}^{2,q}(0,L;\cH)$.
The above observation on the inverse of the added mass operator $M_s^{-1}$, we find that $M_s^{-1} \gamma_\rm N_2 f \in \rW_{\per,\rm}^{2,q}(0,L;\cH)$.
Concatenating the previous two arguments, we conclude the assertion of the corollary as a consequence of \autoref{prop:sol lin stoch probl}.
\end{proof}

\begin{rem}
The assumption on the spatial regularity of $f$ in \autoref{cor:sol lin stoch probl} is stronger than one would expect with regard to \eqref{eq:gamma H X_1/2}. 
The reason is that due to the reformulation in terms of the fluid-structure operator, we consider $\overline{g} = M_s^{-1} P_\rm g + M_s^{-1} \gamma_\rm N_2 f$, so the right-hand side $f$ of the fluid equation also enters the structure equation, leading to a stronger requirement regarding the spatial regularity.
\end{rem}

\section{Local-in-time well-posedness of the linearly coupled stochastic FSI problem}\label{sec:loc wp}

In this section, we address the local well-posedness and provide blow-up criteria for the solution. 
For this purpose, from \autoref{cor:sol lin stoch probl}, we recall the solution $Z = (U,H_1,H_2)$ to the linear stochastic evolution equation \eqref{eq:lin stoch PDE}.
On the other hand, we invoke a solution $\tv = (\tu,\teta_1,\teta_2) = (\tu,\teta,\del_t \teta)$ to the complete stochastic FSI problem \eqref{eq:lin coupled stoch FSI probl} reformulated as a first-order-in-time problem.
We then set
\begin{equation*}
    v = (u,\eta_1,\eta_2) \coloneqq (\tu - U,\teta_1 - H_1,\teta_2 - H_2) = \tv - Z.
\end{equation*}
Observing that the noise terms cancel out, we find that the investigation reduces to a purely {\em deterministic problem} given by
\begin{equation}\label{eq:resulting det PDE}
    \left\{
    \begin{aligned}
        \del_t u - \mdiv \sigma_\rf(u,\pi)
        &= -\bigl((u+U) \cdot \nabla\bigr)(u+U), \enspace \mdiv u = 0, &&\tin (0,T) \times \cF_0,\\
        \del_t \eta_1
        &= \eta_2, &&\tin (0,T) \times (0,L),\\
        \del_t \eta_2 + P_\rm \Delta_s^2 \eta_1 
        &= -P_\rm \Delta_s^2 \eta_2 - P_\rm (\mre_2 \cdot \sigma_\rf(u,\pi)\mre_2), &&\tin (0,T) \times (0,L),\\
        u(t,x,0,)
        &= \eta_2(t,x) \mre_2, &&\tfor t \in (0,T), \enspace x \in (0,L),\\
        (u,\pi)(t,0,y) 
        &= (u,\pi)(t,L,y), \enspace (\eta_1,\eta_2)(t,0) = (\eta_1,\eta_2)(t,L), &&\tfor t \in (0,T), \enspace y \in (-1,0),\\
        u(t,x,-1) 
        &= 0, &&\tfor t \in (0,T), \enspace x \in (0,L),\\
        u(0,\cdot)
        &= u_0(\cdot), &&\tin \cF_0,\\
        \eta_1(0,\cdot) &= \eta_1^0(\cdot), \enspace \eta_2(0,\cdot) = \eta_2^0(\cdot), &&\tin (0,L).
    \end{aligned}
    \right.
\end{equation}

Next, we reformulate the deterministic PDE \eqref{eq:resulting det PDE} as a semilinear PDE in terms of the fluid-structure operator.
In fact, similarly as in \eqref{eq:reform of lin probl in terms of the fluid-structure op}, we find that \eqref{eq:resulting det PDE} can be recast as
\begin{equation}\label{eq:det evol eq}
    \frac{\rd}{\rd t}\begin{pmatrix}
        \bP u\\ \eta_1\\ \eta_2
    \end{pmatrix} = A_\fs \begin{pmatrix}
        \bP u\\ \eta_1\\ \eta_2
    \end{pmatrix} + F(v+Z), \tfor t \in (0,T), \enspace \begin{pmatrix}
        \bP u\\ \eta_1\\ \eta_2
    \end{pmatrix}(0) = \begin{pmatrix}
        \bP u_0\\ \eta_1^0\\ \eta_2^0
    \end{pmatrix},
\end{equation}
where the nonlinearity $F$ is bilinear and given by 
\begin{equation}\label{eq:nonlin term}
    F(v+Z) = G(v+Z,v+Z) = \begin{pmatrix}
        -\bP\bigl[\bigl((u+U) \cdot \nabla\bigr)(u+U)\bigr]\\ 0\\ 0
    \end{pmatrix},
\end{equation}
and $(\bP u_0,\eta_1^0,\eta_2^0)^\top$ captures the deterministic part of the initial data.

In the following, we set $\rX_{0,s} \coloneqq \rL_\sigma^s(\cF_0) \times \rW_{\per,\rm}^{4,s}(0,L) \times \rL_\rm^s(0,L)$, and likewise for $\rX_{1,s}$.
We also define
\begin{equation}\label{eq:X_s,theta}
    \rX_{s,\theta} \coloneqq [\rX_{0,s},\rX_{1,s}]_\theta.
\end{equation}
The local pathwise strong well-posedness result is now given as follows.

\begin{thm}\label{thm:loc pathwise wp}
Let $T \in (0,\infty)$, $p$, $q \in (1,\infty)$, as well as $r$, $s \ge 2$ with $s > 2$ if $r > 2$, and such that $\frac{1}{r} - \frac{1}{2p} < \frac{1}{2}$.
Moreover,
\begin{enumerate}[(i)]
    \item if $q < 2$, assume that $\frac{1}{p} + \frac{1}{q} \le \frac{3}{2}$, $\frac{1}{s} - \frac{2-q}{4} \le \frac{1}{2}$ and $\frac{1}{r} - \frac{1}{2p} \le \frac{1}{2} - \frac{1}{s} + \frac{2-q}{4}$, or
    \item if $q \ge 2$, assume that $\frac{1}{s} - \frac{1}{q} \le \frac{1}{2}$ and $\frac{1}{r} - \frac{1}{2p} \le \frac{1}{2} - \frac{1}{s} + \frac{1}{q}$.
\end{enumerate}
Assume that 
\begin{enumerate}[(a)]
    \item $Z_0 = (U_0,H_1^0,H_2^0)$ is strongly $\frF_0$-measurable such that $U_0 \in \rB_{sr,\per}^{2-\frac{2}{r}}(\cF_0)^2$, $H_1^0 \in \rW_\per^{4,s}(0,L)$, and $H_2^0 \in \rB_{sr,\per}^{4-\frac{4}{s}}(0,L)$ satisfy \eqref{eq:comp conds init data},
    \item $h = (f,0,g)$ is $\frF$-adapted, and the forcing terms $f \in \rL^r(\Omega;\rL^r(0,T;\rW_\per^{1+\frac{1}{s},s}(\cF_0;\cH)^2))$ as well as $g \in \rL^r(\Omega;\rL^r(0,T;\rW_\per^{2,s}(0,L;\cH)))$ are such that \eqref{eq:comp conds RHS} holds true,
    \item and for the deterministic part $v_0 = (u_0,\eta_1^0,\eta_2^0)$ of the initial data, it is valid that $u_0 \in \rB_{qp,\per}^{2-\frac{2}{p}}(\cF_0)^2$, $\eta_1^0 \in \rW_\per^{4,q}(0,L)$, and $\eta_2^0 \in \rB_{qp,\per}^{4-\frac{4}{p}}(0,L)$ fulfill \eqref{eq:comp conds init data}.
\end{enumerate}
Then there exists a unique, local-in-time strong solution $\tv = (\tu,\teta_1,\teta_2) = Z + v$ to \eqref{eq:lin coupled stoch FSI probl} with initial data $\tu_0 = U_0 + u_0$, $\teta_1^0 = H_1^0 + \eta_1^0$, and $\teta_2^0 = H_2^0 + \eta_2^0$.
Here $Z = (\bP U,H_1,H_2)$ is the solution to the linear stochastic PDE \eqref{eq:lin stoch PDE} and satisfies $Z \in \rH^{\theta,r}(0,T;\rX_{s,1-\theta}) \cap \rC([0,T];(\rX_{0,s},\rX_{1,s})_{1-\nicefrac{1}{r},r})$ for all $\theta \in [0,\nicefrac{1}{2})$.
Moreover, there exists $T^* \in (0,T]$ so that $v = (\bP u,\eta_1,\eta_2) \in \rW^{1,p}(0,T^*;\rX_0) \cap \rL^p(0,T^*;\rX_1)$ is a deterministic function solving \eqref{eq:resulting det PDE}.
\end{thm}

For convenience of the reader, and as this corresponds exactly to the setting required for the global well-posedness result shown in the next section, we discuss the special case $p = q = r = s = 2$ below.
The assertion of the corollary is an immediate consequence of \autoref{thm:loc pathwise wp} upon noting that the assumptions on the parameters are satisfied in the present case.

\begin{cor}\label{cor:loc wp Hilbert space}
Let $T \in (0,\infty)$, and assume that 
\begin{enumerate}[(a)]
    \item $Z_0 = (U_0,H_1^0,H_2^0)$ is strongly $\frF_0$-measurable such that $U_0 \in \rH_\per^1(\cF_0)^2$, $H_1^0 \in \rH_\per^4(0,L)$, and $H_2^0 \in \rH_\per^2(0,L)$ fulfill \eqref{eq:comp conds init data},
    \item $h = (f,0,g)$ is $\frF$-adapted such that the forcing terms $f \in \rL^2(\Omega;\rL^2(0,T;\rH_\per^{\frac{3}{2}}(\cF_0;\cH)^2))$ and $g \in \rL^2(\Omega;\rL^2(0,T;\rH_\per^{2}(0,L;\cH)))$ satisfy \eqref{eq:comp conds RHS},
    \item $v_0 = (u_0,\eta_1^0,\eta_2^0)$ fulfills $u_0 \in \rH_\per^1(\cF_0)^2$, $\eta_1^0 \in \rH_\per^4(0,L)$, $\eta_2^0 \in \rH_\per^2(0,L)$, and \eqref{eq:comp conds RHS}.
\end{enumerate}
Then there exists a unique, local-in-time strong solution $\tv = (\tu,\teta_1,\teta_2) = v + Z$ to \eqref{eq:lin coupled stoch FSI probl}, where the solution $Z = (\bP U,H_1,H_2)$ to \eqref{eq:lin stoch PDE} satisfies $Z \in \rH^{\theta}(0,T;\rX_{1-\theta}) \cap \rC([0,T];(\rX_0,\rX_1)_{\nicefrac{1}{2},2})$ for all $\theta \in [0,\nicefrac{1}{2})$, and there exists $T^* \in (0,T]$ so that $v = (\bP u,\eta_1,\eta_2) \in v \in \rH^{1}(0,T^*;\rX_0) \cap \rL^2(0,T^*;\rX_1)$ solves \eqref{eq:resulting det PDE}.
\end{cor}

For the proof of \autoref{thm:loc pathwise wp}, we first show estimates of the bilinear nonlinearity~$F$.

\begin{lem}\label{lem:nonlin ests}
 For $v = (u,\eta_1,\eta_2)$, recall $F(v) = G(v,v)$ from \eqref{eq:nonlin term}, consider $\rX_\beta = [\rX_0,\rX_1]_\beta$ for $\beta \in (0,1)$, with characterization provided in \autoref{prop:char of the interpol spaces}, and let $v_1$, $v_2 \in \rX_\beta$.
\begin{enumerate}[(a)]
    \item Let $q \in (1,2)$.
    Then for $\beta \ge \frac{1}{4}(1+\frac{2}{q})$, it holds that $\| G(v_1,v_2) \|_{\rX_0} \le C \cdot \| v_1 \|_{\rX_\beta} \cdot \| v_2 \|_{\rX_\beta}$.
    \item In the case $q \ge 2$, for $\beta \in [\frac{1}{2},1)$, we find that $\| G(v_1,v_2) \|_{\rX_0} \le C \cdot \| v_1 \|_{\rX_\beta} \cdot \| v_2 \|_{\rX_\beta}$.
\end{enumerate}
\end{lem}

\begin{proof}
Let us start with~(a).
By H\"older's inequality, for $r$, $r' \in (1,\infty)$ such that $\nicefrac{1}{r} + \nicefrac{1}{r'} = 1$, we obtain
\begin{equation*}
    \| G(v_1,v_2) \|_{\rX_0} \le \| (u_1 \cdot \nabla) u_2 \|_{\rL^q(\cF_0)} \le C \cdot \| u_1 \|_{\rL^{qr'}(\cF_0)} \cdot \| u_1 \|_{\rH_\per^{1,qr}(\cF_0)}.
\end{equation*}
The choice $\nicefrac{2}{qr} = \nicefrac{1}{2}(1+\nicefrac{2}{q})$, requiring $q < 2$, yields that the Sobolev indices of $\rL^{qr'}(\cF_0)$ and $\rH_\per^{1,qr}(\cF_0)$ coincide.
By Sobolev embeddings, see, e.g., \cite[Thm.4.6.1]{Tri:78}, we get $\rH_\per^{2 \beta,q}(\cF_0) \hookrightarrow \rL^{qr'}(\cF_0) \cap \rH_\per^{1,qr}(\cF_0)$ for $\beta \ge \frac{1}{4}(1 + \frac{2}{q})$, showing the assertion of~(a).

For~(b), we first observe that for all $r \in (1,\infty)$, it holds that $1 - \frac{2}{qr} \ge -\frac{2}{qr'}$, so the embedding $\rH_\per^{2 \beta,q}(\cF_0) \hookrightarrow \rH_\per^{1,qr}(\cF_0)$ is always the more restrictive one.
For $\beta \ge \frac{1}{2}$, we estimate 
\begin{equation*}
    \| G(v_1,v_2) \|_{\rX_0} \le \| (u_1 \cdot \nabla) u_2 \|_{\rL^q(\cF_0)} \le C \cdot \| u_1 \|_{\rL^\infty(\cF_0)} \cdot \| u_1 \|_{\rH_\per^{1,q}(\cF_0)} \le C \cdot \| v_1 \|_{\rX_\beta} \cdot \| v_2 \|_{\rX_\beta}. \qedhere
\end{equation*}
\end{proof}

\begin{proof}[Proof of \autoref{thm:loc pathwise wp}]
We use theory for {\em deterministic} evolution equations,  see \cite[Thm.~1.2]{PW:17} and \cite[Thm.~2.1]{PSW:18}, to tackle the emerging deterministic problem \eqref{eq:det evol eq}.

First, note that the existence of the stochastic part $Z$ of the solution is guaranteed by \autoref{cor:sol lin stoch probl}.
The bounded $\Hinfty$-calculus of the fluid-structure operator $A_\fs$ as established in \autoref{thm:bdd Hoo-calculus fluid-structure op} especially yields the (deterministic) maximal $\rL^p$-regularity as recalled in \autoref{lem:cons of H00}(b).
Observe that the assumptions on~$v_0$ imply that $v_0 \in \rX_\gamma = (\rX_0,\rX_1)_{1-\nicefrac{1}{p},p}$.
Thus, it remains to estimate the nonlinear term.
In view of the bilinearity of the above $F$, we get
\begin{equation*}
    G(Z + v,Z + v) - G(Z + v',Z + v') = G(v-v',v) + G(v',v-v') + G(Z,v-v') + G(v-v',Z).
\end{equation*}
Making use of \autoref{lem:nonlin ests} and its proof, for the respective $\beta$, we first find that
\begin{equation}\label{eq:est bilin}
    \begin{aligned}
        &\quad \| G(Z + v,Z + v) - G(Z + v',Z + v') \|_{\rX_0}\\
        &\le C\bigl(\| v \|_{\rX_\beta} + \| v' \|_{\rX_\beta} + \| U \|_{\rW_\per^{1,qr'}(\cF_0)} + \| U \|_{\rL^{qr}(\cF_0)}\bigr) \cdot \| v - v' \|_{\rX_\beta}
    \end{aligned}
\end{equation}
for suitable $r$, $r' \in [1,\infty]$.
With regard to \cite[Thm.~1.2]{PW:17}, we need to make sure that $\beta \in (1-\frac{1}{p},1)$ and $2 \beta - 1 \le 1 - \frac{1}{p}$.
In the case $q < 2$, we can find such $\beta$ provided $\frac{1}{p} + \frac{1}{q} \le \frac{3}{2}$ as assumed in~(i).
Note that in the case $q \ge 2$, one can always find such $\beta \ge \frac{1}{2}$.

It now remains to make sure that $U$ possesses the right regularity properties.
In view of \cite[Thm.~1.2]{PW:17} and the estimate \eqref{eq:est bilin}, it suffices to show that $U \in \rL^{2p}\bigl(0,T;\rW_\per^{1,qr'}(\cF_0)^2 \cap \rL^{qr}(\cF_0)^2\bigr)$.

On the other hand, from \autoref{cor:sol lin stoch probl} together with the characterization of the interpolation spaces from \autoref{prop:char of the interpol spaces}, we recall that for all $\theta \in [0,\frac{1}{2})$, we have $U \in \rH^{\theta,r}(0,T;\rH_\per^{2(1-\theta),s}(\cF_0)^2)$.

Thus, we first require that $\theta - \frac{1}{r} \ge -\frac{1}{2p}$, ensuring the validity of the embedding of the time components of the spaces.
Together with $\theta < \frac{1}{2}$, this gives rise to the constraint $\frac{1}{r} - \frac{1}{2p} < \frac{1}{2}$, as assumed in the statement of the theorem.
By the arguments in the proof of \autoref{lem:nonlin ests}, we find that in the case $q \ge 2$, it suffices that $\rH_\per^{2(1-\theta),s}(\cF_0) \hookrightarrow \rW_\per^{1,q}(\cF_0)$, while if $q < 2$, we need $\rH_\per^{2(1-\theta),s}(\cF_0) \hookrightarrow \rW_\per^{1,\frac{4q}{2-q}}(\cF_0)$.
As a result, we need $2(1-\theta) - \frac{2}{s} \ge 1 - \frac{2}{q}$ if $q \ge 2$, and $2(1-\theta) - \frac{2}{s} \ge 1 - \frac{2-q}{2}$ if $q < 2$.
With regard to $\theta \ge 0$, this gives rise to the constraint $\frac{1}{s} - \frac{1}{q} \le \frac{1}{2}$ if $q \ge 2$, and $\frac{1}{s} - \frac{2-q}{4} \le \frac{1}{2}$ if $q < 2$, which is guaranteed by the assumptions of the theorem.
On the other hand, together with the above requirement on $\theta$, this leads to $\frac{1}{r} - \frac{1}{2p} \le \frac{1}{2} - \frac{1}{s} + \frac{1}{q}$ if $q \ge 2$, and $\frac{1}{r} - \frac{1}{2p} \le \frac{1}{2} - \frac{1}{s} + \frac{2-q}{4}$ if $q < 2$.
This is also ensured by the assumptions of the theorem.
In total, in all cases, it follows that $U$ satisfies the required regularity assumptions.
\end{proof}

We complete this section with blow-up criteria for the deterministic part of the solution $v$.
It follows in the same way as \cite[Cor.~2.3]{PSW:18}, and upon noting that the cases $q \ge 2$ as well as $q < 2$ with $\frac{1}{p} + \frac{1}{q} < \frac{3}{2}$ are subcritical in the terminology of \cite{PSW:18}.

\begin{cor}\label{cor:blow-up crit}
The deterministic part $v$ of the solution from \autoref{thm:loc pathwise wp} exists globally if
\begin{enumerate}[(a)]
    \item $v([0,t_+)) \subset (\rX_0,\rX_1)_{1-\nicefrac{1}{p},p}$ is relatively bounded in the cases $q \ge 2$ and $q < 2$ with $\frac{1}{p} + \frac{1}{q} < \frac{3}{2}$, or
    \item $v([0,t_+)) \subset (\rX_0,\rX_1)_{1-\nicefrac{1}{p},p}$ is relatively compact in the case $q < 2$ and $\frac{1}{p} + \frac{1}{q} = \frac{3}{2}$.
\end{enumerate}
\end{cor}

The result below specifies the blow-up criterion for the situation of $p = q = 2$ that is of prime importance for \autoref{sec:glob wp lin coupled case}.
It follows directly from \autoref{cor:blow-up crit} and the observation that this case is subcritical.

\begin{cor}\label{cor:blow-up crit Hilbert space}
For $p = q = 2$, the deterministic part $v$ in \autoref{thm:loc pathwise wp} exists on the interval $(0,t_+)$ provided $v \in \rL^\infty(0,t_+;(\rX_0,\rX_1)_{\nicefrac{1}{2},2})$.
\end{cor}

\section{Global-in-time well-posedness in the 2D/1D case}
\label{sec:glob wp lin coupled case}

In this section, we prove the main result of this manuscript, \autoref{thm:global strong well posedness 2d1d}, on the global pathwise strong well-posedness result in the case $p = q = 2$.

\begin{thm}\label{thm:global strong well posedness 2d1d}
Let $T \in (0,\infty)$, and suppose that 
\begin{enumerate}[(a)]
    \item $Z_0 = (U_0,H_1^0,H_2^0)$ is strongly $\frF_0$-measurable, and $U_0 \in \rH_\per^1(\cF_0)^2$, $H_1^0 \in \rH_\per^4(0,L)$ as well as $H_2^0 \in \rH_\per^2(0,L)$ satisfy \eqref{eq:comp conds init data},
    \item $h = (f,0,g)$ is $\frF$-adapted, with the forcing terms $f \in \rL^2(\Omega;\rL^2(0,T;\rH_\per^{\frac{3}{2}}(\cF_0;\cH)^2))$ as well as $g \in \rL^2(\Omega;\rL^2(0,T;\rH_\per^{2}(0,L;\cH)))$ fulfilling~\eqref{eq:comp conds RHS},
    \item $v_0 = (u_0,\eta_1^0,\eta_2^0)$ has the property that $u_0 \in \rH_\per^1(\cF_0)^2$, $\eta_1^0 \in \rH_\per^4(0,L)$, and $\eta_2^0 \in \rH_\per^2(0,L)$ satisfy \eqref{eq:comp conds init data}.
\end{enumerate}
Then there exists a unique, global-in-time strong solution $\tv = (\tu,\teta_1,\teta_2) = Z + v$ to \eqref{eq:lin coupled stoch FSI probl} with initial data $\tu_0 = U_0 + u_0$, $\teta_1^0 = H_1^0 + \eta_1^0$, and $\teta_2^0 = H_2^0 + \eta_2^0$, i.e., the solution exists on $(0,T)$.

More precisely, $Z = (\bP U,H_1,H_2)$ solves \eqref{eq:lin stoch PDE}, and satisfies $Z \in \rH^{\theta}(0,T;\rX_{1-\theta}) \cap \rC([0,T];(\rX_0,\rX_1)_{\nicefrac{1}{2},2})$ for all $\theta \in [0,\nicefrac{1}{2})$, while $v = (\bP u,\eta_1,\eta_2) \in \rH^{1}(0,T;\rX_0) \cap \rL^2(0,T;\rX_1)$ is deterministic and solves \eqref{eq:det evol eq} on~$(0,T)$.
\end{thm}

By \autoref{cor:blow-up crit Hilbert space}, the task is to show that $v$ remains bounded in $\rL^\infty\bigl(0,T;(\rX_0,\rD(A_\fs))_{\nicefrac{1}{2},2}\bigr)$ on every finite time interval $(0,T)$.
From \autoref{prop:char of the interpol spaces}, we recall that
\[
    (\rX_0,\rD(A_\fs))_{\nicefrac{1}{2},2}
    \hookrightarrow
    \rH_\per^1(\cF_0)^2 \times \rH_\per^4(0,L) \times \rH_\per^2(0,L).
\]
Moreover, we observe that the coupling and boundary conditions as well as the spatial average zero are satisfied by the a priori local-in-time solution, and they do not affect the norms, so it is sufficient to establish the bounds
\[
    u \in \rL^\infty(0,T;\rH_\per^1(\cF_0)^2),
    \enspace
    \eta \in \rL^\infty(0,T;\rH_\per^4(0,L)), \tand
    \del_t \eta \in \rL^\infty(0,T;\rH_\per^2(0,L)).
\]

Recall that \(Z=(\bP U,H_1,H_2)\) denotes the solution to the linear stochastic problem
\eqref{eq:lin stoch PDE} that results from \autoref{cor:sol lin stoch probl}. 
In particular, for every \(\theta \in [0,\nicefrac{1}{2})\), we have
\begin{equation}\label{eq:reg of Z Hilbert space}
    Z \in \rH^{\theta}(0,T;\rX_{1-\theta}) \cap \rC([0,T];(\rX_{0},\rD(A_{\fs}))_{\nicefrac{1}{2},2}).
\end{equation}
Throughout this section, for the sake of notation, we write $\del_{xx}$ and $\del_{xxxx}$ instead of $\Delta_s$ and $\Delta_s^2$, respectively, since we will often use integration by parts.
Moreover, for the same reasons, we will also use $\eta$ and $\del_t \eta$ instead of $\eta_1$ and $\eta_2$ as well as the shorthand notation $\phi(u,\pi) = -\mre_2\cdot\sigma_\rf(u,\pi)\mre_2$ in this section.
With these small notational modifications, we invoke the deterministic problem from \eqref{eq:resulting det PDE} given by
\begin{equation}\label{eq:syst for blow-up crit}
    \left\{
    \begin{aligned}
        \del_t u - \mdiv \sigma_\rf(u,\pi)
        &= -\bigl((u+U) \cdot \nabla\bigr)(u+U), &&\tin (0,T) \times \cF_0,\\
        \mdiv u
        &= 0, &&\tin (0,T) \times \cF_0,\\
       \del_{tt}\eta + P_\rm \del_{xxxx}\eta +P_\rm \del_{xxxxt}\eta
        &= P_\rm \phi(u,\pi), &&\tin (0,T) \times (0,L),\\
        u(t,x,0)
        &= \del_t \eta(t,x) \mre_2, &&\tfor t \in (0,T), \enspace x \in (0,L),\\
        (u,\pi)(t,0,y) = (u,\pi)(t,L,y), \enspace \eta(t,0) &= \eta(t,L), &&\tfor t \in (0,T), \enspace y \in (-1,0),\\
        u(t,x,-1) 
        &= 0, &&\tfor t \in (0,T), \enspace x \in (0,L),\\
        u(0,\cdot)
        &= u_0(\cdot), &&\tin \cF_0,\\
        \eta(0,\cdot) &= \eta_1^0(\cdot), \enspace \del_t\eta(0,\cdot) = \eta_2^0(\cdot), &&\tin (0,L).
    \end{aligned}
    \right.
\end{equation}

\subsection{Basic energy estimate}
\ 

In this subsection, we show the basic energy estimate as a starting point for the a priori estimates required to rule out the blow-up scenario.
The proposition below asserts the energy estimate.

\begin{prop}\label{prop:energy ests}
For $v_0 = (u_0,\eta_1^0,\eta_2^0)$ satisfying the assumptions from \autoref{cor:loc wp Hilbert space}(c), consider the strong solution \((u,\eta,\del_t\eta)\) to 
\eqref{eq:syst for blow-up crit}, whose existence is guaranteed by \autoref{cor:loc wp Hilbert space}.
Then, for every \(T>0\), we have
\begin{equation*}
    \begin{aligned}
        u 
        &\in \rL^\infty(0,T;\rL^2(\cF_0)^2) \cap \rL^2(0,T;\rH_\per^1(\cF_0)^2), \enspace \eta \in \rL^\infty(0,T;\rH_\per^2(0,L)), \tand\\
        \del_t\eta 
        &\in \rL^\infty(0,T;\rL^2(0,L))
        \cap \rL^2(0,T;\rH_\per^2(0,L)).
    \end{aligned}
\end{equation*}
More precisely, if $C_0$ denotes the initial energy, i.e.,
\begin{equation*}
    C_0 \coloneqq \frac12\bigl(\|u(0)\|_{\rL^2(\cF_0)}^2+\|\del_t\eta(0)\|_{\rL^2(0,L)}^2+\|\del_{xx}\eta(0)\|_{\rL^2(0,L)}^2\bigr) = \frac12\bigl(\|u_0\|_{\rL^2(\cF_0)}^2+\|\eta_2^0\|_{\rL^2(0,L)}^2+\|\del_{xx}\eta_1^0\|_{\rL^2(0,L)}^2\bigr),
\end{equation*}
then there exists a nondecreasing function
\(C_1(T)\), depending on \(U\), such that
\begin{equation}\label{eq:basic energy bound}
    \begin{aligned}
        &\|u(t)\|_{\rL^2(\cF_0)}^2
        + \|\del_t\eta(t)\|_{\rL^2(0,L)}^2
        + \|\del_{xx}\eta(t)\|_{\rL^2(0,L)}^2 \\
        &\qquad
        + \int_0^t \|\nabla u(\tau)\|_{\rL^2(\cF_0)}^2\,\rd\tau
        + \int_0^t \|\del_{xx\tau}\eta(\tau)\|_{\rL^2(0,L)}^2\,\rd\tau
        \le C_0 + C_1(T)
    \end{aligned}
\end{equation}
for all \(t\in[0,T]\).
\end{prop}

\begin{proof}
Testing \eqref{eq:syst for blow-up crit}$_1$ by \(u\), integrating by parts, and using the boundary conditions gives
\begin{equation}\label{eq:fluid energy}
    \begin{aligned}
        0
        &=
        \frac12\frac{\rd}{\rd t}\int_{\cF_0}|u|^2\srd(x,y)
        + \int_{\cF_0}|\nabla u|^2\srd(x,y)
        + \int_0^L \phi(u,\pi)\del_t\eta\srd x \\
        &\quad
        + \int_{\cF_0}
        (U\cdot\nabla)u\cdot u
        +(u\cdot\nabla)U\cdot u
        +(U\cdot\nabla)U\cdot u
        \srd(x,y).
    \end{aligned}
\end{equation}
The three perturbative terms in \eqref{eq:fluid energy}$_2$ are estimated as follows. 
First, recalling the regularity properties of $Z$ from \eqref{eq:reg of Z Hilbert space}, and using the characterization of the interpolation spaces from \autoref{prop:char of the interpol spaces} joint with Sobolev embeddings, we get \(U\in \rL^2(0,T;\rL^\infty(\cF_0)^2)\).
Hence, H\"older's and Young's inequalities yield
\begin{equation}\label{eq:fluid pert1}
    \int_{\cF_0}(U\cdot\nabla)u\cdot u\srd(x,y)
    \le
    \frac14\|\nabla u\|_{\rL^2(\cF_0)}^2 + C \|U\|_{\rL^\infty(\cF_0)}^2 \|u\|_{\rL^2(\cF_0)}^2.
\end{equation}
Second, since
\[
    U\in \rL^2(0,T;\rH_\per^2(\cF_0)^2)
    \hookrightarrow
    \rL^2(0,T;\rW_\per^{1,r}(\cF_0)^2) \text{ for every } r<\infty,
\]
we may choose $r$, $r' \in (2,\infty)$ such that $\frac{1}{r} + \frac{1}{r'} = \frac{1}{2}$, and Sobolev embeddings yield that $\rH_\per^1(\cF_0) \hookrightarrow \rL^{r'}(\cF_0)$ for all $r' < \infty$.
H\"older's, Young's, and Poincar\'e's inequalities then lead to
\begin{equation}\label{eq:fluid pert2}
    \begin{aligned}
        \int_{\cF_0}(u\cdot\nabla)U\cdot u\srd(x,y)
        &\le \| u \|_{\rL^{r'}(\cF_0)} \| \nabla U \|_{\rL^r(\cF_0)} \| u \|_{\rL^2(\cF_0)}\\
        &\le \frac{1}{4 C_P}\| u \|_{\rH_\per^1(\cF_0)}^2 + C\| \nabla U \|_{\rL^r(\cF_0)}^2 \| u \|_{\rL^2(\cF_0)}^2\\
        &\le \frac{1}{4} \| \nabla u \|_{\rL^2(\cF_0)}^2 + C\| \nabla U \|_{\rL^r(\cF_0)}^2 \| u \|_{\rL^2(\cF_0)}^2,
    \end{aligned}
\end{equation}
where $C_P > 0$ denotes a constant related to the Poincar\'e inequality.

Similarly as in \eqref{eq:fluid pert1}, we obtain
\begin{equation}\label{eq:fluid pert3}
    \int_{\cF_0}(U\cdot\nabla)U\cdot u\srd(x,y)
    \le \frac14\|\nabla U\|_{\rL^2(\cF_0)}^2 + C \|U\|_{\rL^\infty(\cF_0)}^2 \|u\|_{\rL^2(\cF_0)}^2.
\end{equation}

Testing \eqref{eq:syst for blow-up crit}$_3$ by \(\del_t\eta\) gives
\begin{equation}\label{eq:plate energy}
    \frac12\frac{\rd}{\rd t}\int_0^L|\del_t\eta|^2\srd x
    + \frac12\frac{\rd}{\rd t}\int_0^L|\del_{xx}\eta|^2\srd x
    + \int_0^L|\del_{xxt}\eta|^2\srd x
    - \int_0^L\phi(u,\pi)\del_t\eta\srd x
    =0 .
\end{equation}
Adding \eqref{eq:fluid energy} and \eqref{eq:plate energy}, the coupling term $\phi(u,\pi)$ cancels. 
Using \eqref{eq:fluid pert1}--\eqref{eq:fluid pert3}, we obtain
\begin{equation}\label{eq:basic diff inequality}
    \begin{aligned}
        &\quad \frac12\frac{\rd}{\rd t}\|u\|_{\rL^2(\cF_0)}^2
        +\frac12\|\nabla u(t)\|_{\rL^2(\cF_0)}^2
        +\frac12\frac{\rd}{\rd t}\|\del_t\eta\|_{\rL^2(0,L)}^2
        +\frac12\frac{\rd}{\rd t}\|\del_{xx}\eta\|_{\rL^2(0,L)}^2
        +\|\del_{xxt}\eta(t)\|_{\rL^2(0,L)}^2 \\
        &\le \frac{1}{4} \|\nabla U(t)\|_{\rL^2(\cF_0)}^2 + C \bigl(\|U(t)\|_{\rL^\infty(\cF_0)}^2 + \| \nabla U(t) \|_{\rL^r(\cF_0)}^2\bigr)\|u(t)\|_{\rL^2(\cF_0)}^2.
    \end{aligned}
\end{equation}
Define now the energy
\[
    E(t)
    \coloneqq
    \frac12\|u(t)\|_{\rL^2(\cF_0)}^2
    +\frac12\|\del_t\eta(t)\|_{\rL^2(0,L)}^2
    +\frac12\|\del_{xx}\eta(t)\|_{\rL^2(0,L)}^2
\]
as well as
\begin{equation*}
    a_U(t) \coloneqq C\bigl(\|U(t)\|_{\rL^\infty(\cF_0)}^2 + \| \nabla U(t) \|_{\rL^r(\cF_0)}^2\bigr) \tand b_U(t) \coloneqq \frac{1}{4}\| \nabla U(t) \|_{\rL^2(\cF_0)}^2.
\end{equation*}
The above arguments imply $a_U$, $b_U \in \rL^1(0,T)$ for all $T > 0$.
The inequality \eqref{eq:basic diff inequality} yields that
\begin{equation*}
    E'(t) + \frac{1}{2} \| \nabla u(t) \|_{\rL^2(\cF_0)}^2 + \| \del_{xxt} \eta(t) \|_{\rL^2(0,L)}^2 \le a_U(t) E(t) + b_U(t).
\end{equation*}
Gronwall's inequality leads to
\begin{equation*}
    E(t) \le \left(E(0) + \int_0^t b_U(\tau) \srd \tau\right) \mre^{\int_0^t a_U(\tau) \srd \tau}.
\end{equation*}
Thus, integrating in time in \eqref{eq:basic diff inequality}, we find that
\[
    \begin{aligned}
        &E(t)
        + \frac12 \int_0^t
            \|\nabla u(\tau)\|_{\rL^2(\cF_0)}^2\,\rd\tau
        + \int_0^t
            \|\del_{xx\tau}\eta(\tau)\|_{\rL^2(0,L)}^2\,\rd\tau \\
        &\le
        E(0)
        + \int_0^t
        \left[
            b_U(\tau)
            +
            a_U(\tau)
            \left(
                E(0)
                +
                \int_0^\tau b_U(s)\,\rd s
            \right)
            \mre^{
                \int_0^\tau a_U(s)\,\rd s
            }
        \right]
        \rd\tau
    \end{aligned}
\]
for all $t \in [0,T]$, showing the assertion.
\end{proof}

\subsection{Estimate for the fluid and structure velocities}
\ 

In this subsection, we derive the estimate needed to control
\begin{equation*}
    u \in \rL^\infty(0,T;\rH_\per^1(\cF_0)^2) \cap \rH^1(0,T;\rL^2(\cF_0)^2) \tand \del_t\eta \in \rL^\infty(0,T;\rH_\per^2(0,L))
    \cap \rH^1(0,T;\rL^2(0,L)).
\end{equation*}
The idea is to test the fluid equation by \(\del_t u\) and the plate equation by
\(\del_{tt}\eta\). 
This is compatible with the kinematic boundary condition, since, at least formally,
\[
    \del_t u(t,x,0)=\del_{tt}\eta(t,x)\mre_2 .
\]

The following elliptic estimate, which we recall from \cite[Lemma~1]{GH:16}, will be an important ingredient to control the convective term.

\begin{lem}\label{lem:ell est Stokes}
Let \(f\in\rL^2(\cF_0)^2\) and $\del_t\eta \in \rH_{\per,\rm}^{\frac32}(0,L)$.
Then there exists a unique solution
\[
    (u,\pi)
    \in
    \rH_\per^2(\cF_0)^2
    \times \rH_{\per,\rm}^1(\cF_0)
\]
to the stationary Stokes system
\begin{equation}\label{eq:stat Stokes probl}
    \left\{
    \begin{aligned}
        -\Delta u+\nabla \pi
        &=f, \enspace \mdiv u =0,
        &&\tin\cF_0,\\
        u(x,0)
        &=\del_t\eta(x)\mre_2,
        &&\tfor x\in(0,L),\\
        u(x,-1)
        &=0,
        &&\tfor x\in(0,L),\\
        (u,\pi)(0,y)
        &= (u,\pi)(L,y), &&\tfor y \in (-1,0).
    \end{aligned}
    \right.
\end{equation}
Moreover, there exists \(C>0\) such that
\begin{equation}\label{eq:ell est Stokes}
    \|u\|_{\rH_\per^2(\cF_0)}
    + \|\pi\|_{\rH_\per^1(\cF_0)}
    \le
    C\Bigl(
        \|f\|_{\rL^2(\cF_0)}
        + \|\del_t\eta\|_{\rH_\per^{\frac32}(0,L)}
    \Bigr).
\end{equation}
\end{lem}

We are now in a position to tackle the estimates described in the beginning of this subsection.

\begin{prop}\label{prop:first further est}
For $v_0 = (u_0,\eta_1^0,\eta_2^0)$ as specified in \autoref{cor:loc wp Hilbert space}(c), from \autoref{cor:loc wp Hilbert space}, recall the strong solution \((u,\eta,\del_t\eta)\) to 
\eqref{eq:syst for blow-up crit}.
Then, for every \(T>0\), we have
\[
    u \in \rL^\infty(0,T;\rH_\per^1(\cF_0)^2) \cap \rH^1(0,T;\rL^2(\cF_0)^2), \tand \del_t\eta \in \rL^\infty(0,T;\rH_\per^2(0,L)) \cap \rH^1(0,T;\rL^2(0,L)).
\]
\end{prop}

\begin{proof}
Testing \eqref{eq:syst for blow-up crit}$_1$ by \(\del_t u\), integrating by parts, and using the boundary condition gives
\begin{equation}\label{eq:fluid 2nd est}
    \begin{aligned}
        0
        &=
        \|\del_tu\|_{\rL^2(\cF_0)}^2
        + \int_{\cF_0}(u\cdot\nabla)u\cdot\del_tu\srd(x,y)
        + \frac12\frac{\rd}{\rd t}\|\nabla u\|_{\rL^2(\cF_0)}^2
        + \int_0^L \phi(u,\pi)\del_{tt}\eta\srd x \\
        &\quad
        + \int_{\cF_0}
        (U\cdot\nabla)u\cdot\del_tu
        +(u\cdot\nabla)U\cdot\del_tu
        +(U\cdot\nabla)U\cdot\del_tu
        \srd(x,y).
    \end{aligned}
\end{equation}
First, we handle the terms involving \(U\).
In fact, H\"older's and Young's inequalities first yield that
\begin{equation*}
    \begin{aligned}
        \int_{\cF_0} (U\cdot\nabla)u\cdot\del_tu \srd(x,y) 
        &\le \| U \|_{\rL^\infty(\cF_0)} \| \nabla u \|_{\rL^2(\cF_0)} \| \del_t u \|_{\rL^2(\cF_0)}\\
        &\le C_1 \| U \|_{\rL^\infty(\cF_0)}^4 + C_2 \| \nabla u \|_{\rL^2(\cF_0)}^4 + \frac{1}{16} \| \del_t u \|_{\rL^2(\cF_0)}^2.
    \end{aligned}
\end{equation*}
For the next term, we first note that for $r \in (1,\infty)$, interpolation implies that
\begin{equation*}
    \| u \|_{\rL^{2r}(\cF_0)} \le C \| u \|_{\rL^{2}(\cF_0)}^{\frac{1}{r}} \| u \|_{\rH_\per^1(\cF_0)}^{1-\frac{1}{r}} \le C \| u \|_{\rL^{2}(\cF_0)}^{\frac{1}{r}} \| \nabla u \|_{\rL^2(\cF_0)}^{1-\frac{1}{r}}.
\end{equation*}
Thus, for such $r \in (1,\infty)$, Young's inequality leads to
\begin{equation*}
    \| u \|_{\rL^{2r}(\cF_0)}^2 \le C\bigl(\| u \|_{\rL^2(\cF_0)}^2 + \| \nabla u \|_{\rL^2(\cF_0)}^2\bigr).
\end{equation*}
Using this estimate together with H\"older's and Young's inequalities, we find that
\begin{equation*}
    \begin{aligned}
        \int_{\cF_0} (u\cdot\nabla)U\cdot\del_tu \srd(x,y) 
        &\le \| u \|_{\rL^{2r}(\cF_0)} \| \nabla U \|_{\rL^{2r'}(\cF_0)} \| \del_t u \|_{\rL^2(\cF_0)}\\
        &\le C_1 \| \nabla U \|_{\rL^{2r'}(\cF_0)}^2 + C_2\bigl(\| u \|_{\rL^2(\cF_0)}^4 + \| \nabla u \|_{\rL^2(\cF_0)}^4\bigr) + \frac{1}{16} \| \del_t u \|_{\rL^2(\cF_0)}^2.
    \end{aligned}
\end{equation*}
For the third term, by \eqref{eq:reg of Z Hilbert space}, for $\theta \in (\frac{1}{4},\frac{1}{2})$, we have 
\begin{equation*}
    U \in \rH^{\theta}(0,T;\rH_\per^{2(1-\theta)}(\cF_0)^2) \hookrightarrow \rL^4(0,T;\rL^\infty(\cF_0)^2) \cap \rL^4(0,T;\rH_\per^1(\cF_0)^2).
\end{equation*}
Hence, we get
\begin{equation*}
    \begin{aligned}
        \int_{\cF_0} (U\cdot\nabla)U\cdot\del_tu
        \srd(x,y)
        &\le \| U \|_{\rL^\infty(\cF_0)} \| \nabla U \|_{\rL^2(\cF_0)} \| \del_t u \|_{\rL^2(\cF_0)}\\
        &\le C_1\bigl(\| U \|_{\rL^\infty(\cF_0)}^4 + \| \nabla U \|_{\rL^2(\cF_0)}^4\bigr) + \frac{1}{16} \| \del_t u \|_{\rL^2(\cF_0)}^2.
    \end{aligned}
\end{equation*}
Concatenating the above estimates, we obtain
\begin{equation}\label{eq:pert fluid all 2nd est}
    \begin{aligned}
        &\quad \left|
        \int_{\cF_0}
        (U\cdot\nabla)u\cdot\del_tu
        +(u\cdot\nabla)U\cdot\del_tu
        +(U\cdot\nabla)U\cdot\del_tu
        \srd(x,y)
        \right| \\
        &\le
        \frac{3}{16}\|\del_tu\|_{\rL^2(\cF_0)}^2 + C_2\bigl(\| u \|_{\rL^2(\cF_0)}^4 + \| \nabla u \|_{\rL^2(\cF_0)}^4\bigr)\\
        &\quad +C_1\bigl(\| U \|_{\rL^\infty(\cF_0)}^4 + \| \nabla U \|_{\rL^{2r'}(\cF_0)}^2 + \| \nabla U \|_{\rL^2(\cF_0)}^4\bigr).
    \end{aligned}
\end{equation}
Here the terms in \eqref{eq:pert fluid all 2nd est}$_3$ belongs to \(\rL^1(0,T)\) by the regularity of \(U\) following from \eqref{eq:reg of Z Hilbert space}.

Testing \eqref{eq:syst for blow-up crit}$_3$ by \(\del_{tt}\eta\) yields
\begin{equation}\label{eq:plate 2nd est}
    \|\del_{tt}\eta\|_{\rL^2(0,L)}^2
    + \int_0^L \del_{xxxx}\eta\,\del_{tt}\eta\srd x
    + \frac12\frac{\rd}{\rd t}
        \|\del_{xxt}\eta\|_{\rL^2(0,L)}^2
    - \int_0^L \phi(u,\pi)\del_{tt}\eta\srd x
    =0 .
\end{equation}
Adding \eqref{eq:fluid 2nd est} and \eqref{eq:plate 2nd est}, and using
\eqref{eq:pert fluid all 2nd est}, gives
\begin{equation}\label{eq:tested total eq}
    \begin{aligned}
        &\quad \frac{13}{16}\|\del_tu(t)\|_{\rL^2(\cF_0)}^2
        + \int_{\cF_0}(u\cdot\nabla)u\cdot\del_tu\srd(x,y)
        + \frac12\frac{\rd}{\rd t}\|\nabla u\|_{\rL^2(\cF_0)}^2 \\
        &\qquad
        + \|\del_{tt}\eta(t)\|_{\rL^2(0,L)}^2
        + \int_0^L \del_{xxxx}\eta\,\del_{tt}\eta\srd x
        + \frac12\frac{\rd}{\rd t}
            \|\del_{xxt}\eta\|_{\rL^2(0,L)}^2 \\
        &\le
        B_U(t)
        + C_2\bigl(\| u(t) \|_{\rL^2(\cF_0)}^4 + \| \nabla u(t) \|_{\rL^2(\cF_0)}^4\bigr),
    \end{aligned}
\end{equation}
where
\[
    B_U(t)
    \coloneqq
    C_1\bigl(\| U(t) \|_{\rL^\infty(\cF_0)}^4 + \| \nabla U(t) \|_{\rL^{2r'}(\cF_0)}^2 + \| \nabla U(t) \|_{\rL^2(\cF_0)}^4\bigr)
    \in \rL^1(0,T).
\]

We next handle the term involving \(\del_{xxxx}\eta\). By integration by parts in space and time,
\begin{equation}\label{eq:plate high integration by parts}
    \begin{aligned}
        &\quad \int_0^T\int_0^L
        \del_{xxxx}\eta\,\del_{tt}\eta
        \srd x\,\rd t\\
        &=
        \int_0^T\int_0^L
        \del_{xx}\eta\,\del_{xxtt}\eta
        \srd x\,\rd t \\
        &=
        -\int_0^T\|\del_{xxt}\eta(t)\|_{\rL^2(0,L)}^2\,\rd t
        + \int_0^L
        \del_{xx}\eta(T)\del_{xxt}\eta(T)
        \srd x
        - \int_0^L
        \del_{xx}\eta(0)\del_{xxt}\eta(0)
        \srd x .
    \end{aligned}
\end{equation}
Integrating \eqref{eq:tested total eq} in time, and using
\eqref{eq:plate high integration by parts}, for all $t \in [0,T]$, we obtain
\begin{equation}\label{eq:a priori est}
    \begin{aligned}
        &\quad \frac{13}{16}\int_0^t
            \|\del_\tau u(\tau)\|_{\rL^2(\cF_0)}^2\,\rd \tau
        + \frac12\|\nabla u(t)\|_{\rL^2(\cF_0)}^2
        + \int_0^t
            \|\del_{\tau\tau}\eta(\tau)\|_{\rL^2(0,L)}^2\,\rd \tau
        + \frac12\|\del_{xxt}\eta(t)\|_{\rL^2(0,L)}^2 \\
        &\le
        \Tilde{C}_0
        - \int_0^t\int_{\cF_0}
            (u\cdot\nabla)u\cdot\del_\tau u
            \srd(x,y)\,\rd \tau
        + \int_0^t
            \|\del_{xx\tau}\eta(\tau)\|_{\rL^2(0,L)}^2\,\rd \tau
        - \int_0^L
            \del_{xx}\eta(t)\del_{xxt}\eta(t)
            \srd x \\
        &\quad
        + \int_0^t B_U(\tau)\,\rd \tau
        + C_2\int_0^t
            \bigl(
                \|u(\tau)\|_{\rL^2(\cF_0)}^4
                + \|\nabla u(\tau)\|_{\rL^2(\cF_0)}^4
            \bigr)\,\rd \tau ,
    \end{aligned}
\end{equation}
where \(\Tilde{C}_0\) depends only on the initial data $u_0$, $\eta_1^0$, and $\eta_2^0$.

The term 
\[
    \int_0^t B_U(\tau)\,\rd \tau
\]
is bounded for every $t \in (0,T)$ due to $B_U \in \rL^1(0,T)$ as observed above, while the term 
\[
    \int_0^t\|\del_{xx\tau}\eta(\tau)\|_{\rL^2(0,L)}^2\,\rd \tau,
\]
is bounded by $C_0 + C_1(T)$ for every $t \in (0,T)$ by virtue of \autoref{prop:energy ests}.
Moreover, by Young's inequality and
\eqref{eq:basic energy bound}, for every $t \in (0,T)$, we get
\begin{equation*}
    \begin{aligned}
        \left|
    \int_0^L
    \del_{xx}\eta(t)\del_{xxt}\eta(t)\srd x
    \right|
    &\le
    C\|\del_{xx}\eta(t)\|_{\rL^2(0,L)}^2
    + \frac14\|\del_{xxt}\eta(t)\|_{\rL^2(0,L)}^2\\
    &\le C(C_0 + C_1(T))
    + \frac14\|\del_{xxt}\eta(t)\|_{\rL^2(0,L)}^2 .        
    \end{aligned}
\end{equation*}

It remains to estimate the convective term. By Young's inequality,
\begin{equation}\label{eq:convective young}
    \left|
    \int_0^t\int_{\cF_0}
        (u\cdot\nabla)u\cdot\del_\tau u
        \srd(x,y)\,\rd \tau
    \right|
    \le
    \frac12\int_0^t
        \|((u\cdot\nabla)u)(\tau)\|_{\rL^2(\cF_0)}^2\,\rd \tau
    + \frac12\int_0^t
        \|\del_\tau u(\tau)\|_{\rL^2(\cF_0)}^2\,\rd \tau .
\end{equation}
The idea is to absorb the term involving the time derivative into the left-hand side, and to handle the convective term with Ladyzhenskaya's inequality, which requires controlling the Laplacian, so \autoref{lem:ell est Stokes} will be used.
We first handle the terms involving $U$.
To this end, we recall that $U \in \rL^4(0,T;\rL^\infty(\cF_0))$, and for $\theta \in (\frac{1}{4},\frac{1}{3})$, it follows from~\eqref{eq:reg of Z Hilbert space} that
\begin{equation*}
    U \in \rH^{\theta}(0,T;\rH_\per^{2(1-\theta)}(\cF_0)) \hookrightarrow \rL^4(0,T;\rH_\per^{1,3}(\cF_0)).
\end{equation*}
H\"older's, and Young's inequalities, the embedding $\rH^1(\cF_0) \hookrightarrow \rL^6(\cF_0)$, and Poincar\'e's inequality imply
\begin{equation*}
    \begin{aligned}
        \| (U \cdot \nabla)u \|_{\rL^2(\cF_0)}
        &\le \| U \|_{\rL^\infty(\cF_0)} \| \nabla u \|_{\rL^2(\cF_0)} \le \frac{1}{2} \| U \|_{\rL^\infty(\cF_0)}^2 + \frac{1}{2} \| \nabla u \|_{\rL^2(\cF_0)}^2,\\
        \| (u \cdot \nabla)U \|_{\rL^2(\cF_0)}
        &\le \| u \|_{\rL^6(\cF_0)} \| \nabla U \|_{\rL^3(\cF_0)} \le C \| \nabla u \|_{\rL^2(\cF_0)}^2 + \| \nabla U \|_{\rL^3(\cF_0)}^2, \tand\\
        \| (U \cdot \nabla)U \|_{\rL^2(\cF_0)}
        &\le \| U \|_{\rL^\infty(\cF_0)} \| \nabla U \|_{\rL^2(\cF_0)} \le \frac{1}{2} \| U \|_{\rL^\infty(\cF_0)}^2 + \| \nabla U \|_{\rL^2(\cF_0)}^2.
    \end{aligned}
\end{equation*}
We set $f = -\del_t u - ((u+U) \cdot \nabla)(u+U)$.
Using the above estimates, we get
\begin{equation}\label{eq:est of f}
    \| f(t) \|_{\rL^2(\cF_0)} \le \| \del_t u(t) \|_{\rL^2(\cF_0)} + \| ((u\cdot\nabla)u)(t)\|_{\rL^2(\cF_0)} + C_2 \| \nabla u(t) \|_{\rL^2(\cF_0)}^2 + c_U(t),
\end{equation}
where $c_U(t) \coloneqq C_1\bigl(\| U(t) \|_{\rL^\infty(\cF_0)}^2 + \| \nabla U(t) \|_{\rL^3(\cF_0)}^2 + \| \nabla U(t) \|_{\rL^2(\cF_0)}^2\bigr)$.
By \autoref{lem:ell est Stokes} applied to a Stokes system of the form \eqref{eq:stat Stokes probl} with right-hand side $f$ and \eqref{eq:est of f}, we then obtain
\begin{equation}\label{eq:est of the Laplacian}
    \begin{aligned}
        &\quad \|\Delta u(t)\|_{\rL^2(\cF_0)}\\
        &\le C'(t)\Bigl(
        \|\del_tu(t)\|_{\rL^2(\cF_0)} + \|((u\cdot\nabla)u)(t)\|_{\rL^2(\cF_0)} + \|\del_t\eta(t)\|_{\rH_\per^{\frac32}(0,L)} + C_2 \| \nabla u(t) \|_{\rL^2(\cF_0)}^2 + c_U(t)\Bigr),
    \end{aligned}
\end{equation}
where \(C'(t)\) is a nondecreasing function of time, bounded on bounded intervals.

By Ladyzhenskaya's inequality in two dimensions,
\[
    \|(u\cdot\nabla)u\|_{\rL^2(\cF_0)}^2
    \le
    C\|u\|_{\rL^2(\cF_0)}
    \|\nabla u\|_{\rL^2(\cF_0)}^2
    \|\Delta u\|_{\rL^2(\cF_0)}.
\]
Combining this with \eqref{eq:est of the Laplacian} and using $\|\del_t\eta\|_{\rH_\per^{\frac32}(0,L)} \le C\|\del_{xxt}\eta\|_{\rL^2(0,L)}$, for \(Y(t)=\|((u\cdot\nabla)u)(t)\|_{\rL^2(\cF_0)}^2\), we get
\[
    Y(t)
    \le
    \Tilde{C}'(t) A(t)\bigl(
        \|\del_tu(t)\|_{\rL^2(\cF_0)}
        + Y^{1/2}(t)
        + \|\del_{xxt}\eta(t)\|_{\rL^2(0,L)}
        + \|\nabla u(t)\|_{\rL^2(\cF_0)}^2
        + c_U(t)
    \bigr),
\]
where $\Tilde{C}'(t)$ is again nondecreasing in $t$ and bounded on $(0,T)$, and \(A(t)=\|u(t)\|_{\rL^2(\cF_0)}\|\nabla u(t)\|_{\rL^2(\cF_0)}^2\). 
Applying Young's inequality to the term \(A(t)Y^{1/2}(t)\) and absorbing \(\frac12Y(t)\) into the left-hand side yields
\begin{equation}\label{eq:est of convective term}
    \begin{aligned}
        \|((u\cdot\nabla)u)(t)\|_{\rL^2(\cF_0)}^2
        &\le
        \Tilde{C}'(t)C_3\|u(t)\|_{\rL^2(\cF_0)}^2
        \|\nabla u(t)\|_{\rL^2(\cF_0)}^4
        + \frac18\|\del_tu(t)\|_{\rL^2(\cF_0)}^2\\
        &\quad 
        + \|\del_{xxt}\eta(t)\|_{\rL^2(0,L)}^2 + C_2^2 \| \nabla u(t) \|_{\rL^2(\cF_0)}^4
        + c_U^2(t),
    \end{aligned}
\end{equation}
where $c_U \in \rL^2(0,T)$, and hence $c_U^2 \in \rL^1(0,T)$ by the above arguments.

Using the basic energy estimate \eqref{eq:basic energy bound} to control $\|u\|_{\rL^\infty(0,t;\rL^2(\cF_0))}^2$ and $\|\del_{xxt}\eta\|_{\rL^2(0,t;\rL^2(0,L))}^2$, this implies
\begin{equation}\label{eq:integrated convective estimate}
    \begin{aligned}
        &\left|
        \int_0^t\int_{\cF_0}
            (u\cdot\nabla)u\cdot\del_\tau u
            \,\srd(x,y)\,\rd \tau
        \right| \\
        &\le
        \frac{9}{16}
        \int_0^t
            \|\del_\tau u(\tau)\|_{\rL^2(\cF_0)}^2\,\rd \tau
        +
        C_T
        \int_0^t
            \|\nabla u(\tau)\|_{\rL^2(\cF_0)}^4\,\rd \tau
        +
        C_T,
    \end{aligned}
\end{equation}
where $C_T$ depends on $C_0 + C_1(T)$, $\Tilde{C}'(T)$, and the above norms of $U$.

Inserting the previous estimates into \eqref{eq:a priori est}, and decreasing the positive constants on the left-hand side if necessary, gives, for every $t \in (0,T)$, the estimate
\begin{equation}\label{eq:almost closed high estimate}
    \begin{aligned}
        &\quad \frac{1}{4}\int_0^t
            \|\del_\tau u(\tau)\|_{\rL^2(\cF_0)}^2\,\rd \tau
        + \frac12\|\nabla u(t)\|_{\rL^2(\cF_0)}^2
        + \int_0^t
            \|\del_{\tau\tau}\eta(\tau)\|_{\rL^2(0,L)}^2\,\rd \tau
        + \frac14\|\del_{xx\tau}\eta(t)\|_{\rL^2(0,L)}^2 \\
        &\le \Tilde{C}_0 + C_T + C_T \int_0^t \| \nabla u(\tau) \|_{\rL^2(\cF_0)}^4 \srd \tau,
    \end{aligned}
\end{equation}
where by a slight abuse of notation, the constant $C_T$ may be different from the estimate \eqref{eq:integrated convective estimate}.

We close the estimate by absorption.
As a preparation, we define
\[
    F(a,b) \coloneqq \int_a^b \| \nabla u(t) \|_{\rL^2(\cF_0)}^2 \srd t.
\]
The basic energy estimate yields that $F(a,b) < \infty$, and the absolute continuity of the integral implies that $F(a,b) \to 0$ as $b - a \to 0$.

For $0 \le a < b \le T$, define
\[
    X(a,b) \coloneqq \sup_{a\le t\le b}
    \left(
        \frac{1}{2}\|\nabla u(t)\|_{\rL^2(\cF_0)}^2
        + \frac{1}{4}\|\del_{xxt}\eta(t)\|_{\rL^2(0,L)}^2
    \right)
    +
    \frac{1}{4}\int_a^b
        \|\del_tu(t)\|_{\rL^2(\cF_0)}^2\,\rd t
    +
    \int_a^b
        \|\del_{tt}\eta(t)\|_{\rL^2(0,L)}^2\,\rd t .
\]

For suitable constants $D_0(a)$ and $D_1(b)$, which is nondecreasing in $b$, \eqref{eq:almost closed high estimate} yields that
\begin{equation*}
    \begin{aligned}
        X(a,b) 
        &\le D_0(a) + D_1(b) + D_1(b) \int_a^b \| \nabla u(t) \|_{\rL^2(\cF_0)}^2 \srd t \sup_{a\le t\le b} \|\nabla u(t)\|_{\rL^2(\cF_0)}^2\\
        &\le  D_0(a) + D_1(b) + D_1(b) F(a,b) X(a,b).
    \end{aligned}
\end{equation*}
Now, choosing $b > a$ such that $D_1(b) F(a,b) \le \frac{1}{2}$, we get $X(a,b) \le 2 D_0(a) + 2 D_1(b)$.
Since \(F(0,T)<\infty\), by absolute continuity of the integral we may
choose a finite partition $0=t_0<t_1<\cdots<t_N=T$
with
\[
    D_1(T)F(t_{j-1},t_j)\le \frac12
    \qquad\text{for every }j=1,\dots,N.
\]
Applying the preceding estimate successively on each interval \([t_{j-1},t_j]\), and using the already obtained endpoint bound as part of
the initial constant \(D_0(t_{j-1})\), yields \(X(0,T)<\infty\).
\end{proof}

\subsection{Estimate for the structure displacement}
\ 

The purpose of this subsection is to prove the bound for \(\eta\) in \(\rL^\infty(0,T;\rH_\per^4(0,L))\), which will complete the required a priori estimates together with those obtained in the previous subsection.

First, we invoke the following auxiliary estimate, cf.~\cite[Cor.~2]{GH:16}. It can also be obtained from \autoref{lem:ell est Stokes} together with the trace theorem.

\begin{lem}\label{lem:est of the boundary term}
Let \(f\in\rL^2(\cF_0)^2\) as well as $\del_t\eta\in\rH_{\per,\rm}^{\frac32}(0,L)$, and let \((u,\pi)\) be the solution to \eqref{eq:stat Stokes probl} that results from \autoref{lem:ell est Stokes}. 
Then there exists \(C>0\) such that for the coupling term $\phi(u,\pi)$, we get
\begin{equation}\label{eq:boundary stress estimate}
    \|\phi(u,\pi)\|_{\rH_\per^{\frac12}(0,L)}
    \le
    C\bigl(
        \|f\|_{\rL^2(\cF_0)}
        + \|\del_t\eta\|_{\rH_\per^{\frac32}(0,L)}
    \bigr).
\end{equation}
\end{lem}

In the proposition below, we discuss the estimate of $\eta$.

\begin{prop}\label{prop:H4 eta estimate}
Let $v_0 = (u_0,\eta_1^0,\eta_2^0)$ be as specified in \autoref{cor:loc wp Hilbert space}(c), and consider the associated strong solution $(u,\eta,\del_t \eta)$ to \eqref{eq:syst for blow-up crit}.
Then, for every \(T>0\), we have $\eta\in\rL^\infty(0,T;\rH_\per^4(0,L))$.
\end{prop}

\begin{proof}
Testing \eqref{eq:syst for blow-up crit}$_3$ by \(\del_{xxxx}\eta\) gives
\begin{equation}\label{eq:H4 plate test}
    \int_0^L
        \del_{tt}\eta\,\del_{xxxx}\eta\srd x
    + \|\del_{xxxx}\eta\|_{\rL^2(0,L)}^2
    + \frac12\frac{\rd}{\rd t}
        \|\del_{xxxx}\eta\|_{\rL^2(0,L)}^2
    =
    \int_0^L
        \phi(u,\pi)\del_{xxxx}\eta\srd x .
\end{equation}
The first term on the left-hand side is estimated by Young's inequality:
\[
    \left|
    \int_0^L
        \del_{tt}\eta\,\del_{xxxx}\eta\srd x
    \right|
    \le
    \frac14\|\del_{xxxx}\eta\|_{\rL^2(0,L)}^2
    + C\|\del_{tt}\eta\|_{\rL^2(0,L)}^2 .
\]
For the right-hand side, Young's inequality as well as \autoref{lem:est of the boundary term} applied to the case of a right-hand side $f = -\del_t u + ((u+U) \cdot \nabla)(u+U)$, the estimate of $f$ from \eqref{eq:est of f}, and \eqref{eq:est of convective term} for the estimate of the convective term imply
\begin{equation*}
    \begin{aligned}
        &\quad \left|\int_0^L\phi(u,\pi)\del_{xxxx}\eta\srd x\right|\\
        &\le \frac14\|\del_{xxxx}\eta\|_{\rL^2(0,L)}^2
        + C\|\phi(u,\pi)\|_{\rL^2(0,L)}^2\\
        &\le \frac14\|\del_{xxxx}\eta\|_{\rL^2(0,L)}^2
        + C\bigl(\| f \|_{\rL^2(\cF_0)}^2 + \| \del_{xxt} \eta \|_{\rL^2(0,L)}^2\bigr)\\
        &\le \frac14\|\del_{xxxx}\eta\|_{\rL^2(0,L)}^2 + C\bigl(\| \del_t u(t) \|_{\rL^2(\cF_0)}^2 + \| ((u\cdot\nabla)u)(t)\|_{\rL^2(\cF_0)}^2 + C_2^2 \| \nabla u(t) \|_{\rL^2(\cF_0)}^4 + c_U^2(t)\bigr)\\
        &\le
        \frac14\|\del_{xxxx}\eta\|_{\rL^2(0,L)}^2
        + C\Bigl(
            \|\del_tu\|_{\rL^2(\cF_0)}^2
            + \bigl(\Tilde{C}'(t)C_3
                \|u\|_{\rL^2(\cF_0)}^2 + C_2^2\bigr)
                \|\nabla u\|_{\rL^2(\cF_0)}^4
            + \|\del_{xxt}\eta\|_{\rL^2(0,L)}^2
            + c_U^2(t)
        \Bigr).
    \end{aligned}
\end{equation*}

Substituting these bounds into \eqref{eq:H4 plate test}, we obtain
\begin{equation}\label{eq:H4 diff inequality}
    \begin{aligned}
        &\quad \frac12\frac{\rd}{\rd t}
            \|\del_{xxxx}\eta\|_{\rL^2(0,L)}^2
        + \frac12\|\del_{xxxx}\eta\|_{\rL^2(0,L)}^2\\
        &\le
        C\|\del_{tt}\eta\|_{\rL^2(0,L)}^2
        + C\|\del_tu\|_{\rL^2(\cF_0)}^2
        + \bigl(C'(t)C_3
                \|u\|_{\rL^2(\cF_0)}^2 + C_2^2\bigr)
            \|\nabla u\|_{\rL^2(\cF_0)}^4\\
         &\qquad + C\|\del_{xxt}\eta\|_{\rL^2(0,L)}^2
        + C c_U^2(t).
    \end{aligned}
\end{equation}
The right-hand side belongs to \(\rL^1(0,t)\) for every $t \in (0,T)$, since for every $t \in (0,T)$ by \autoref{prop:energy ests} and \autoref{prop:first further est}, we have
\[
    \int_0^t \|\nabla u(\tau)\|_{\rL^2(\cF_0)}^4\,\rd \tau
    \le
    \|\nabla u\|_{\rL^\infty(0,t;\rL^2(\cF_0))}^2
    \int_0^t \|\nabla u(\tau)\|_{\rL^2(\cF_0)}^2\,\rd \tau
    <\infty.
\]
Therefore, integrating~\eqref{eq:H4 diff inequality} over \((0,t)\), \(t\le T\), yields $\sup_{0\le t\le T} \|\del_{xxxx}\eta(t)\|_{\rL^2(0,L)}^2<\infty$.
Together with the lower-order energy estimate, this proves $\eta\in\rL^\infty(0,T;\rH_\per^4(0,L))$, as desired.
\end{proof}

\subsection{Proof of \autoref{thm:global strong well posedness 2d1d}}
\ 

We are finally in a position to prove the global strong pathwise well-posedness.

\begin{proof}[Proof of \autoref{thm:global strong well posedness 2d1d}]
The local-in-time existence has been shown in \autoref{cor:loc wp Hilbert space}.
Moreover, the estimates in \autoref{prop:energy ests}, \autoref{prop:first further est}, and \autoref{prop:H4 eta estimate} show that the norm specified in the blow-up criterion \autoref{cor:blow-up crit Hilbert space} remains finite on $(0,T)$ for every finite $T > 0$, so a maximal time of existence $t_+$ of the deterministic part $v$ of the solution with $t_+ < T$ is ruled out, showing the assertion of the theorem.
\end{proof}

\section{Local-in-time well-posedness and blow-up criteria in the 3D/2D case}\label{sec:further discussion}

In this section, we elaborate on the local-in-time pathwise well-posedness of the 3D/2D linearly coupled stochastic FSI problem, and we also provide blow-up criteria in the end.

First, let us make precise the setting in the present case.
Here, the underlying domain will be $\cG_0 = (0,L) \times (0,L) \times (-1,0)$, and we still use $\tu$, $\tpi$, and $\teta$ to denote the fluid velocity and pressure as well as the plate displacement, respectively.
To shorten notation, we will also write $G = (0,L) \times (0,L)$ in this case, so $\cG_0 = G \times (-1,0)$.
The space $\rL_\rm^q(G)$ is defined in the same way as $\rL_\rm^q(0,L)$ in \eqref{eq:avg zero space and projection}, while by a slight abuse of notation, we will still denote the associated projection in the wake of \eqref{eq:avg zero space and projection} by $P_\rm$.
Another relevant piece of notation will be $\Delta_\rH f \coloneqq \del_{xx} f + \del_{yy} f$, while in this case, $\nabla$ and $\mdiv$ represent the gradient and divergence in 3D in this case.
Moreover, we observe that the definition of periodic functions is analogous to the 2D/1D case.

With these preparations, and for the stochastic setting as specified around \eqref{eq:repr cylindr Wiener proc}, the full stochastic problem is given by
\begin{equation}\label{eq:3D lin coupled stoch FSI}
    \left\{
    \begin{aligned}
        \rd\tu +(\tu\cdot\nabla)\tu\srd t-\mdiv\sigma_\rf(\tu,\tpi)\srd t
        &=
        f\srd W_\cH, , \enspace \mdiv\tu = 0,
        &&\tin (0,T)\times\cG_0,\\
        \rd\del_t\teta
        +\bigl(P_\rm \Delta_\rH\teta+P_\rm\Delta_\rH\teta\bigr)\srd t
        &=
        -P_\rm\bigl(\mre_3\cdot\sigma_\rf(\tu,\tpi)\mre_3\bigr)\srd t\\
        &\quad +g\srd W_\cH,
        &&\tin (0,T)\times G,\\
        \tu(t,x,y,0)
        &=
        \del_t\teta(t,x,y)\mre_3,
        &&\tfor t\in(0,T), \enspace (x,y)\in G,\\
        \tu(t,x,y,-1)
        &=
        0,
        &&\tfor t\in(0,T), \enspace (x,y)\in G,\\
        \tu(0,\cdot)
        &=
        \tu_0(\cdot),
        &&\tin\cG_0,\\
        \teta(0,\cdot)
        &=
        \teta_1^0(\cdot),
        \enspace
        \del_t\teta(0,\cdot)
        =
        \teta_2^0(\cdot),
        &&\tin G,
    \end{aligned}
    \right.
\end{equation}
and similarly as in \eqref{eq:lin coupled stoch FSI probl}, the problem is complemented by periodic boundary conditions at the lateral boundary of the domain.
In the above, the fluid velocity, the fluid pressure, and the structure displacement are again random variables, i.e.,
\[
    \tu\colon (0,T)\times\Omega\times\cG_0\to\R^3,
    \enspace
    \tpi\colon (0,T)\times\Omega\times\cG_0\to\R, \tand \teta\colon (0,T)\times\Omega\times G\to\R.
\]

Let us now elaborate on the associated linear stochastic problem and its solvability.
Using the same notation as in the 2D/1D case for the variables, and denoting by $\Delta_s^2$ the $\rL^q$-realization of $\Delta_\rH^2$ subject to periodic boundary conditions, the underlying linear stochastic problem is given by
\begin{equation}\label{eq:3D lin stoch PDE}
    \left\{
    \begin{aligned}
        \rd U - \mdiv \sigma_\rf(U,\Pi) \srd t
        &= f \srd W_\cH, \enspace \mdiv U = 0, &&\tin (0,T) \times \cG_0,\\
        \rd H_1 
        &= H_2 \srd t, &&\tin (0,T) \times G,\\
        \rd H_2 + \bigl(P_\rm \Delta_s^2 H_1 + P_\rm \Delta_s^2 H_2\bigr) \srd t
        &= -P_\rm\bigl(\mre_3 \cdot \sigma_\rf(U,\Pi) \mre_3\bigr) \srd t\\
        &\quad + g \srd W_\cH, &&\tin (0,T) \times G,\\
        U(t,x,y,0) 
        &= H_2(t,x,y) \mre_3, &&\tfor t \in (0,T), \enspace (x,y) \in G,\\
        U(t,x,y,-1) 
        &= 0, &&\tfor t \in (0,T), \enspace (x,y) \in G,\\
        U(0,\cdot)
        &= U_0(\cdot), &&\tin \cG_0,\\
        H_1(0,\cdot) &= H_1^0(\cdot), \enspace H_2(0,\cdot) = H_2^0(\cdot), &&\tin G,
    \end{aligned}
    \right.
\end{equation}
supplemented by periodic boundary conditions at the lateral boundary.
By adapting the approach from \autoref{sec:Hoo-calculus fluid-structure op}, we find that the corresponding linear deterministic problem can be reformulated in terms of a fluid-structure operator, and similar as in \autoref{thm:bdd Hoo-calculus fluid-structure op}, it is possible to show the boundedness of its $\Hinfty$-calculus.
This paves the way for the solvability of \eqref{eq:3D lin stoch PDE} in an analogous manner as in \autoref{cor:sol lin stoch probl}.
The regularity properties of this solution will be made precise in the statement of the theorem below.

Next, as in the 2D/1D case, it is possible to reduce the problem to a deterministic one, given by
\begin{equation}\label{eq:resulting det PDE 3D}
    \left\{
    \begin{aligned}
        \del_t u - \mdiv \sigma_\rf(u,\pi)
        &= -\bigl((u+U) \cdot \nabla\bigr)(u+U), \enspace \mdiv u = 0, &&\tin (0,T) \times \cG_0,\\
        \del_t \eta_1
        &= \eta_2, &&\tin (0,T) \times G,\\
        \del_t \eta_2 + P_\rm \Delta_s^2 \eta_1 
        &= -P_\rm \Delta_s^2 \eta_2 - P_\rm (\mre_3 \cdot \sigma_\rf(u,\pi)\mre_3), &&\tin (0,T) \times G,\\
        u(t,x,y,0)
        &= \eta_2(t,x,y) \mre_3, &&\tfor t \in (0,T), \enspace (x,y) \in G,\\
        u(t,x,y,-1) 
        &= 0, &&\tfor t \in (0,T), \enspace (x,y) \in G,\\
        u(0,\cdot)
        &= u_0(\cdot), &&\tin \cG_0,\\
        \eta_1(0,\cdot) &= \eta_1^0(\cdot), \enspace \eta_2(0,\cdot) = \eta_2^0(\cdot), &&\tin G,
    \end{aligned}
    \right.
\end{equation}
where periodic boundary conditions at the lateral boundary are taken into account.

Let us now state the local-in-time pathwise strong well-posedness in the present 3D/2D case.
For simplicity of the statement, we focus here on the situation of a Hilbert space for the deterministic part of the solution.
Note that similarly as in \autoref{thm:loc pathwise wp}, this result also generalizes to a broader range of integrability parameters $p$ and $q$.
Let us note that the spaces $\rX_0$ and $\rX_1$ denote the 3D analogues of the spaces introduced in \eqref{eq:ground space} and \eqref{eq:fluid-structure op}, respectively, while $\rX_{0,s}$, $\rX_{1,s}$, and $\rX_{s,\theta}$ are defined as in \eqref{eq:X_s,theta}.
For the characterization of the interpolation spaces, an analogous result to \autoref{prop:char of the interpol spaces} can be obtained.

\begin{thm}\label{thm:loc pathwise wp 3D}
Let $T \in (0,\infty)$, as well as $r \ge 4$ and $s > 2$.
Besides, assume that 
\begin{enumerate}[(a)]
    \item $Z_0 = (U_0,H_1^0,H_2^0)$ is strongly $\frF_0$-measurable such that $U_0 \in \rB_{sr,\per}^{2-\frac{2}{r}}(\cG_0)^3$, $H_1^0 \in \rW_\per^{4,s}(G)$, and $H_2^0 \in \rB_{sr,\per}^{4-\frac{4}{s}}(G)$ satisfy $\mdiv U_0 = 0$ in $\cG_0$, $\int_G H_1^0 \srd(x,y) = \int_G H_2^0 \srd(x,y) = 0$, $U_0(x,y,-1) = 0$, and $U_0(x,y,0) = H_2^0(x,y) \mre_3$ for all $(x,y) \in G$.
    \item $h = (f,0,g)$ is $\frF$-adapted, with $f \in \rL^r(\Omega;\rL^r(0,T;\rW_\per^{1+\frac{1}{s},s}(\cG_0;\cH)^3))$ and $g \in \rW_\per^{2,s}(G;\cH)$ such that $\bP f(x,y,-1) = 0$ and $\bP f(x,y,0) = \overline{g}(x,y)$ for all $(x,y) \in G$, where $\overline{g} = M_s^{-1} P_\rm g + M_s^{-1} \gamma_\rm N_2 f$, where $M_s^{-1}$, $\gamma_\rm$, and $N_2$ denote the 3D analogues of the inverse of the added mass operator, the modified trace operator, and the Neumann lifting operator,
    \item for the deterministic part $v_0 = (u_0,\eta_1^0,\eta_2^0)$ of the initial data, it is valid that $u_0 \in \rH_{\per}^1(\cG_0)^3$, $\eta_1^0 \in \rH_\per^{4}(G)$, and $\eta_2^0 \in \rH_{\per}^{2}(G)$ fulfill $\mdiv u_0 = 0$ in $\cG_0$, $\int_G \eta_1^0 \srd(x,y) = \int_G \eta_2^0 \srd(x,y) = 0$, $u_0(x,y,-1) = 0$, and $u_0(x,y,0) = \eta_2^0(x,y) \mre_3$ for all $(x,y) \in G$.
\end{enumerate}
Then there exists a unique, local-in-time strong solution $\tv = (\tu,\teta_1,\teta_2) = Z + v$ to \eqref{eq:3D lin coupled stoch FSI}, where $Z = (\bP U,H_1,H_2)$ with $Z \in \rH^{\theta,r}(0,T;\rX_{s,1-\theta}) \cap \rC([0,T];(\rX_{0,s},\rX_{1,s})_{1-\nicefrac{1}{r},r})$ for all $\theta \in [0,\nicefrac{1}{2})$ is the solution to the linear stochastic PDE \eqref{eq:3D lin stoch PDE}, and there is $T^* \in (0,T]$ so that the deterministic function $v = (\bP u,\eta_1,\eta_2) \in \rH^1(0,T^*;\rX_0) \cap \rL^2(0,T^*;\rX_1)$ solves \eqref{eq:resulting det PDE 3D}.
\end{thm}

The proof of \autoref{thm:loc pathwise wp 3D} is similar to that of \autoref{thm:loc pathwise wp} upon making the required adjustments for the 3D case in the estimates of the nonlinear terms discussed in \autoref{lem:nonlin ests}.

We complete this section with a blow-up criterion for the deterministic part of the solution.

\begin{cor}
The deterministic part $v$ of the solution from \autoref{thm:loc pathwise wp 3D} exists globally provided $v([0,t_+)) \subset (\rX_0,\rX_1)_{\nicefrac{1}{2},2}$ is relatively bounded.
\end{cor}

\medskip 

{\bf Acknowledgements.}
F.B.~acknowledges the support by the German National Academy of Sciences Leopoldina through the Leopoldina Fellowship Program with grant number~LPDS 2024-07. 
He also extends his gratitude to the Basque Center for Applied Mathematics for their generous hospitality during his visit, where a portion of this work was completed.
M.H.~acknowledges the support by DFG project FOR~5528. 
A.R.~is supported by the Grant RYC2022-036183-I funded by 
MICIU/AEI/10.13039/501100011033 
and by ESF+. 
A.R.~has been partially supported by the Basque Government through the BERC 2022-2025 program and by the Spanish State Research Agency through BCAM Severo Ochoa CEX2021-001142-S and through project PID2023-146764NB-I00 funded by MICIU/AEI/10.13039/501100011033 and cofunded by the European Union.

\bibliography{strong_stochastic_FSI}

@article{mindrilua2026existence,
  title={Existence of weak solutions for incompressible fluid-Koiter shell interactions with {N}avier slip boundary condition},
  author={M{\^\i}ndril{\u{a}}, Claudiu and Roy, Arnab},
  journal={arXiv preprint arXiv:2602.20016},
  year={2026}
}

@article {MR1808843,
    AUTHOR = {Koch, Herbert and Tataru, Daniel},
     TITLE = {Well-posedness for the {N}avier-{S}tokes equations},
   JOURNAL = {Adv. Math.},
  FJOURNAL = {Advances in Mathematics},
    VOLUME = {157},
      YEAR = {2001},
    NUMBER = {1},
     PAGES = {22--35},
      ISSN = {0001-8708,1090-2082},
   MRCLASS = {35Q30},
  MRNUMBER = {1808843},
MRREVIEWER = {Radjesvarane\ Alexandre},
       DOI = {10.1006/aima.2000.1937},
       URL = {https://doi.org/10.1006/aima.2000.1937},
}

@book {MR4475666,
    AUTHOR = {Bedrossian, Jacob and Vicol, Vlad},
     TITLE = {The mathematical analysis of the incompressible {E}uler and
              {N}avier-{S}tokes equations---an introduction},
    SERIES = {Graduate Studies in Mathematics},
    VOLUME = {225},
 PUBLISHER = {American Mathematical Society, Providence, RI},
      YEAR = {[2022] \copyright 2022},
     PAGES = {xiii+218},
      ISBN = {[9781470470494]},
   MRCLASS = {35-01 (35Q30 35Q31)},
  MRNUMBER = {4475666},
       DOI = {10.1090/gsm/225},
       URL = {https://doi.org/10.1090/gsm/225},
}

@article {AV:25,
    AUTHOR = {Agresti, Antonio and Veraar, Mark},
     TITLE = {Nonlinear {SPDE}s and maximal regularity: an extended survey},
   JOURNAL = {NoDEA Nonlinear Differential Equations Appl.},
  FJOURNAL = {NoDEA. Nonlinear Differential Equations and Applications},
    VOLUME = {32},
      YEAR = {2025},
    NUMBER = {6},
      ISSN = {1021-9722,1420-9004},
   MRCLASS = {60H15 (35K57 35K90 47D06 76M35)},
  MRNUMBER = {4952170},
       DOI = {10.1007/s00030-025-01090-2},
       URL = {https://doi.org/10.1007/s00030-025-01090-2},
       NOTE = {Paper No. 123}
}

@book {CKMT:25,
    AUTHOR = {{\v{C}}ani\'{c}, Sun\v{c}ica and Kuan, Jeffrey and Muha, Boris
              and Tawri, Krutika},
     TITLE = {Deterministic and stochastic fluid-structure interaction},
    SERIES = {Advances in Mathematical Fluid Mechanics},
 PUBLISHER = {Birkh\"{a}user/Springer, Cham},
      YEAR = {[2025] \copyright 2025},
     PAGES = {xvii+613},
      ISBN = {978-3-032-00897-8; 978-3-032-00898-5},
   MRCLASS = {74-01 (35Q35 74B05 74F10 74Hxx 74S60 76Dxx)},
  MRNUMBER = {5033898},
       DOI = {10.1007/978-3-032-00898-5},
       URL = {https://doi.org/10.1007/978-3-032-00898-5},
}

@book {Gal:11,
    AUTHOR = {Galdi, G. P.},
     TITLE = {An introduction to the mathematical theory of the
              {N}avier-{S}tokes equations},
    SERIES = {Springer Monographs in Mathematics},
   EDITION = {Second},
      NOTE = {Steady-state problems},
 PUBLISHER = {Springer, New York},
      YEAR = {2011},
     PAGES = {xiv+1018},
      ISBN = {978-0-387-09619-3},
   MRCLASS = {35Q30 (35-02 76D03 76D05 76D07)},
  MRNUMBER = {2808162},
       DOI = {10.1007/978-0-387-09620-9},
       URL = {https://doi.org/10.1007/978-0-387-09620-9},
}

@article {DP:26,
    AUTHOR = {Debussche, Arnaud and Pappalettera, Umberto},
     TITLE = {Second order perturbation theory of two-scale systems in fluid
              dynamics},
   JOURNAL = {J. Eur. Math. Soc. (JEMS)},
  FJOURNAL = {Journal of the European Mathematical Society (JEMS)},
    VOLUME = {28},
      YEAR = {2026},
    NUMBER = {4},
     PAGES = {1533--1595},
      ISSN = {1435-9855,1435-9863},
   MRCLASS = {76F02 (60H15 60H30)},
  MRNUMBER = {5032984},
       DOI = {10.4171/jems/1501},
       URL = {https://doi.org/10.4171/jems/1501},
}

@book {FL:23,
    AUTHOR = {Flandoli, Franco and Luongo, Eliseo},
     TITLE = {Stochastic partial differential equations in fluid mechanics},
    SERIES = {Lecture Notes in Mathematics},
    VOLUME = {2330},
 PUBLISHER = {Springer, Singapore},
      YEAR = {[2023] \copyright 2023},
     PAGES = {viii+199},
      ISBN = {978-981-99-0387-0; 978-981-99-0385-6},
   MRCLASS = {76M35 (35Q35 35R60 60H15)},
  MRNUMBER = {4628185},
MRREVIEWER = {Luigi Amedeo Bianchi},
       DOI = {10.1007/978-981-99-0385-6},
       URL = {https://doi.org/10.1007/978-981-99-0385-6},
}

@book {Sohr:01,
    AUTHOR = {Sohr, Hermann},
     TITLE = {The {N}avier-{S}tokes equations},
    SERIES = {Birkh\"{a}user Advanced Texts: Basler Lehrb\"{u}cher.
              [Birkh\"{a}user Advanced Texts: Basel Textbooks]},
      NOTE = {An elementary functional analytic approach},
 PUBLISHER = {Birkh\"{a}user Verlag, Basel},
      YEAR = {2001},
     PAGES = {x+367},
      ISBN = {3-7643-6545-5},
   MRCLASS = {35Q30 (35Q35 76D03 76D05)},
  MRNUMBER = {1928881},
MRREVIEWER = {Joel\ David\ Avrin},
       DOI = {10.1007/978-3-0348-8255-2},
       URL = {https://doi.org/10.1007/978-3-0348-8255-2},
}

@article {FK:64,
    AUTHOR = {Fujita, Hiroshi and Kato, Tosio},
     TITLE = {On the {N}avier-{S}tokes initial value problem. {I}},
   JOURNAL = {Arch. Rational Mech. Anal.},
  FJOURNAL = {Archive for Rational Mechanics and Analysis},
    VOLUME = {16},
      YEAR = {1964},
     PAGES = {269--315},
      ISSN = {0003-9527},
   MRCLASS = {35.79},
  MRNUMBER = {166499},
MRREVIEWER = {P.\ C.\ Fife},
       DOI = {10.1007/BF00276188},
       URL = {https://doi.org/10.1007/BF00276188},
}

@book {Lad:69,
    AUTHOR = {Ladyzhenskaya, O. A.},
     TITLE = {The mathematical theory of viscous incompressible flow},
    SERIES = {Mathematics and its Applications, Vol. 2},
      NOTE = {Second English edition, revised and enlarged,
              Translated from the Russian by Richard A. Silverman and John
              Chu},
 PUBLISHER = {Gordon and Breach Science Publishers, New York-London-Paris},
      YEAR = {1969},
     PAGES = {xviii+224},
   MRCLASS = {35.00 (76.00)},
  MRNUMBER = {254401},
}

@article {Hopf:51,
    AUTHOR = {Hopf, Eberhard},
     TITLE = {\"{U}ber die {A}nfangswertaufgabe f\"{u}r die hydrodynamischen
              {G}rundgleichungen},
   JOURNAL = {Math. Nachr.},
  FJOURNAL = {Mathematische Nachrichten},
    VOLUME = {4},
      YEAR = {1951},
     PAGES = {213--231},
      ISSN = {0025-584X,1522-2616},
   MRCLASS = {76.1X},
  MRNUMBER = {50423},
MRREVIEWER = {J.\ Kamp\'{e} de F\'{e}riet},
       DOI = {10.1002/mana.3210040121},
       URL = {https://doi.org/10.1002/mana.3210040121},
}

@article {Ler:34,
    AUTHOR = {Leray, Jean},
     TITLE = {Sur le mouvement d'un liquide visqueux emplissant l'espace},
   JOURNAL = {Acta Math.},
  FJOURNAL = {Acta Mathematica},
    VOLUME = {63},
      YEAR = {1934},
    NUMBER = {1},
     PAGES = {193--248},
      ISSN = {0001-5962,1871-2509},
   MRCLASS = {DML},
  MRNUMBER = {1555394},
       DOI = {10.1007/BF02547354},
       URL = {https://doi.org/10.1007/BF02547354},
}

@article {HM:06,
    AUTHOR = {Hairer, Martin and Mattingly, Jonathan C.},
     TITLE = {Ergodicity of the 2{D} {N}avier-{S}tokes equations with
              degenerate stochastic forcing},
   JOURNAL = {Ann. of Math. (2)},
  FJOURNAL = {Annals of Mathematics. Second Series},
    VOLUME = {164},
      YEAR = {2006},
    NUMBER = {3},
     PAGES = {993--1032},
      ISSN = {0003-486X,1939-8980},
   MRCLASS = {37L55 (35Q30 35R60 60H15 76D05 76M35)},
  MRNUMBER = {2259251},
MRREVIEWER = {Hakima\ Bessaih},
       DOI = {10.4007/annals.2006.164.993},
       URL = {https://doi.org/10.4007/annals.2006.164.993},
}

@article {Goo:23b,
    AUTHOR = {Goodair, Daniel},
     TITLE = {Navier-{S}tokes equations with {N}avier boundary conditions
              and stochastic {L}ie transport: well-posedness and inviscid
              limit},
   JOURNAL = {J. Differential Equations},
  FJOURNAL = {Journal of Differential Equations},
    VOLUME = {429},
      YEAR = {2025},
     PAGES = {1--49},
      ISSN = {0022-0396,1090-2732},
   MRCLASS = {35Q30 (35R60)},
  MRNUMBER = {4865749},
MRREVIEWER = {Hui\ Liu},
       DOI = {10.1016/j.jde.2025.02.036},
       URL = {https://doi.org/10.1016/j.jde.2025.02.036},
}

@InProceedings{Goo:23a,
author="Goodair, Daniel",
editor="Chapron, Bertrand
and Crisan, Dan
and Holm, Darryl
and M{\'e}min, Etienne
and Radomska, Anna",
title="Existence and Uniqueness of Maximal Solutions to a 3D Navier-Stokes Equation with Stochastic Lie Transport",
booktitle="Stochastic Transport in Upper Ocean Dynamics",
year="2023",
publisher="Springer International Publishing",
address="Cham",
pages="87--107",
isbn="978-3-031-18988-3"
}

@article {Hol:15,
    AUTHOR = {Holm, Darryl D.},
     TITLE = {Variational principles for stochastic fluid dynamics},
   JOURNAL = {Proc. A},
  FJOURNAL = {Proceedings A},
    VOLUME = {471},
      YEAR = {2015},
    NUMBER = {2176},
      ISSN = {1364-5021,1471-2946},
   MRCLASS = {35Q35 (35R60 60H15 76B99 76D06 76N99)},
  MRNUMBER = {3325187},
MRREVIEWER = {Marko\ Nedeljkov},
       DOI = {10.1098/rspa.2014.0963},
       URL = {https://doi.org/10.1098/rspa.2014.0963},
       NOTE = {Paper No. 20140963},
}

@article {BF:20,
    AUTHOR = {Bianchi, Luigi Amedeo and Flandoli, Franco},
     TITLE = {Stochastic {N}avier-{S}tokes equations and related models},
   JOURNAL = {Milan J. Math.},
  FJOURNAL = {Milan Journal of Mathematics},
    VOLUME = {88},
      YEAR = {2020},
    NUMBER = {1},
     PAGES = {225--246},
      ISSN = {1424-9286,1424-9294},
   MRCLASS = {35R60 (35Q35 60H50 76D06)},
  MRNUMBER = {4103436},
MRREVIEWER = {Paul\ Andr\'{e}\ Razafimandimby},
       DOI = {10.1007/s00032-020-00312-9},
       URL = {https://doi.org/10.1007/s00032-020-00312-9},
}

@article {FR:08,
    AUTHOR = {Flandoli, Franco and Romito, Marco},
     TITLE = {Markov selections for the 3{D} stochastic {N}avier-{S}tokes
              equations},
   JOURNAL = {Probab. Theory Related Fields},
  FJOURNAL = {Probability Theory and Related Fields},
    VOLUME = {140},
      YEAR = {2008},
    NUMBER = {3-4},
     PAGES = {407--458},
      ISSN = {0178-8051,1432-2064},
   MRCLASS = {76D05 (35Q30 35R60 60H15 60H30 76M35)},
  MRNUMBER = {2365480},
MRREVIEWER = {Hakima\ Bessaih},
       DOI = {10.1007/s00440-007-0069-y},
       URL = {https://doi.org/10.1007/s00440-007-0069-y},
}

@article {DPD:03,
    AUTHOR = {Da Prato, Giuseppe and Debussche, Arnaud},
     TITLE = {Ergodicity for the 3{D} stochastic {N}avier-{S}tokes
              equations},
   JOURNAL = {J. Math. Pures Appl. (9)},
  FJOURNAL = {Journal de Math\'{e}matiques Pures et Appliqu\'{e}es.
              Neuvi\`eme S\'{e}rie},
    VOLUME = {82},
      YEAR = {2003},
    NUMBER = {8},
     PAGES = {877--947},
      ISSN = {0021-7824},
   MRCLASS = {60H15 (35Q30 35R60 76D05 76M35)},
  MRNUMBER = {2005200},
MRREVIEWER = {Marek\ Capi\'{n}ski},
       DOI = {10.1016/S0021-7824(03)00025-4},
       URL = {https://doi.org/10.1016/S0021-7824(03)00025-4},
}

@article {FG:95,
    AUTHOR = {Flandoli, Franco and G{\k{a}}tarek, Dariusz},
     TITLE = {Martingale and stationary solutions for stochastic
              {N}avier-{S}tokes equations},
   JOURNAL = {Probab. Theory Related Fields},
  FJOURNAL = {Probability Theory and Related Fields},
    VOLUME = {102},
      YEAR = {1995},
    NUMBER = {3},
     PAGES = {367--391},
      ISSN = {0178-8051,1432-2064},
   MRCLASS = {60H15 (35Q30 35R60 76D05 76M35)},
  MRNUMBER = {1339739},
MRREVIEWER = {Marek\ Capi\'{n}ski},
       DOI = {10.1007/BF01192467},
       URL = {https://doi.org/10.1007/BF01192467},
}

@article {Ben:95,
    AUTHOR = {Bensoussan, A.},
     TITLE = {Stochastic {N}avier-{S}tokes equations},
   JOURNAL = {Acta Appl. Math.},
  FJOURNAL = {Acta Applicandae Mathematicae},
    VOLUME = {38},
      YEAR = {1995},
    NUMBER = {3},
     PAGES = {267--304},
      ISSN = {0167-8019,1572-9036},
   MRCLASS = {35R60 (35Q30 60H15 76D05 76M35)},
  MRNUMBER = {1326637},
MRREVIEWER = {Ana\ Bela\ Cruzeiro},
       DOI = {10.1007/BF00996149},
       URL = {https://doi.org/10.1007/BF00996149},
}

@article {BT:73,
    AUTHOR = {Bensoussan, A. and Temam, R.},
     TITLE = {\'{E}quations stochastiques du type {N}avier-{S}tokes},
   JOURNAL = {J. Functional Analysis},
  FJOURNAL = {Journal of Functional Analysis},
    VOLUME = {13},
      YEAR = {1973},
     PAGES = {195--222},
      ISSN = {0022-1236},
   MRCLASS = {60H15 (35Q10)},
  MRNUMBER = {348841},
MRREVIEWER = {J.\ A.\ Goldstein},
       DOI = {10.1016/0022-1236(73)90045-1},
       URL = {https://doi.org/10.1016/0022-1236(73)90045-1},
}

@article {KC:22,
    AUTHOR = {Kuan, Jeffrey and \v{C}ani\'{c}, Sun\v{c}ica},
     TITLE = {A stochastically perturbed fluid-structure interaction problem
              modeled by a stochastic viscous wave equation},
   JOURNAL = {J. Differential Equations},
  FJOURNAL = {Journal of Differential Equations},
    VOLUME = {310},
      YEAR = {2022},
     PAGES = {45--98},
      ISSN = {0022-0396},
   MRCLASS = {76M35 (74F10)},
  MRNUMBER = {4352602},
MRREVIEWER = {Benedetta Ferrario},
       DOI = {10.1016/j.jde.2021.11.028},
       URL = {https://doi-org.libproxy.berkeley.edu/10.1016/j.jde.2021.11.028},
}

@article {KC:24,
    AUTHOR = {Kuan, Jeffrey and \v{C}ani\'{c}, Sun\v{c}ica},
     TITLE = {Well-posedness of solutions to stochastic fluid-structure
              interaction},
   JOURNAL = {J. Math. Fluid Mech.},
  FJOURNAL = {Journal of Mathematical Fluid Mechanics},
    VOLUME = {26},
      YEAR = {2024},
    NUMBER = {1},
      ISSN = {1422-6928},
   MRCLASS = {76M35 (74F10 76D03)},
  MRNUMBER = {4668054},
MRREVIEWER = {Chengfeng Sun},
       DOI = {10.1007/s00021-023-00839-y},
       URL = {https://doi-org.libproxy.berkeley.edu/10.1007/s00021-023-00839-y},
      NOTE = {Paper No. 4},
}

@article {Taw:24,
    AUTHOR = {Tawri, Krutika},
     TITLE = {A stochastic fluid-structure interaction problem with the
              {N}avier-slip boundary condition},
   JOURNAL = {SIAM J. Math. Anal.},
  FJOURNAL = {SIAM Journal on Mathematical Analysis},
    VOLUME = {56},
      YEAR = {2024},
    NUMBER = {6},
     PAGES = {7508--7544},
      ISSN = {0036-1410},
   MRCLASS = {60H15 (35A01 35D30 35Q30 35R60 74F10)},
  MRNUMBER = {4823178},
       DOI = {10.1137/24M164029X},
       URL = {https://doi-org.libproxy.berkeley.edu/10.1137/24M164029X},
}

@article {TC:25,
    AUTHOR = {Tawri, Krutika and \v{C}ani\'c, Sun\v{c}ica},
     TITLE = {Existence of martingale solutions to a nonlinearly coupled
              stochastic fluid-structure interaction problem},
   JOURNAL = {Comm. Partial Differential Equations},
  FJOURNAL = {Communications in Partial Differential Equations},
    VOLUME = {50},
      YEAR = {2025},
    NUMBER = {3},
     PAGES = {353--406},
      ISSN = {0360-5302,1532-4133},
   MRCLASS = {60H15 (35A01 35Q30 35R60 76)},
  MRNUMBER = {4870991},
       DOI = {10.1080/03605302.2025.2450375},
       URL = {https://doi.org/10.1080/03605302.2025.2450375},
}

@article {Taw:25,
    AUTHOR = {Tawri, Krutika},
     TITLE = {A 2{D} stochastic nonlinearly coupled fluid-structure
              interaction problem in compliant arteries with unrestricted
              structural displacement},
   JOURNAL = {J. Differential Equations},
  FJOURNAL = {Journal of Differential Equations},
    VOLUME = {431},
      YEAR = {2025},
      ISSN = {0022-0396,1090-2732},
   MRCLASS = {35R60 (35D30 35Q30 35R37 60H15 74F15)},
  MRNUMBER = {4882811},
MRREVIEWER = {Elena\ Alexanrovna\ Strelnikova},
       DOI = {10.1016/j.jde.2025.113243},
       URL = {https://doi.org/10.1016/j.jde.2025.113243},
      NOTE = {Paper No. 113243},
}

@article {BMM:24,
    AUTHOR = {Breit, Dominic and Mensah, Prince Romeo and Moyo, Thamsanqa
              Castern},
     TITLE = {Martingale solutions in stochastic fluid-structure
              interaction},
   JOURNAL = {J. Nonlinear Sci.},
  FJOURNAL = {Journal of Nonlinear Science},
    VOLUME = {34},
      YEAR = {2024},
    NUMBER = {2},
      ISSN = {0938-8974},
   MRCLASS = {76M35 (60H15 74F10 76D05 76D09)},
  MRNUMBER = {4702649},
       DOI = {10.1007/s00332-023-10012-4},
       URL = {https://doi-org.libproxy.berkeley.edu/10.1007/s00332-023-10012-4},
      NOTE = {Paper No. 34},
}

@article {BMR:26,
    AUTHOR = {Brandt, Felix and M\^{i}ndril\u{a}, Claudiu and Roy, Arnab},
     TITLE = {Strong time-periodic solutions for a multilayered
              fluid-structure interaction system with nonlinear coupling},
   JOURNAL = {Nonlinearity},
  FJOURNAL = {Nonlinearity},
    VOLUME = {39},
      YEAR = {2026},
    NUMBER = {3},
      ISSN = {0951-7715,1361-6544},
   MRCLASS = {35Q30 (35B10 35R35 74F10)},
  MRNUMBER = {5048990},
       DOI = {10.1088/1361-6544/ae5203},
       URL = {https://doi.org/10.1088/1361-6544/ae5203},
    NOTE = {Paper No. 035018},
}

@article {DHP:01,
    AUTHOR = {Desch, Wolfgang and Hieber, Matthias and Pr\"{u}ss, Jan},
     TITLE = {{$L^p$}-theory of the {S}tokes equation in a half space},
   JOURNAL = {J. Evol. Equ.},
  FJOURNAL = {Journal of Evolution Equations},
    VOLUME = {1},
      YEAR = {2001},
    NUMBER = {1},
     PAGES = {115--142},
      ISSN = {1424-3199,1424-3202},
   MRCLASS = {35Q35 (35K90 76D03 76D07)},
  MRNUMBER = {1838323},
MRREVIEWER = {Maria\ Specovius-Neugebauer},
       DOI = {10.1007/PL00001362},
       URL = {https://doi.org/10.1007/PL00001362},
}

@article {BR:26,
    AUTHOR = {Brandt, Felix and Roy, Arnab},
     TITLE = {Strong well-posedness for a stochastic fluid-rigid body system
              via stochastic maximal regularity},
   JOURNAL = {J. Lond. Math. Soc. (2)},
  FJOURNAL = {Journal of the London Mathematical Society. Second Series},
    VOLUME = {113},
      YEAR = {2026},
    NUMBER = {5},
      ISSN = {0024-6107,1469-7750},
   MRCLASS = {Prelim},
  MRNUMBER = {5071027},
       DOI = {10.1112/jlms.70561},
       URL = {https://doi.org/10.1112/jlms.70561},
       NOTE = {Paper No. e70561},
}

@article {BBH:26,
    AUTHOR = {Binz, Tim and Brandt, Felix and Hieber, Matthias},
     TITLE = {Interaction of geophysical flows with sea ice dynamics},
   JOURNAL = {NoDEA Nonlinear Differential Equations Appl.},
  FJOURNAL = {NoDEA. Nonlinear Differential Equations and Applications},
    VOLUME = {33},
      YEAR = {2026},
    NUMBER = {2},
      ISSN = {1021-9722,1420-9004},
   MRCLASS = {35Q35 (35K59 35Q86 86A05 86A10)},
  MRNUMBER = {5004673},
       DOI = {10.1007/s00030-025-01169-w},
       URL = {https://doi.org/10.1007/s00030-025-01169-w},
       NOTE = {Paper No. 35},
}

@article {Nau:15,
    AUTHOR = {Nau, Tobias},
     TITLE = {The {$L^p$}-{H}elmholtz projection in finite cylinders},
   JOURNAL = {Czechoslovak Math. J.},
  FJOURNAL = {Czechoslovak Mathematical Journal},
    VOLUME = {65(140)},
      YEAR = {2015},
    NUMBER = {1},
     PAGES = {119--134},
      ISSN = {0011-4642,1572-9141},
   MRCLASS = {35Q30 (35J20 35J25 42B15 46E40)},
  MRNUMBER = {3336028},
       DOI = {10.1007/s10587-015-0163-8},
       URL = {https://doi.org/10.1007/s10587-015-0163-8},
}

@article {PW:17,
    AUTHOR = {Pr\"{u}ss, Jan and Wilke, Mathias},
     TITLE = {Addendum to the paper ``{O}n quasilinear parabolic evolution
              equations in weighted {$L_p$}-spaces {II}'' [{MR}3250797]},
   JOURNAL = {J. Evol. Equ.},
  FJOURNAL = {Journal of Evolution Equations},
    VOLUME = {17},
      YEAR = {2017},
    NUMBER = {4},
     PAGES = {1381--1388},
      ISSN = {1424-3199,1424-3202},
   MRCLASS = {35K90 (35B30 35B65 35K55 35K57 35R35)},
  MRNUMBER = {3722064},
       DOI = {10.1007/s00028-017-0382-6},
       URL = {https://doi.org/10.1007/s00028-017-0382-6},
}

@article {PSW:18,
    AUTHOR = {Pr\"{u}ss, Jan and Simonett, Gieri and Wilke, Mathias},
     TITLE = {Critical spaces for quasilinear parabolic evolution equations
              and applications},
   JOURNAL = {J. Differential Equations},
  FJOURNAL = {Journal of Differential Equations},
    VOLUME = {264},
      YEAR = {2018},
    NUMBER = {3},
     PAGES = {2028--2074},
      ISSN = {0022-0396,1090-2732},
   MRCLASS = {35K59 (35B40 35K58 35K90 35Q30 35Q35 76D05)},
  MRNUMBER = {3721420},
MRREVIEWER = {Stefanie\ Sonner},
       DOI = {10.1016/j.jde.2017.10.010},
       URL = {https://doi.org/10.1016/j.jde.2017.10.010},
}

@article {vNVW:12b,
    AUTHOR = {van Neerven, Jan and Veraar, Mark and Weis, Lutz},
     TITLE = {Maximal {$L^p$}-regularity for stochastic evolution equations},
   JOURNAL = {SIAM J. Math. Anal.},
  FJOURNAL = {SIAM Journal on Mathematical Analysis},
    VOLUME = {44},
      YEAR = {2012},
    NUMBER = {3},
     PAGES = {1372--1414},
      ISSN = {0036-1410,1095-7154},
   MRCLASS = {60H15 (35R60 46B09 47D06)},
  MRNUMBER = {2982717},
MRREVIEWER = {Elisa\ Al\`os},
       DOI = {10.1137/110832525},
       URL = {https://doi.org/10.1137/110832525},
}

@book {HvNVW:17,
    AUTHOR = {Hyt\"{o}nen, Tuomas and van Neerven, Jan and Veraar, Mark and
              Weis, Lutz},
     TITLE = {Analysis in {B}anach spaces. {V}ol. {II}. {P}robabilistic
              methods and operator theory},
    SERIES = {Ergebnisse der Mathematik und ihrer Grenzgebiete. 3. Folge. A
              Series of Modern Surveys in Mathematics [Results in
              Mathematics and Related Areas. 3rd Series. A Series of Modern
              Surveys in Mathematics]},
    VOLUME = {67},
 PUBLISHER = {Springer, Cham},
      YEAR = {2017},
     PAGES = {xxi+616},
      ISBN = {978-3-319-69807-6; 978-3-319-69808-3},
   MRCLASS = {46-02 (42B35 46E30 47-02 60B11 60H30)},
  MRNUMBER = {3752640},
MRREVIEWER = {Adam\ Os\polhk ekowski},
       DOI = {10.1007/978-3-319-69808-3},
       URL = {https://doi.org/10.1007/978-3-319-69808-3},
}

@article {vNVW:12a,
    AUTHOR = {van Neerven, Jan and Veraar, Mark and Weis, Lutz},
     TITLE = {Stochastic maximal {$L^p$}-regularity},
   JOURNAL = {Ann. Probab.},
  FJOURNAL = {The Annals of Probability},
    VOLUME = {40},
      YEAR = {2012},
    NUMBER = {2},
     PAGES = {788--812},
      ISSN = {0091-1798,2168-894X},
   MRCLASS = {60H15 (35B65 42B25 47A60 47D06)},
  MRNUMBER = {2952092},
MRREVIEWER = {Feng-Yu\ Wang},
       DOI = {10.1214/10-AOP626},
       URL = {https://doi.org/10.1214/10-AOP626},
}

@phdthesis{Bra:24, title={Geophysical Flow Models: An Approach by Quasilinear Evolution Equations}, url={https://tuprints.ulb.tu-darmstadt.de/handle/tuda/11844}, DOI={https://doi.org/10.26083/tuprints-00027378}, abstractNote={This thesis develops rigorous analysis of geophysical flow models in the context of Hibler’s viscous-plastic sea ice model by means of quasilinear evolution equations. In a first step, well-posedness results for a fully parabolic variant are shown. Another focal point is the interaction problem of sea ice with a rigid body. Moreover, a coupled atmosphere-sea ice-ocean model is analyzed from a rigorous mathematical point of view. The first part of the thesis is completed by the local strong well-posedness of a parabolic-hyperbolic variant of Hibler’s model. In the second part of the thesis, frameworks to quasilinear time periodic evolution equations are presented. One approach relies on maximal periodic regularity and the Arendt-Bu theorem, whereas the other one is based on the classical Da Prato-Grisvard theorem. Finally, applications of these frameworks to Hibler’s sea ice model, Keller-Segel systems as well as a Nernst-Planck-Poisson type system are provided.}, author={Brandt, Felix Christopher Helmut Ludwig}, school={Technische Universit\"at Darmstadt}, year={2024}}

@article {NS:12,
    AUTHOR = {Nau, Tobias and Saal, J\"{u}rgen},
     TITLE = {{$H^\infty$}-calculus for cylindrical boundary value problems},
   JOURNAL = {Adv. Differential Equations},
  FJOURNAL = {Advances in Differential Equations},
    VOLUME = {17},
      YEAR = {2012},
    NUMBER = {7-8},
     PAGES = {767--800},
      ISSN = {1079-9389},
   MRCLASS = {35J40 (35K35)},
  MRNUMBER = {2963804},
MRREVIEWER = {Yakov\ Yakubov},
}

@book {Nau:12,
    AUTHOR = {Nau, Tobias},
     TITLE = {{$L^p$}-theory of cylindrical boundary value
              problems},
      NOTE = {An operator-valued Fourier multiplier and functional calculus
              approach,
              Dissertation, University of Konstanz, Konstanz, 2012},
 PUBLISHER = {Springer Spektrum, Wiesbaden},
      YEAR = {2012},
     PAGES = {xvi+188},
      ISBN = {978-3-8348-2504-9; 978-3-8348-2505-6},
   MRCLASS = {35-02 (35K90 42B20 47A56)},
  MRNUMBER = {2987207},
       DOI = {10.1007/978-3-8348-2505-6},
       URL = {https://doi.org/10.1007/978-3-8348-2505-6},
}

@article {MRR:20,
    AUTHOR = {Maity, Debayan and Raymond, Jean-Pierre and Roy, Arnab},
     TITLE = {Maximal-in-time existence and uniqueness of strong solution of
              a 3{D} fluid-structure interaction model},
   JOURNAL = {SIAM J. Math. Anal.},
  FJOURNAL = {SIAM Journal on Mathematical Analysis},
    VOLUME = {52},
      YEAR = {2020},
    NUMBER = {6},
     PAGES = {6338--6378},
      ISSN = {0036-1410},
   MRCLASS = {35Q35 (35Q30 74F10 76D05)},
  MRNUMBER = {4189724},
       DOI = {10.1137/18M1178451},
       URL = {https://doi.org/10.1137/18M1178451},
}

@article {BdV:04,
    AUTHOR = {Beir\~{a}o da Veiga, H.},
     TITLE = {On the existence of strong solutions to a coupled
              fluid-structure evolution problem},
   JOURNAL = {J. Math. Fluid Mech.},
  FJOURNAL = {Journal of Mathematical Fluid Mechanics},
    VOLUME = {6},
      YEAR = {2004},
    NUMBER = {1},
     PAGES = {21--52},
      ISSN = {1422-6928},
   MRCLASS = {35Q30 (74F10 76D03 76D05 76F10 76Z05)},
  MRNUMBER = {2027753},
MRREVIEWER = {Shu Ming Sun},
       DOI = {10.1007/s00021-003-0082-5},
       URL = {https://doi.org/10.1007/s00021-003-0082-5},
}

@article {Gra:08,
    AUTHOR = {Grandmont, C\'{e}line},
     TITLE = {Existence of weak solutions for the unsteady interaction of a
              viscous fluid with an elastic plate},
   JOURNAL = {SIAM J. Math. Anal.},
  FJOURNAL = {SIAM Journal on Mathematical Analysis},
    VOLUME = {40},
      YEAR = {2008},
    NUMBER = {2},
     PAGES = {716--737},
      ISSN = {0036-1410},
   MRCLASS = {35Q35 (35D05 74F10 76D03)},
  MRNUMBER = {2438783},
MRREVIEWER = {Dmitry A. Vorotnikov},
       DOI = {10.1137/070699196},
       URL = {https://doi.org/10.1137/070699196},
}

@article {CDEG:05,
    AUTHOR = {Chambolle, Antonin and Desjardins, Beno\^{i}t and Esteban, Maria
              J. and Grandmont, C\'{e}line},
     TITLE = {Existence of weak solutions for the unsteady interaction of a
              viscous fluid with an elastic plate},
   JOURNAL = {J. Math. Fluid Mech.},
  FJOURNAL = {Journal of Mathematical Fluid Mechanics},
    VOLUME = {7},
      YEAR = {2005},
    NUMBER = {3},
     PAGES = {368--404},
      ISSN = {1422-6928},
   MRCLASS = {35Q35 (35D05 35Q30 74F10 76D03)},
  MRNUMBER = {2166981},
MRREVIEWER = {Changxing Miao},
       DOI = {10.1007/s00021-004-0121-y},
       URL = {https://doi.org/10.1007/s00021-004-0121-y},
}

@Article{LR:14,
 Author = {Lengeler, Daniel and  R\r{u}{\v{z}}i{\v{c}}ka, Michael},
 Title = {Weak solutions for an incompressible {Newtonian} fluid interacting with a {Koiter} type shell},
 FJournal = {Archive for Rational Mechanics and Analysis},
 Journal = {Arch. Ration. Mech. Anal.},
 ISSN = {0003-9527},
 Volume = {211},
 Number = {1},
 Pages = {205--255},
 Year = {2014},
 Language = {English},
 DOI = {10.1007/s00205-013-0686-9},
 zbMATH = {6260961},
 Zbl = {1293.35211}
}

@Article{MC13,
 Author = {Muha, Boris and {\v{C}}ani{\'c}, Suncica},
 Title = {Existence of a weak solution to a nonlinear fluid-structure interaction problem modeling the flow of an incompressible, viscous fluid in a cylinder with deformable walls},
 FJournal = {Archive for Rational Mechanics and Analysis},
 Journal = {Arch. Ration. Mech. Anal.},
 ISSN = {0003-9527},
 Volume = {207},
 Number = {3},
 Pages = {919--968},
 Year = {2013},
 Language = {English},
 DOI = {10.1007/s00205-012-0585-5},
 zbMATH = {6148986},
 Zbl = {1260.35148}
}

@book {PS:16,
    AUTHOR = {Pr\"{u}ss, Jan and Simonett, Gieri},
     TITLE = {Moving Interfaces and Quasilinear Parabolic Evolution
              Equations},
    SERIES = {Monographs in Mathematics},
    VOLUME = {105},
 PUBLISHER = {Birkh\"{a}user/Springer, [Cham]},
      YEAR = {2016},
     PAGES = {xix+609},
      ISBN = {978-3-319-27697-7; 978-3-319-27698-4},
   MRCLASS = {35-02 (35B30 35K93 35R35 47F05 58Jxx 76A15 80A22)},
  MRNUMBER = {3524106},
MRREVIEWER = {Glen E. Wheeler},
       DOI = {10.1007/978-3-319-27698-4},
       URL = {https://doi.org/10.1007/978-3-319-27698-4},
}

@book {Ama:95,
    AUTHOR = {Amann, Herbert},
     TITLE = {Linear and Quasilinear Parabolic Problems. {V}ol. {I}. Abstract Linear Theory},
    SERIES = {Monographs in Mathematics},
    VOLUME = {89},
 PUBLISHER = {Birkh\"{a}user Boston, Inc., Boston, MA},
      YEAR = {1995},
     PAGES = {xxxvi+335},
      ISBN = {3-7643-5114-4},
   MRCLASS = {34Gxx (35Kxx 46M35 46N20 47D06 47N20)},
  MRNUMBER = {1345385},
MRREVIEWER = {Paolo Acquistapace},
       DOI = {10.1007/978-3-0348-9221-6},
       URL = {https://doi.org/10.1007/978-3-0348-9221-6},
}

@book {Tri:78,
    AUTHOR = {Triebel, Hans},
     TITLE = {Interpolation Theory, Function Spaces, Differential Operators},
    SERIES = {North-Holland Mathematical Library},
    VOLUME = {18},
 PUBLISHER = {North-Holland Publishing Co., Amsterdam-New York},
      YEAR = {1978},
     PAGES = {528},
      ISBN = {0-7204-0710-9},
   MRCLASS = {46E35 (35Jxx 46M35)},
  MRNUMBER = {503903},
MRREVIEWER = {Robert D. Brown},
}

@article {DHP:03,
    AUTHOR = {Denk, Robert and Hieber, Matthias and Pr\"{u}ss, Jan},
     TITLE = {{$\mathcal{R}$}-boundedness, {F}ourier multipliers and problems of
              elliptic and parabolic type},
   JOURNAL = {Mem. Amer. Math. Soc.},
  FJOURNAL = {Memoirs of the American Mathematical Society},
    VOLUME = {166},
      YEAR = {2003},
    NUMBER = {788},
     PAGES = {viii+114},
      ISSN = {0065-9266},
   MRCLASS = {35-02 (35J30 35K25 42B15)},
  MRNUMBER = {2006641},
MRREVIEWER = {Alain Brillard},
       DOI = {10.1090/memo/0788},
       URL = {https://doi.org/10.1090/memo/0788},
}

@article {GH:16,
    AUTHOR = {Grandmont, C\'{e}line and Hillairet, Matthieu},
     TITLE = {Existence of global strong solutions to a beam-fluid
              interaction system},
   JOURNAL = {Arch. Ration. Mech. Anal.},
  FJOURNAL = {Archive for Rational Mechanics and Analysis},
    VOLUME = {220},
      YEAR = {2016},
    NUMBER = {3},
     PAGES = {1283--1333},
      ISSN = {0003-9527},
   MRCLASS = {74F10 (35A01 35D35 35Q30 35Q74 74G25)},
  MRNUMBER = {3466847},
MRREVIEWER = {Merab Svanadze},
       DOI = {10.1007/s00205-015-0954-y},
       URL = {https://doi.org/10.1007/s00205-015-0954-y},
}

@article {GHL:19,
    AUTHOR = {Grandmont, C\'{e}line and Hillairet, Matthieu and Lequeurre,
              Julien},
     TITLE = {Existence of local strong solutions to fluid-beam and
              fluid-rod interaction systems},
   JOURNAL = {Ann. Inst. H. Poincar\'{e} C Anal. Non Lin\'{e}aire},
  FJOURNAL = {Annales de l'Institut Henri Poincar\'{e} C. Analyse Non Lin\'{e}aire},
    VOLUME = {36},
      YEAR = {2019},
    NUMBER = {4},
     PAGES = {1105--1149},
      ISSN = {0294-1449},
   MRCLASS = {76D05 (35D35 35Q35 74F10)},
  MRNUMBER = {3955112},
       DOI = {10.1016/j.anihpc.2018.10.006},
       URL = {https://doi.org/10.1016/j.anihpc.2018.10.006},
}

@article {Ray:07,
    AUTHOR = {Raymond, J.-P.},
     TITLE = {Stokes and {N}avier-{S}tokes equations with nonhomogeneous
              boundary conditions},
   JOURNAL = {Ann. Inst. H. Poincar\'{e} C Anal. Non Lin\'{e}aire},
  FJOURNAL = {Annales de l'Institut Henri Poincar\'{e} C. Analyse Non Lin\'{e}aire},
    VOLUME = {24},
      YEAR = {2007},
    NUMBER = {6},
     PAGES = {921--951},
      ISSN = {0294-1449},
   MRCLASS = {35Q30 (35Q35 76D05 76D07)},
  MRNUMBER = {2371113},
MRREVIEWER = {Vladimir V. Shelukhin},
       DOI = {10.1016/j.anihpc.2006.06.008},
       URL = {https://doi.org/10.1016/j.anihpc.2006.06.008},
}

@article {AH:23,
    AUTHOR = {Agresti, Antonio and Hussein, Amru},
     TITLE = {Maximal {$L^p$}-regularity and {$H^\infty$}-calculus for block
              operator matrices and applications},
   JOURNAL = {J. Funct. Anal.},
  FJOURNAL = {Journal of Functional Analysis},
    VOLUME = {285},
      YEAR = {2023},
    NUMBER = {11},
      ISSN = {0022-1236},
   MRCLASS = {47A55 (35M13 35Q35 47A60 47D06)},
  MRNUMBER = {4642564},
MRREVIEWER = {Julio Delgado},
       DOI = {10.1016/j.jfa.2023.110146},
       URL = {https://doi.org/10.1016/j.jfa.2023.110146},
       NOTE = {Paper No. 110146},
}

@article {MT:21,
    AUTHOR = {Maity, Debayan and Takahashi, Tak\'{e}o},
     TITLE = {{$L^p$} theory for the interaction between the incompressible
              {N}avier-{S}tokes system and a damped plate},
   JOURNAL = {J. Math. Fluid Mech.},
  FJOURNAL = {Journal of Mathematical Fluid Mechanics},
    VOLUME = {23},
      YEAR = {2021},
    NUMBER = {4},
      ISSN = {1422-6928},
   MRCLASS = {35Q30 (74F10 76D03 76D05)},
  MRNUMBER = {4323336},
       DOI = {10.1007/s00021-021-00628-5},
       URL = {https://doi.org/10.1007/s00021-021-00628-5},
      NOTE = {Paper No. 103},
}

@incollection {SS:92,
    AUTHOR = {Simader, Christian G. and Sohr, Hermann},
     TITLE = {A new approach to the {H}elmholtz decomposition and the
              {N}eumann problem in {$L^q$}-spaces for bounded and exterior
              domains},
 BOOKTITLE = {Mathematical problems relating to the {N}avier-{S}tokes
              equation},
    SERIES = {Ser. Adv. Math. Appl. Sci.},
    VOLUME = {11},
     PAGES = {1--35},
 PUBLISHER = {World Sci. Publ., River Edge, NJ},
      YEAR = {1992},
   MRCLASS = {35J05 (35Q30)},
  MRNUMBER = {1190728},
MRREVIEWER = {Volker Vogelsang},
       DOI = {10.1142/9789814503594\_0001},
       URL = {https://doi.org/10.1142/9789814503594_0001},
}

@book {Tem:79,
    AUTHOR = {Temam, Roger},
     TITLE = {Navier-{S}tokes Equations},
    SERIES = {Studies in Mathematics and its Applications},
    VOLUME = {2},
   EDITION = {Revised},
      NOTE = {Theory and numerical analysis,
              With an appendix by F. Thomasset},
 PUBLISHER = {North-Holland Publishing Co., Amsterdam-New York},
      YEAR = {1979},
     PAGES = {x+519},
      ISBN = {0-444-85307-3},
   MRCLASS = {35Q10 (49D99 65P05 76D05)},
  MRNUMBER = {603444},
}

@article {LM:62,
    AUTHOR = {Lions, J.-L. and Magenes, E.},
     TITLE = {Problemi ai limiti non omogenei. {V}},
   JOURNAL = {Ann. Scuola Norm. Sup. Pisa Cl. Sci. (3)},
  FJOURNAL = {Annali della Scuola Normale Superiore di Pisa. Classe di
              Scienze. Serie III},
    VOLUME = {16},
      YEAR = {1962},
     PAGES = {1--44},
      ISSN = {0391-173X},
   MRCLASS = {35.45},
  MRNUMBER = {146527},
MRREVIEWER = {C. B. Morrey, Jr.},
}

@article {Wei:01,
    AUTHOR = {Weis, Lutz},
     TITLE = {Operator-valued {F}ourier multiplier theorems and maximal
              {$L_p$}-regularity},
   JOURNAL = {Math. Ann.},
  FJOURNAL = {Mathematische Annalen},
    VOLUME = {319},
      YEAR = {2001},
    NUMBER = {4},
     PAGES = {735--758},
      ISSN = {0025-5831},
   MRCLASS = {42B15 (47D06)},
  MRNUMBER = {1825406},
MRREVIEWER = {Alfonso Montes-Rodr\'{\i}guez},
       DOI = {10.1007/PL00004457},
       URL = {https://doi.org/10.1007/PL00004457},
}

@incollection {McIY:90,
    AUTHOR = {McIntosh, Alan and Yagi, Atsushi},
     TITLE = {Operators of type {$\omega$} without a bounded {$H_\infty$}
              functional calculus},
 BOOKTITLE = {Miniconference on {O}perators in {A}nalysis ({S}ydney, 1989)},
    SERIES = {Proc. Centre Math. Anal. Austral. Nat. Univ.},
    VOLUME = {24},
     PAGES = {159--172},
 PUBLISHER = {Austral. Nat. Univ., Canberra},
      YEAR = {1990},
   MRCLASS = {47A60 (47A10)},
  MRNUMBER = {1060121},
MRREVIEWER = {J\"{o}rg Eschmeier},
}

@incollection {KW:04,
    AUTHOR = {Kunstmann, Peer C. and Weis, Lutz},
     TITLE = {Maximal {$L_p$}-regularity for parabolic equations, {F}ourier
              multiplier theorems and {$H^\infty$}-functional calculus},
 BOOKTITLE = {Functional Analytic Methods for Evolution Equations},
    SERIES = {Lecture Notes in Math.},
    VOLUME = {1855},
     PAGES = {65--311},
 PUBLISHER = {Springer, Berlin},
      YEAR = {2004},
   MRCLASS = {47D06 (34G10 35D10 35J55 35K20 35K90 42B20 47A60)},
  MRNUMBER = {2108959},
MRREVIEWER = {Xuan Thinh Duong},
       DOI = {10.1007/978-3-540-44653-8\_2},
       URL = {https://doi-org.libproxy.berkeley.edu/10.1007/978-3-540-44653-8_2},
}

@book {Kat:95,
    AUTHOR = {Kato, Tosio},
     TITLE = {Perturbation Theory for Linear Operators},
    SERIES = {Classics in Mathematics},
      NOTE = {Reprint of the 1980 edition},
 PUBLISHER = {Springer-Verlag, Berlin},
      YEAR = {1995},
     PAGES = {xxii+619},
      ISBN = {3-540-58661-X},
   MRCLASS = {47A55 (46-00 47-00)},
  MRNUMBER = {1335452},
}
\bibliographystyle{siam}

\end{document}